\documentclass{article}
\usepackage[obeyDraft, color=lightgray]{todonotes} 
\usepackage[left=1in,right=1in,top=1in,bottom=1in]{geometry}
\usepackage{amsfonts, amssymb, amsmath, amsthm}
\usepackage{authblk} 
\usepackage[section]{placeins} 
\usepackage{mathtools}
\usepackage{bbm}
\usepackage{cases} 
\usepackage{fancyhdr}
\usepackage{hyperref}
\usepackage{tikz}
\usetikzlibrary{arrows.meta}
\usetikzlibrary{decorations.pathreplacing}
\usepackage{pdfpages}
\usepackage{setspace}
\usepackage{pgffor}
\usepackage{tcolorbox}
\tcbuselibrary{raster, skins}
\usepackage{caption}
\usepackage{subcaption}
\usepackage[backend=biber]{biblatex}
\newtheorem{theorem}{Theorem}
\numberwithin{theorem}{subsection}

\newtheorem{conjecture}[theorem]{Conjecture}

\newtheorem{definition}[theorem]{Definition}
\newtheorem{example}[theorem]{Example}
\newtheorem{lemma}[theorem]{Lemma}

\newtheorem{proposition}{Proposition}
\numberwithin{proposition}{subsubsection}
\newtheorem{exampleGrid}[proposition]{Illustration}

\newcommand{\RG}{\text{RG}} 
\newcommand{\SG}{\text{SG}} 
\newcommand{\CY}{\text{C}}  
\newcommand{\RT}{\text{RT}} 
\newcommand{\ST}{\text{ST}} 
\newcommand{\R}{\mathbb{R}}
\newcommand{\Z}{\mathbb{Z}}
\newcommand{\ps}[3]{\left((#1,#2),#3\right)}
\newcommand{\xyd}{{\ps xyd}}
\newcommand{\abg}{{\ps abg}}
\newcommand{\grid}{\Z/n\Z \times \Z/m\Z}
\newcommand{\cylinderSymmetries}{\Z/n\Z \rtimes R}
\newcommand{\torusSymmetries}{(\grid) \rtimes R}
\newcommand{\NA}{\multicolumn{1}{c|}{---}}
\newcommand{\orb}[2]{\mathcal{O}^{\langle #1 \rangle}_{\langle #2 \rangle}}
\newcommand{\orbid}[1]{\mathcal{O}^{\langle #1 \rangle}_{\mathbbm 1}}

\newcommand{\appendixpagenumbering}{
  \break
  \pagenumbering{arabic}
  \renewcommand{\thepage}{\thesection-\arabic{page}}
}

\newcommand{\seqTableEntry}[2]{
  \begin{tabular}[c]{@{}l@{}}
    Sequence \ref{seq:#1}\\
    OEIS: #2
  \end{tabular}
}

\newcommand{\tableTableEntry}[2]{
  \begin{tabular}[c]{@{}l@{}}
    Table \ref{tabl:#1}\\
    OEIS: #2
  \end{tabular}
}

\DeclareMathOperator{\fxpt}{fxpt}
\DeclareMathOperator{\lcm}{lcm}
\DeclareMathOperator{\id}{id}

\newcommand*{\nXnGridIllustration}[3]{
\begin{exampleGrid}
\noindent \noindent[This is shown in Sequence \ref{seq:nXnGrid_#1}.]

\noindent%
\begin{minipage}{\linewidth}\captionsetup{type=figure}%
  \includegraphics[width=\textwidth]{Assets/Grid/grid_comb_#1_#2.pdf}
  \captionof{figure}{#3}
  \label{fig:nXnGrid_#1}%
\end{minipage}%
\end{exampleGrid}
}
\newcommand*{\nXmGridIllustration}[3]{
\begin{exampleGrid}
\noindent \noindent[This is shown in Table \ref{tabl:nXmGrid_#1}.]

\noindent%
\begin{minipage}{\linewidth}\captionsetup{type=figure}%
  \includegraphics[width=\textwidth]{Assets/Grid/grid_comb_#1_#2.pdf}
  \captionof{figure}{#3}
  \label{fig:nXmGrid_#1}%
\end{minipage}%
\end{exampleGrid}
}
\newcommand*{\nXmCylinderIllustration}[3]{
\begin{exampleGrid}
\noindent \noindent[This is shown in Table \ref{tabl:nXmCyl_#1}.]

\noindent%
\begin{minipage}{\linewidth}\captionsetup{type=figure}%
  \includegraphics[width=\textwidth]{Assets/Cylinder/cyl_comb_#1_#2.pdf}
  \captionof{figure}{#3}
  \label{fig:nXmCyl_#1}%
\end{minipage}%
\end{exampleGrid}
}
\newcommand*{\nXnTorusIllustration}[3]{
\begin{exampleGrid}
\noindent \noindent[This is shown in Sequence \ref{seq:nXnTorus_#1}.]

\noindent%
\begin{minipage}{\linewidth}\captionsetup{type=figure}%
  \includegraphics[width=\textwidth]{Assets/Torus/torus_comb_#1_#2.pdf}
  \captionof{figure}{#3}
  \label{fig:nXnTorus_#1}%
\end{minipage}%
\end{exampleGrid}
}
\newcommand*{\nXmTorusIllustration}[3]{
\begin{exampleGrid}
\noindent \noindent[This is shown in Table \ref{tabl:nXmTorus_#1}.]

\noindent%
\begin{minipage}{\linewidth}\captionsetup{type=figure}%
  \includegraphics[width=\textwidth]{Assets/Torus/torus_comb_#1_#2.pdf}
  \captionof{figure}{#3}
  \label{fig:nXmTorus_#1}%
\end{minipage}%
\end{exampleGrid}
}

\newcommand*{\tileColor}[1][black]{%
  \begin{tikzpicture}[scale=0.7,baseline=6]%
    \fill[#1] (0,0) rectangle (1,1); %
    \draw[dotted] (0,0) rectangle (1,1);%
  \end{tikzpicture}%
}%
\newcommand{\tileDiag}[1][black]{%
  \begin{tikzpicture}[scale=0.7,baseline=6]%
    \draw[dotted] (0,0) rectangle (1,1);%
    \fill[#1] (0,0.5) -- (0.5,1) -- (0,1) -- cycle; 
    \fill[#1] (0.5,0) -- (1,0) -- (1,0.5) -- cycle; 
  \end{tikzpicture}%
}%
\newcommand{\tileAdiag}[1][black]{%
  \begin{tikzpicture}[scale=0.7,baseline=6]%
    \draw[dotted] (0,0) rectangle (1,1);%
    \fill[#1] (0,0) -- (0,0.5) -- (0.5,0) -- cycle; 
    \fill[#1] (0.5,1) -- (1,1) -- (1,0.5) -- cycle; 
  \end{tikzpicture}%
}%
\newcommand{\tileU}[1][black]{
  \begin{tikzpicture}[scale=0.7,baseline=6]%
    \draw[dotted] (0,0) rectangle (1,1);%
    \fill[#1] (0,0)--(1,0)--(0.5,1); %
  \end{tikzpicture}%
}%
\newcommand{\tileD}[1][black]{
  \begin{tikzpicture}[scale=0.7,baseline=6]%
    \draw[dotted] (0,0) rectangle (1,1);%
    \fill[#1] (0,1)--(0.5,0)--(1,1); %
  \end{tikzpicture}%
}%
\newcommand{\tileL}[1][black]{
  \begin{tikzpicture}[scale=0.7,baseline=6]%
    \draw[dotted] (0,0) rectangle (1,1);%
    \fill[#1] (0,0)--(0,1)--(1,0.5); %
  \end{tikzpicture}%
}%
\newcommand{\tileR}[1][black]{
  \begin{tikzpicture}[scale=0.7,baseline=6]%
    \draw[dotted] (0,0) rectangle (1,1);%
    \fill[#1] (1,0)--(1,1)--(0,0.5); %
  \end{tikzpicture}%
}%
\newcommand{\tileNE}[1][black]{%
  \begin{tikzpicture}[scale=0.7,baseline=6]%
    \draw[dotted] (0,0) rectangle (1,1);%
    \fill[#1] (0,1) -- (1,1) -- (1,0) -- cycle; 
  \end{tikzpicture}%
}%
\newcommand{\tileNW}[1][black]{%
  \begin{tikzpicture}[scale=0.7,baseline=6]%
    \draw[dotted] (0,0) rectangle (1,1);%
    \fill[#1] (0,0) -- (0,1) -- (1,1) -- cycle; 
  \end{tikzpicture}%
}%
\newcommand{\inlineTileNW}[1][black]{%
  \begin{tikzpicture}[scale=0.35,baseline=2.5]%
    \draw (0,0) rectangle (1,1);%
    \fill[#1] (0,0) -- (1,1) -- (0,1) -- cycle; 
  \end{tikzpicture}%
}%
\newcommand{\tileSE}[1][black]{%
  \begin{tikzpicture}[scale=0.7,baseline=6]%
    \draw[dotted] (0,0) rectangle (1,1);%
    \fill[#1] (0,0) -- (1,0) -- (1,1) -- cycle; 
  \end{tikzpicture}%
}%
\newcommand{\tileSW}[1][black]{%
  \begin{tikzpicture}[scale=0.7,baseline=6]%
    \draw[dotted] (0,0) rectangle (1,1);%
    \fill[#1] (0,0) -- (0,1) -- (1,0) -- cycle; 
  \end{tikzpicture}%
}%
\newcommand{\tileVert}[1][black]{%
  \begin{tikzpicture}[scale=0.7,baseline=6]%
    \draw[dotted] (0,0) rectangle (1,1);%
    \fill[#1] (0,0) rectangle (1/3,1); 
    \fill[#1] (2/3,0) rectangle (1,1); 
  \end{tikzpicture}%
}%
\newcommand{\tileHor}[1][black]{%
  \begin{tikzpicture}[scale=0.7,baseline=6]%
    \draw[dotted] (0,0) rectangle (1,1);%
    \fill[#1] (0,0) rectangle (1,1/3);%
    \fill[#1] (0,2/3) rectangle (1,1);%
  \end{tikzpicture}%
}
\newcommand{\tileRot}[2][black]{%
  \begin{tikzpicture}[scale=0.7,baseline=6]%
    \draw[dotted] (0,0) rectangle (1,1);%
    \fill[#1]%
      ({(0  )*(1 - 2*#2) + #2}, 0)   -- ({0*(1 - 2*#2) + #2}, 1/4) -- ({3/4*(1 - 2*#2) + #2}, 0) 
      ({(0  )*(1 - 2*#2) + #2}, 1/4) -- ({0*(1 - 2*#2) + #2}, 1)   -- ({1/4*(1 - 2*#2) + #2}, 1) 
      ({(1/4)*(1 - 2*#2) + #2}, 1)   -- ({1*(1 - 2*#2) + #2}, 1)   -- ({1  *(1 - 2*#2) + #2}, 3/4) 
      ({(1  )*(1 - 2*#2) + #2}, 3/4) -- ({1*(1 - 2*#2) + #2}, 0)   -- ({3/4*(1 - 2*#2) + #2}, 0) 
    ;
  \end{tikzpicture}%
}%
\newcommand{\tileRR}[2][black]{%
  \begin{tikzpicture}[scale=0.7,baseline=6]%
    \draw[dotted] (0,0) rectangle (1,1);%
    \fill[#1]%
    \ifnum#2=0%
      (0, 0)   -- (0, 1/2) -- (1, 0) 
      (1, 1)   -- (1, 1/2) -- (0, 1)
    \fi%
    \ifnum#2=1%
      (0, 0) -- (0, 1) -- (1/2, 1) 
      (1, 1) -- (1, 0) -- (1/2, 0) 
    \fi%
    \ifnum#2=2%
      (0, 1) -- (0, 0) -- (1/2, 0) 
      (1, 0) -- (1, 1) -- (1/2, 1) 
    \fi%
    \ifnum#2=3%
      (0, 1)   -- (0, 1/2) -- (1, 1)
      (1, 0)   -- (1, 1/2) -- (0, 0)
    \fi
    ;
  \end{tikzpicture}%
}%
\newcommand{\tileAsym}[2][black]{%
  \begin{tikzpicture}[scale=0.7,baseline=6]%
    \draw[dotted] (0,0) rectangle (1,1);%
    \fill[#1]%
    \ifnum#2=0 (0, 0) -- (0, 1/2) -- (1, 0)    \fi%
    \ifnum#2=1 (0, 1) -- (0, 0)   -- (1/2, 0) \fi%
    \ifnum#2=2 (1, 1) -- (1, 1/2) -- (0, 1)    \fi%
    \ifnum#2=3 (1, 0) -- (1, 1)   -- (1/2, 1) \fi%
    \ifnum#2=4 (0, 0) -- (0, 1)   -- (1/2, 1)  \fi%
    \ifnum#2=5 (0, 1) -- (0, 1/2) -- (1, 1)   \fi%
    \ifnum#2=6 (1, 1) -- (1, 0)   -- (1/2, 0)  \fi%
    \ifnum#2=7 (1, 0) -- (1, 1/2) -- (0, 0)   \fi%
    ;
  \end{tikzpicture}%
}%

\title{Counting tilings of the \texorpdfstring{$n \times m$}{n by m} grid, cylinder, and torus}
\author[1]{Peter Kagey}
\author[2]{William Keehn}
\affil[1]{Department of Mathematics, Harvey Mudd College}
\affil[2]{Prison Mathematics Project}

\begin{document}
\maketitle
\tableofcontents

\begin{abstract}
  We count tilings of the $n \times m$ rectangular grid, cylinder, and
  torus with arbitrary tile sets up to arbitrary symmetries of the square and
  rectangle, along with cyclic shifting of rows and columns.
  This provides a unifying framework for understanding a family of
  counting problems, expanding on the work by Ethier and Lee counting
  tilings of the torus by tiles of two colors.
\end{abstract}
\todo[inline]{Errata}
\todo[inline]{The only part I don't really follow well is when you demonstrate that the "outside semi-direct product" is appropriate.}
\todo[inline]{sec 1.4, I'm a little unclear on this part. What does Aut(Z/nZ) mean?}
Given some set of arbitrary patterns, we are interested in counting ways of
tiling the $n \times m$ square grid, of tiling the infinite strip in a
periodic way, and of tiling the Euclidean plane in a way that is periodic both
left-to-right and top-to-bottom, up to various symmetries.
These counting problems are of actual physical interest: if we are designing a
square table that will be tiled with four tiles that are rotations
of the tile pattern $\inlineTileNW$, then there are $43$ ways of doing so up to
rotation and reflection of the table, as illustrated in
Figure \ref{fig:squareTilingTruchet}.

However we may be interested in counting the number of tile designs up to other
symmetries: if the table is fixed, then there are $4^4 = 256$ tilings;
if two tilings are considered the same if one is a rotation of another,
then there are $70$ such tilings.
In the setting of the Euclidean plane we might imagine that we plan to tile
the floor of a building with a repeating $2\times2$ pattern of these tiles,
where we consider two tilings to be the same if one is a rotation, reflection,
or translation of another; in this case, there are $17$ such tilings.

\begin{figure}[ht]
  \centering
  \includegraphics[width=\textwidth]{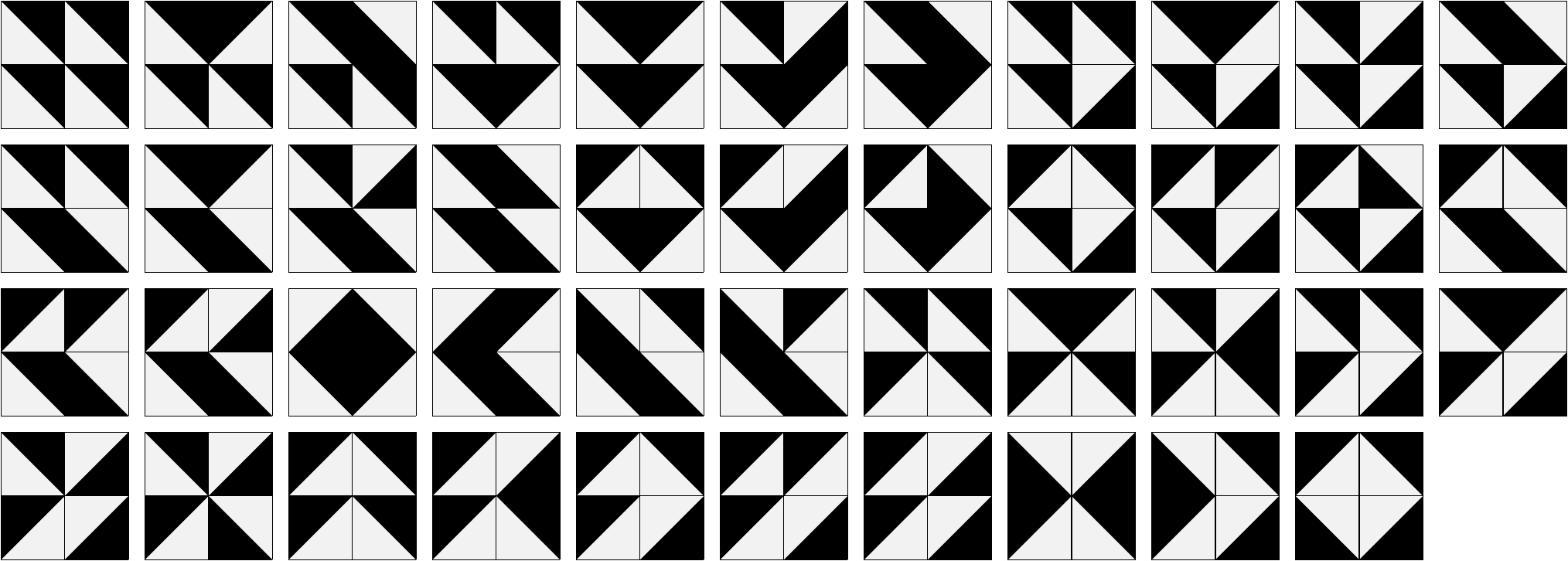}
  \caption{The $43$ ways of tiling the $2 \times 2$ grid by tiles of the form
  $\left\{\tileSW,\tileNW,\tileNE,\tileSE \right\}$}
  \label{fig:squareTilingTruchet}
\end{figure}

Ultimately, we will construct a framework for counting the number of ways of
tiling the grid up to various symmetries. This provides a unifying theory for a
family of problems that appear to have only been analyzed in an
\textit{ad hoc} manner.\footnote{This will give a unifying framework for over a dozen OEIS
sequences including but not limited to
A054247, 
A225910, 
A086675, 
A179043, 
A184271, 
A184284, 
A184277, 
A222187, 
A222188, 
A255015, 
A255016, 
A295223, 
A302484, 
A295229, 
A047937, 
A343095, 
A343096, and 
A200564. 
} This has resulted in the addition of 49 new sequences to the On-Line Encyclopedia of Integer Sequences (OEIS) \cite{OEIS}.

When we look at the grid up to symmetries of the square, we call this the
$n \times m$ grid, an example of which is illustrated in Figure \ref{fig:squareTilingTruchet}.
When we additionally allow cyclic shifting of columns, we call this the
$n \times m$ cylinder, which is illustrated in Figure \ref{fig:cylinderIdentification}.
When we allow cyclic shifting of the rows in addition to the above symmetries,
we call this the $n \times m$ torus, which is illustrated in Figure \ref{fig:torusIdentification}.

\begin{figure}
  \begin{subfigure}[b]{0.55\textwidth}
  \centering
  \begin{tikzpicture}
    \foreach \k in {0,1,2} {
    \foreach \x/\y in {0/1, 0/2, 1/0, 1/3, 2/0, 2/1, 2/3} {
      \fill[black!50!white] (\x + 3*\k, \y) rectangle (\x + 1 + 3*\k, \y + 1);
    }}
    \draw (0,0) grid (9,4);
    \draw[line width=3, white] (0,0) rectangle (3,4);
    \draw[ultra thick, dashed, red] (0,0) rectangle (3,4);
    \draw[line width=3, white] (5,0) rectangle (8,4);
    \draw[ultra thick, dotted, blue] (5,0) rectangle (8,4);
  \end{tikzpicture}
  \caption{}
\end{subfigure}
\hfill
\begin{subfigure}[b]{0.2\textwidth}
  \centering
  \begin{tikzpicture}
    \foreach \x/\y in {0/1, 0/2, 1/0, 1/3, 2/0, 2/1, 2/3} {
      \fill[black!50!white] (\x, \y) rectangle (\x + 1, \y + 1);
    }
    \draw[] (0,0) grid (3,4);
    \draw[ultra thick, white]       (0,0) rectangle (3,4);
    \draw[ultra thick, red, dashed] (0,0) rectangle (3,4);
    \draw[red, -{Triangle[length=3mm, width=3mm]}] (0,2.02)--(0,2.025);
    \draw[red, -{Triangle[length=3mm, width=3mm]}] (3,2.02)--(3,2.025);
  \end{tikzpicture}
  \caption{}
\end{subfigure}
\hfill
\begin{subfigure}[b]{0.2\textwidth}
  \centering
  \begin{tikzpicture}
    \foreach \x/\y in {1/1, 1/2, 2/0, 2/3, 0/0, 2/2, 0/3} {
      \fill[black!50!white] (\x, \y) rectangle (\x + 1, \y + 1);
    }
    \draw[] (0,0) grid (3,4);
    \draw[ultra thick, white]        (0,0) rectangle (3,4);
    \draw[ultra thick, blue, dotted] (0,0) rectangle (3,4);
    \draw[blue, -{Triangle[length=3mm, width=3mm]}] (0,2.02)--(0,2.025);
    \draw[blue, -{Triangle[length=3mm, width=3mm]}] (3,2.02)--(3,2.025);
  \end{tikzpicture}
  \caption{}
\end{subfigure}
\caption{
  Part (a) shows a $3 \times 4$ cylinder repeated three times horizontally
  with two $3 \times 4$ regions selected.
  Part (b) is one of the grid representations of this cylinder.
  Part (c) is an equivalent grid representation if $180^\circ$ rotation is
  allowed.
}
\label{fig:cylinderIdentification}
\end{figure}
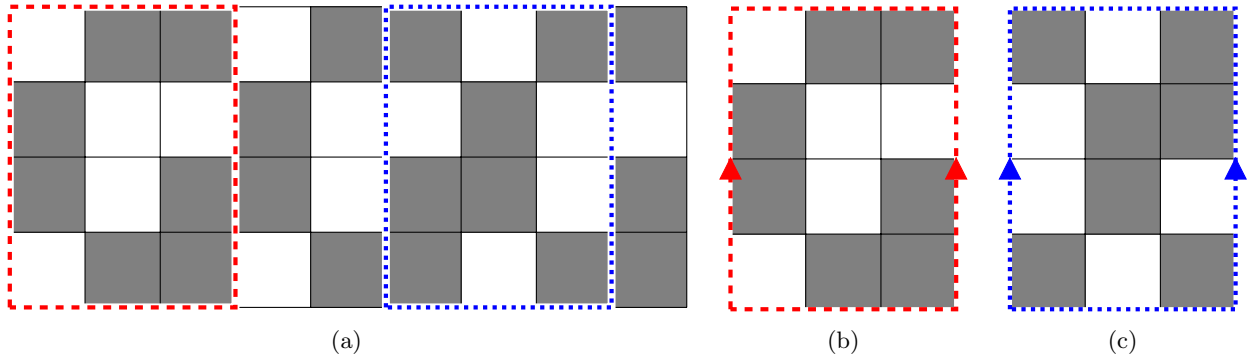
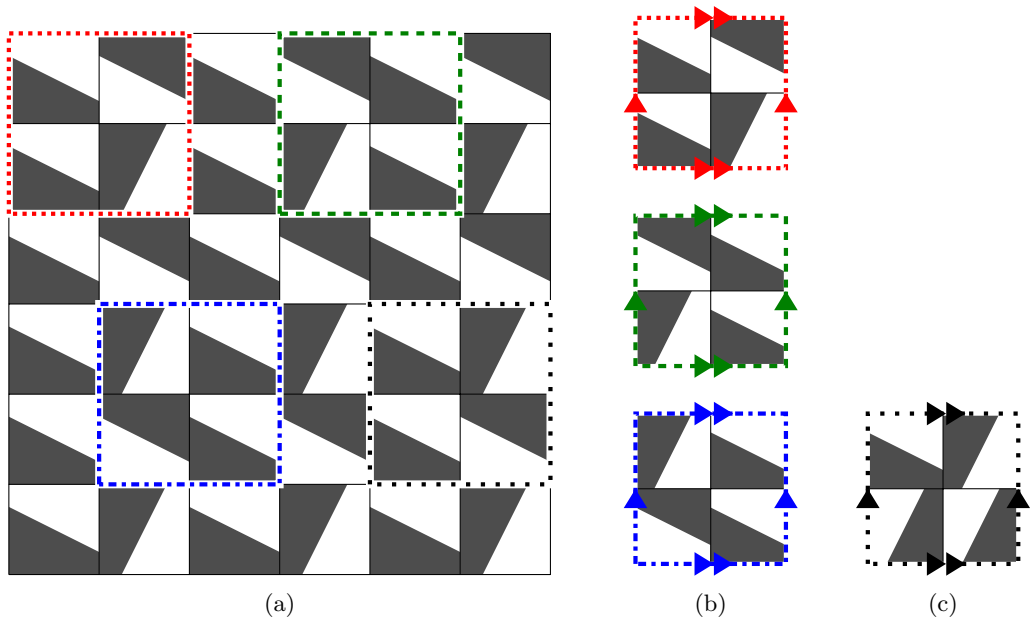
\begin{figure}
\centering
  \begin{subfigure}[b]{0.5\textwidth}
    \centering
    \begin{tikzpicture}[scale=1.2]
      \foreach \x in {0, 2, 4} { \foreach \y in {0, 2, 4} {
        \fill[black!70!white, shift={(\x,\y)}] (1,0)--(0,0)--(0,0.75)--(1,0.25);
        \fill[black!70!white, shift={(\x,\y)}] (1,1)--(1,0)--(1.25,0)--(1.75,1);
        \fill[black!70!white, shift={(\x,\y)}] (1,1)--(0,1)--(0,1.75)--(1,1.25);
        \fill[black!70!white, shift={(\x,\y)}] (1,2)--(2,2)--(2,1.25)--(1,1.75);
      }}
      \draw (0,0) grid (6,6);
      \draw[line width=3, white]      (0,4) rectangle (2,6);
      \draw[ultra thick, red, dotted] (0,4) rectangle (2,6);
      \draw[line width=3, white]                 (3,4) rectangle (5,6);
      \draw[ultra thick, green!50!black, dashed] (3,4) rectangle (5,6);
      \draw[line width=3, white]         (1,1) rectangle (3,3);
      \draw[ultra thick, blue, dash dot] (1,1) rectangle (3,3);
      \draw[line width=3, white] (4,1) rectangle (6,3);
      \draw[ultra thick, black, loosely dotted] (4,1) rectangle (6,3);
    \end{tikzpicture}
    \caption{}
  \end{subfigure}
  \begin{subfigure}[b]{0.18\textwidth}
    \centering
    \begin{tikzpicture}
      \fill[black!70!white] (1,0)--(0,0)--(0,0.75)--(1,0.25);
      \fill[black!70!white] (1,1)--(0,1)--(0,1.75)--(1,1.25);
      \fill[black!70!white] (1,2)--(2,2)--(2,1.25)--(1,1.75);
      \fill[black!70!white] (1,0)--(1,1)--(1.75,1)--(1.25,0);
      \draw (0,0) grid (2,2);
      \draw[ultra thick, white] (0,0) rectangle (2,2);
      \draw[ultra thick, red, dotted] (0,0) rectangle (2,2);
      \draw[ultra thick, red, -{Triangle}] (0,0.99)--(0,1);
      \draw[ultra thick, red, -{Triangle}] (2,0.99)--(2,1);
      \draw[ultra thick, red, -{Triangle}{Triangle}] (1.29,0)--(1.3,0);
      \draw[ultra thick, red, -{Triangle}{Triangle}] (1.29,2)--(1.3,2);
    \end{tikzpicture}\vspace{0.3cm}
    \begin{tikzpicture}
      \fill[black!70!white] (2,0)--(1,0)--(1,0.75)--(2,0.25);
      \fill[black!70!white] (2,1)--(1,1)--(1,1.75)--(2,1.25);
      \fill[black!70!white] (0,2)--(1,2)--(1,1.25)--(0,1.75);
      \fill[black!70!white] (0,0)--(0,1)--(0.75,1)--(0.25,0);
      \draw (0,0) grid (2,2);
      \draw[ultra thick, white] (0,0) rectangle (2,2);
      \draw[ultra thick, green!50!black, dashed] (0,0) rectangle (2,2);
      \draw[ultra thick, green!50!black, -{Triangle}] (0,0.99)--(0,1);
      \draw[ultra thick, green!50!black, -{Triangle}] (2,0.99)--(2,1);
      \draw[ultra thick, green!50!black, -{Triangle}{Triangle}] (1.29,0)--(1.3,0);
      \draw[ultra thick, green!50!black, -{Triangle}{Triangle}] (1.29,2)--(1.3,2);
    \end{tikzpicture}\vspace{0.3cm}
    \begin{tikzpicture}
      \fill[black!70!white] (2,0)--(1,0)--(1,0.75)--(2,0.25);
      \fill[black!70!white] (2,1)--(1,1)--(1,1.75)--(2,1.25);
      \fill[black!70!white] (0,1)--(0,2)--(0.75,2)--(0.25,1);
      \fill[black!70!white] (0,1)--(1,1)--(1,0.25)--(0,0.75);
      \draw (0,0) grid (2,2);
      \draw[ultra thick, white] (0,0) rectangle (2,2);
      \draw[ultra thick, blue, dash dot] (0,0) rectangle (2,2);
      \draw[ultra thick, blue, -{Triangle}] (0,0.99)--(0,1);
      \draw[ultra thick, blue, -{Triangle}] (2,0.99)--(2,1);
      \draw[ultra thick, blue, -{Triangle}{Triangle}] (1.29,0)--(1.3,0);
      \draw[ultra thick, blue, -{Triangle}{Triangle}] (1.29,2)--(1.3,2);
    \end{tikzpicture}
    \caption{}
  \end{subfigure}
  \begin{subfigure}[b]{0.18\textwidth}
    \centering
    \begin{tikzpicture}
      \fill[black!70!white] (1,1)--(0,1)--(0,1.75)--(1,1.25);
      \fill[black!70!white] (1,1)--(1,2)--(1.75,2)--(1.25,1);
      \fill[black!70!white] (0.25,0)--(1,0)--(1,1)--(0.75,1);
      \fill[black!70!white] (1.25,0)--(2,0)--(2,1)--(1.75,1);
      \draw (0,0) grid (2,2);
      \draw[ultra thick, white] (0,0) rectangle (2,2);
      \draw[ultra thick, black, loosely dotted] (0,0) rectangle (2,2);
      \draw[ultra thick, black, -{Triangle}] (0,0.99)--(0,1);
      \draw[ultra thick, black, -{Triangle}] (2,0.99)--(2,1);
      \draw[ultra thick, black, -{Triangle}{Triangle}] (1.29,0)--(1.3,0);
      \draw[ultra thick, black, -{Triangle}{Triangle}] (1.29,2)--(1.3,2);
    \end{tikzpicture}
    \caption{}
  \end{subfigure}
  \caption{Part (a) shows a $2 \times 2$ torus repeated three
  times horizontally and three times vertically, with three $2 \times 2$ regions
  selected. Part (b) shows three tilings of the $2 \times 2$ grid that are
  equivalent under the toroidal action $\Z/2\Z \times \Z/2\Z$.
  Part (c) shows a $2 \times 2$ torus that is equivalent to the other tori under
  the dihedral action $r^3$.}
\label{fig:torusIdentification}
\end{figure}

The number of such tilings depends on
the size of the grid,
the symmetries of the grid under consideration,
the symmetries of the tile designs, and
the number of tile designs with a given symmetry.
We formalize each of these four notions below, and use them to give a formula
that counts the number of corresponding tile designs.

\section{Notation and preliminaries}
In this section, we will
formalize the notation of a grid and its size,
the symmetries of the grid that we count up to,
the symmetries of the tile designs, and
the number of tile designs with a given symmetry.

\subsection{Tilings and the grid}
In order to talk about tilings of the $n \times m$ grid, cylinder, and torus,
it is important to first formalize what these are. All of these ideas start
with the fundamental idea of the $n \times m$ grid, which we define as
follows.
\begin{definition}
  The $n \times m$ \textbf{grid} is the set $\grid$, and the elements
  of this set are called \textbf{cells}.
\end{definition}
When illustrating grids, we use the convention that the $n \times m$ grid has
$n$ columns and $m$ rows, which are described using $0$-indexed Cartesian
coordinates, where
$(0,0)$ is the cell in the lower left corner and
$(n-1, m-1)$ is the cell in the upper right corner.

\subsection{Symmetries of the grid}
We will count grids up to various symmetries, some of which may be specified by
subgroups of the dihedral group of the square,
\(R \leq D_8 = \langle r, f \mid r^4 = f^2 = (rf)^2 = \id \rangle\).
Because this is a group of \underline{r}otations and \underline{r}eflections of
the grid, we call this subgroup $R$. When considering the $n \times m$ grid for
$n \neq m$, we will further specify that $R \leq D_4$, where
$D_4 = \langle r^2, f \mid (r^2)^2 = f^2 = \id \rangle$ is the dihedral
group of the rectangle.\footnote{We will later see that we use $D_4$ in the case of the
$n \times n$ cylinder as well.}

In all cases, we will use the convention that our symmetry groups act on the
grid via \textit{right actions},
illustrated in Figure \ref{fig:dihedralConvention}. We will also use the
conventions that
$r$ acts on the grid by $+90^\circ$ rotations and
$f$ acts on the grid by horizontal reflection (i.e. over the vertical line).
Because the group acts on the right,
$rf$ reflects the square grid over the line $y = x$ and
$r^3f$ corresponds to matrix transposition.

\begin{figure}[ht]
  $\underbrace{\begin{tikzpicture}[scale=1.5]
    \path (0,-0.1) rectangle (1,1);
    \draw (0,0) rectangle (1,1);
  \end{tikzpicture}}_{\id}$
  \hfill
  \(\underbrace{\begin{tikzpicture}[scale=1.5]
    \path (0,-0.1) rectangle (1,1);
    \draw (0,0) rectangle (1,1);
    \draw[->] (5/6,1/2) arc(0:90:1/3);
    \node at (0.5,0.5) {$90^\circ$};
  \end{tikzpicture}}_{r}\)
  \hfill
  $\underbrace{\begin{tikzpicture}[scale=1.5]
    \path (0,-0.1) rectangle (1,1);
    \draw (0,0) rectangle (1,1);
    \draw[->] (5/6,1/2) arc(0:180:1/3);
    \node at (0.5,0.5) {$180^\circ$};
  \end{tikzpicture}}_{r^2}$
  \hfill
  $\underbrace{\begin{tikzpicture}[scale=1.5]
    \path (0,-0.1) rectangle (1,1);
    \draw (0,0) rectangle (1,1);
    \draw[->] (5/6,1/2) arc(0:-90:1/3);
    \node at (0.5,0.5) {$90^\circ$};
  \end{tikzpicture}}_{r^3}$
  \hfill
  $\underbrace{\begin{tikzpicture}[scale=1.5]
    \path (0,-0.1) rectangle (1,1);
    \draw (0,0) rectangle (1,1);
    \draw[dotted] (0.5,0) -- (0.5,1);
    \draw[<->] (0.25,0.5) to[bend left] (0.75,0.5);
  \end{tikzpicture}}_{f}$
  \hfill
  $\underbrace{\begin{tikzpicture}[scale=1.5]
    \path (0,-0.1) rectangle (1,1);
    \draw (0,0) rectangle (1,1);
    \draw[dotted] (0,0) -- (1,1);
    \draw[<->] (0.3,0.7) to[bend left] (0.7,0.3);
  \end{tikzpicture}}_{rf}$
  \hfill
  $\underbrace{\begin{tikzpicture}[scale=1.5]
    \path (0,-0.1) rectangle (1,1);
    \draw (0,0) rectangle (1,1);
    \draw[dotted] (0,0.5) -- (1,0.5);
    \draw[<->] (0.5,0.25) to[bend left] (0.5,0.75);
  \end{tikzpicture}}_{r^2f}$
  \hfill
  $\underbrace{\begin{tikzpicture}[scale=1.5]
    \path (0,-0.1) rectangle (1,1);
    \draw (0,0) rectangle (1,1);
    \draw[dotted] (0,1) -- (1,0);
    \draw[<->] (0.3,0.3) to[bend left] (0.7,0.7);
  \end{tikzpicture}}_{r^3f}$
\caption{Illustrations of the eight group actions of the dihedral group of the
square $D_8$, which we call
  ``identity'',
  ``$90^\circ$ rotation'',
  ``$180^\circ$ rotation'',
  ``$-90^\circ$ rotation'',
  ``horizontal reflection'',
  ``diagonal reflection'',
  ``vertical reflection'', and
  ``antidiagonal reflection'' respectively.
}
\label{fig:dihedralConvention}
\end{figure}
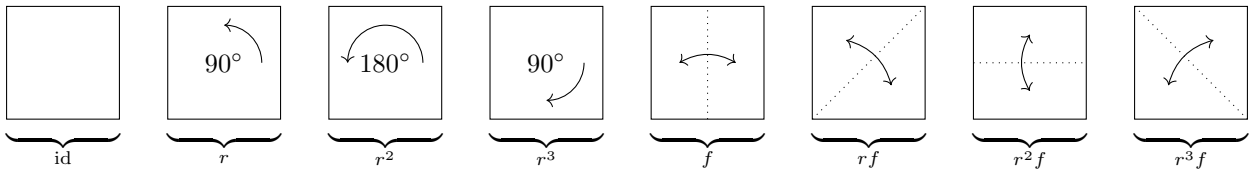

Formally, we describe the action on the cell by specifying how the generators of
$D_8$ act.
\begin{definition}
  The (right) action of an element of $D_4$ on a cell $(x, y) \in \grid$ is
  given by
  \begin{alignat*}{2}
    (x,y) &\cdot f    &&= (n-1-x,y), \qquad\text{ and} \\
    (x,y) &\cdot r^2  &&= (n-1-x,m-1-y).
  \end{alignat*}

  In the case of an $n \times n$ grid, the action of $r \in D_8$ is given by
  \begin{equation}
    (x,y) \cdot r = (n-1-y, x).
  \end{equation}
  \label{def:dihedralAction}
\end{definition}
These actions can be extended to all of $D_4$ and $D_8$ via the binary operation
of the group, since the group action is specified for the generators.

\subsection{Symmetries of the cylinder and torus}

Now that we know how the dihedral group acts on the $n \times n$ and
$n \times m$ grids, we can also look at symmetries of the grid by cyclic
shifting of rows and/or columns. When we shift just the columns\footnote{
  Shifting just the rows is equivalent to shifting just the columns.
  Here we will always use the convention of shifting just the columns.
}, we call this
a \textit{cylindrical action}, which we describe with the group $\Z/n\Z$;
when we shift the rows and columns, we call this a \textit{toroidal},
which we describe with the group $\Z/n\Z \times \Z/m\Z$.
Both of these are named in reference to the corresponding topological
identification of the square.

\begin{definition}
  The \textbf{cylindrical action} of $a \in \Z/n\Z$ on a cell
  $(x,y) \in \grid$
  corresponds to a (rightward) cyclic shift of columns: \[
    (x,y) \cdot a = (x+a,y).
  \]
  \label{def:cylindricalAction}
\end{definition}

\begin{definition}
  The \textbf{toroidal} of
  $(a,b) \in \Z/n \Z \times \Z/m \Z$ on a cell
  $(x,y) \in \grid$
  corresponds to
  a (rightward) cyclic shift of columns and
  an (upward) cyclic shift of rows: \[
    (x,y) \cdot (a,b) = (x+a,y+b).
  \]
  \label{def:torusAction}
\end{definition}
Definition \ref{def:torusAction} is illustrated in Figure \ref{fig:torusAction}.
\begin{figure}[ht]
  \[
  \begin{tikzpicture}[baseline=38]
    \foreach \a/\x/\y in {
      I/0/2, J/1/2, K/2/2, L/3/2,
      E/0/1, F/1/1, G/2/1, H/3/1,
      A/0/0, B/1/0, C/2/0, D/3/0
    } {
      \draw (\x,\y) rectangle (\x+1,\y+1);
      \node at (\x+1/2, \y+1/2) {\a};
    }
    \path (-0.1,0) rectangle (4.1,3);
    \draw[ultra thick] (0,0) rectangle (4,3);
    \draw[ultra thick, -{Triangle}] (0,1.52)--(0,1.525);
    \draw[ultra thick, -{Triangle}] (4,1.52)--(4,1.525);
    \draw[ultra thick, -{Triangle}{Triangle}] (2.29,0)--(2.3,0);
    \draw[ultra thick, -{Triangle}{Triangle}] (2.29,3)--(2.3,3);
  \end{tikzpicture}\cdot(3,1) =
  \begin{tikzpicture}[baseline=38]
    \foreach \a/\x/\y in {
      F/0/2, G/1/2, H/2/2, E/3/2,
      B/0/1, C/1/1, D/2/1, A/3/1,
      J/0/0, K/1/0, L/2/0, I/3/0
    } {
      \draw (\x,\y) rectangle (\x+1,\y+1);
      \node at (\x+1/2, \y+1/2) {\a};
    }
    \path (-0.1,0) rectangle (4.1,3);
    \draw[ultra thick] (0,0) rectangle (4,3);
    \draw[ultra thick, -{Triangle}] (0,1.52)--(0,1.525);
    \draw[ultra thick, -{Triangle}] (4,1.52)--(4,1.525);
    \draw[ultra thick, -{Triangle}{Triangle}] (2.29,0)--(2.3,0);
    \draw[ultra thick, -{Triangle}{Triangle}] (2.29,3)--(2.3,3);
  \end{tikzpicture}
  \]
  \caption{
    The $(3,1) \in \Z/4\Z \times \Z/3\Z$ acts on the
    $4 \times 3$ grid identified as a torus
    by cyclically shifting columns to the right by $3$ and
    cyclically shifting rows by $1$.
  }
\label{fig:torusAction}
\end{figure}
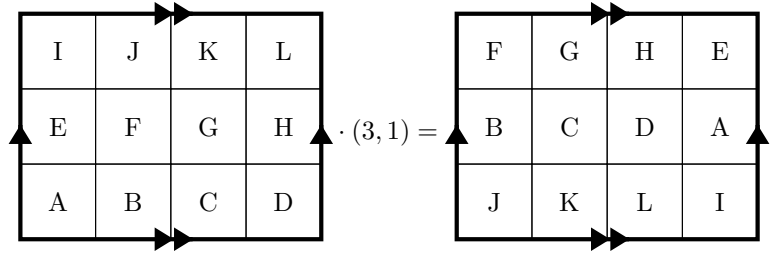

\subsection{Compatibility of grid symmetries}
Notice that we can act on the grid with both the dihedral actions and the
cylindrical/toroidal actions.
In order to make the group actions of the dihedral group compatible with the
the cylindrical action ($\Z/n\Z$)
or the toroidal action ($\grid$),
we define their (outer) semidirect product,
$\cylinderSymmetries$
or
$\torusSymmetries$ respectively.
As in the examples above, we define this as a right action,
thinking of this as first cyclically shifting the rows
and columns and then rotating or reflecting according to the element $R$.

The outer semidirect products are defined with respect to the homomorphisms
$\psi\colon D_4 \to \operatorname{Aut}(\Z/n\Z)$
and
$\phi\colon D_8 \to \operatorname{Aut}(\grid)$
respectively. (We use $D_4$ rather than $D_8$ in the case of the cylinder
because a $90^\circ$ rotation is not an isometry of the infinite strip, which
is the universal cover of the cylinder.)

\begin{definition}
Let $\psi\colon D_4 \to \operatorname{Aut}(\Z/n\Z)$ be
given by
\begin{alignat*}{2}
  \psi_f(x)     &= \psi_{r^2}(x)  &&= -x \text{ and}\\
  \psi_{\id}(x) &= \psi_{r^2f}(x) &&= x.
\end{alignat*}
Then the product of two elements in ${\Z/n\Z \rtimes D_4}$ is given by
\[
  \left(a_1,g_1\right)
  \left(a_2,g_2\right)
  =
  \left(a_1+\psi_{g_1}(a_1),g_1g_2\right).
\]
\end{definition}

In the case of the torus, the definition of the semidirect product
$\torusSymmetries$, where
$R \leq D_8$, is essentially similar.

\begin{definition}
  Let $\phi\colon D_8 \to \operatorname{Aut}(\grid)$
  be defined on the generators $r$ and $f$ by
  \begin{align*}
    \phi_f((x,y)) &= (-x,  y) \text{ and}\\
    \phi_r((x,y)) &= ( y, -x),
  \end{align*}
  and extended to the other elements of $D_8$.
  Then the binary operation of the semidirect product of
  $(\grid) \rtimes D_8$ is given by
  \[
    \ps{a_1}{b_1}{g_1}
    \ps{a_2}{b_2}{g_2}
    =
    \left((a_1,a_2)+\phi_{g_1}(a_1,a_2),g_1g_2\right).
  \]
\end{definition}
Using the facts that $r$ and $f$ generate $D_8$ and $\phi$ is a
homomorphism, together with function composition yields:
\begin{alignat*}{8}
  &\phi_{\id} ((x,y)) &&= ( x, y), \hspace{0.85cm}
  &\phi_{r}   ((x,y)) &=  ( y,-x), \hspace{0.85cm}
  &\phi_{r^2} ((x,y)) &=  (-x,-y), \hspace{0.85cm}
  &\phi_{r^3} ((x,y)) &=  (-y, x), \\
  &\phi_f     ((x,y)) &&= (-x, y), \hspace{0.85cm}
  &\phi_{rf}  ((x,y)) &=  ( y, x), \hspace{0.85cm}
  &\phi_{r^2f}((x,y)) &=  ( x,-y), \hspace{0.85cm}
  &\phi_{r^3f}((x,y)) &=  (-y,-x).
\end{alignat*}
\todo[inline]{Notice that this agrees with Definition \ref{def:groupActionOnTile} below.}

\begin{example}
  We check our work on an individual cell.
  For every choice of tile and pair of symmetries, we should have \[
    \left((x,y)\cdot\ps{a_1}{b_1}{g_1}\right) \cdot \ps{a_2}{b_2}{g_2}
    =
    (x,y) \cdot \left(\ps{a_1}{b_1}{g_1}\ps{a_2}{b_2}{g_2}\right).
  \]
  In particular, we check in the case of the $4 \times 4$ torus with
  $(x,y) = (1,0)$,
  $\ps{a_1}{b_1}{g_1} = \ps11f$
  and
  $\ps{a_2}{b_2}{g_2} = \ps20r$.

  \begin{align*}
    \left((1,0) \cdot \ps11f\right) \cdot \ps20r
    &= \left((2,1) \cdot f \right) \cdot \ps20r \\
    &= (1,1) \cdot \ps20r \\
    &= (3,1) \cdot r \\
    &= (2,3).
  \end{align*}
  Now using the semidirect product, \begin{align*}
    (1,0)    \cdot \left(\ps11f\ps20r\right)
    &= (1,0) \cdot \left((1,1) + \phi_f(2,0), fr\right) \\
    &= (1,0) \cdot \left((1,1) + (-2,0), r^3f\right) \\
    &= (1,0) \cdot \ps31{r^3f} \\
    &= (0,1) \cdot r^3f \\
    &= (2,3).
  \end{align*}
  This suggests that this semidirect product is the appropriate way to make the
  dihedral actions compatible with the toroidal action.
\end{example}

\subsection{Symmetries of tile designs}
We are now ready to start filling in our grid with tiles. Before defining
what tiles are, we introduce the following definition for convenience.
\begin{definition}
  If $X$ is a set and $G$ has a group action on $X$, then we call $X$ a $G$\textbf{-set}.
\end{definition}
\begin{definition}
  Given $R \leq D_8$, a \textbf{set of tile designs} is simply an $R$-set.
  A \textbf{tile design} is any element of such a set.
\end{definition}

We will always illustrate our tile designs with squares that have
designs in them, but any abstract $R$-set will work in place of these
illustrations. When we specify one of these tile designs together with a
cell, we get a tile.
\begin{definition}
  A \textbf{tile} in the $n \times m$ grid with the set of tile designs $T$ is
  an element of $(\grid) \times T$.
\end{definition}
Now we are ready to define a tiling of the grid, which is a specification of a
tile design for each cell.
\begin{definition}
  A \textbf{tiling} of the $n \times m$ grid with the set of tile designs $T$ is
  a map $f\colon(\grid) \to T$.
  The tiles associated with the tiling are elements of the graph of the map \[
    \{\ps xy{f(x,y)} \mid (x,y) \in \grid\}.
  \]
\end{definition}
We next give an example to illustrate all of these definitions.
\begin{example}
  Suppose that $R = D_4 = \langle r^2, f \rangle$, the dihedral group of the
  rectangle. Then \[
    T = \left\{
      \tileColor[black!90!white],
      \tileVert,
      \tileDiag,
      \tileAdiag,
      \tileAsym{0},
      \tileAsym{2},
      \tileAsym{5},
      \tileAsym{7},
      \tileRR{1},
      \tileRR{2},
      \tileU,
      \tileD
    \right\}.
  \]
  is a set of tile designs, because it is an $R$-set.
  The $3 \times 2$ tiling \[
    \begin{tikzpicture}
      \fill[black!90!white] (0,0) rectangle (1,1);
      \fill[black!90!white] (2,1) rectangle (3,2);
      \fill (0,1) rectangle (1/3,2)
            (2/3,1) rectangle (1,2);
      \fill (1,1) -- (1.5,2) -- (2,1) -- cycle;
      \fill (1,0) -- (2,0) -- (2,1/2) -- cycle;
      \fill (2,1) -- (2.5,0) -- (3,1) -- cycle;
      \draw[gray] (0,0) grid (3,2);
    \end{tikzpicture}
  \] consists of the six tiles \[
    \left\{
    \ps00{\tileColor[black!90!white]}
    \ps01{\tileVert}
    \ps10{\tileAsym{7}}
    \ps11{\tileU}
    \ps20{\tileD}
    \ps21{\tileColor[black!90!white]}
    \right\}.
  \]
\end{example}

Since we have now defined tiles, we are ready to talk about how the symmetries
of our grid, cylinder, or torus act on tiles: dihedral
actions can act on the tile design nontrivially, but cylindrical/toroidal
actions always act on a tile design by the identity.
We extend Definitions \ref{def:dihedralAction}, \ref{def:cylindricalAction}, and
\ref{def:torusAction} to this setting in a direct way, and illustrate this
in Figure \ref{fig:dihedralAction}.

\begin{definition}
  If $\xyd$ is a tile in an $n \times m$ grid, then
  $g \in D_4$ acts on $\xyd$ by \begin{alignat*}{2}
    \xyd &\cdot \id  &&= \xyd, \\
    \xyd &\cdot r^2  &&= \ps{n-1-x}{m-1-y}{d\cdot  r^2}, \\
    \xyd &\cdot f    &&= \ps{n-1-x}{    y}{d\cdot    f}, \text{ and} \\
    \xyd &\cdot r^2f &&= \ps{    x}{m-1-y}{d\cdot r^2f}.
  \end{alignat*}
  Furthermore, if $\xyd$ is a tile in an $n \times n$ grid, then
  $g \in D_8$ acts on $\xyd$ by the above actions together with
  \begin{alignat*}{2}
    \xyd &\cdot r    &&= \ps{n-1-y}{    x}{d\cdot    r}, \\
    \xyd &\cdot r^3  &&= \ps{    y}{n-1-x}{d\cdot  r^3}, \\
    \xyd &\cdot rf   &&= \ps{    y}{    x}{d\cdot   rf}, \text{ and} \\
    \xyd &\cdot r^3f &&= \ps{n-1-y}{n-1-x}{d\cdot r^3f}.
  \end{alignat*}
  \label{def:groupActionOnTile}
\end{definition}

\begin{figure}[ht]
  \centering
  \begin{subfigure}[b]{0.5\textwidth}
  \[
  \begin{tikzpicture}[baseline=23]
    \fill[black!90!white] (0,0) rectangle (1,1);
    \fill[black!90!white] (2,1) rectangle (3,2);
    \fill (0,1) rectangle (1/3,2)
          (2/3,1) rectangle (1,2);
    \fill (1,1) -- (1.5,2) -- (2,1) -- cycle;
    \fill (1,0) -- (2,0) -- (2,1/2) -- cycle;
    \fill (2,1) -- (2.5,0) -- (3,1) -- cycle;
    \draw[gray] (0,0) grid (3,2);
  \end{tikzpicture}\cdot r^2f =
  \begin{tikzpicture}[baseline=23]
    \fill[black!90!white] (0,1) rectangle (1,2);
    \fill[black!90!white] (2,0) rectangle (3,1);
    \fill (0,0) rectangle (1/3,1)
          (2/3,0) rectangle (1,1);
    \fill (1,1) -- (1.5,0) -- (2,1) -- cycle;
    \fill (1,2) -- (2,2) -- (2,3/2) -- cycle;
    \fill (2,1) -- (2.5,2) -- (3,1) -- cycle;
    \draw[gray] (0,0) grid (3,2);
  \end{tikzpicture}
  \]
  \caption{}
\end{subfigure}
\hfill
\begin{subfigure}[b]{0.4\textwidth}
  \centering
  \[\ps10{\tileAsym{7}}\cdot r^2f = \ps11{\tileAsym{2}}\]
  \caption{}
\end{subfigure}
  \caption{An example of the action of $r^2f$ (a vertical reflection)
    on
    (a) a tiling of the $3 \times 2$ grid
    and on
    (b) a specific tile.
  }
  \label{fig:dihedralAction}
\end{figure}
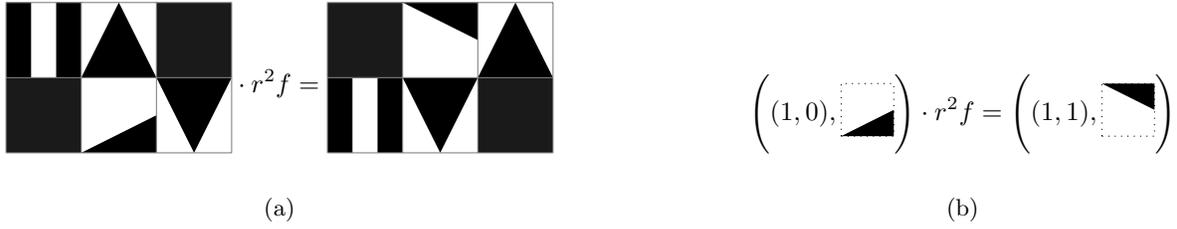

In order to understand what makes two tiles \textit{essentially different} with
respect to counting tilings, it is useful to define the notion of a stabilizer
subgroup.
\begin{definition}
  Let $X$ be a $G$-set. Then the \textbf{stabilizer subgroup} of an element
  $x \in X$ is the subgroup \[
    G_x = \{g \in G \mid x \cdot g = x\} \leq G.
  \]
\end{definition}

Because our set of tile designs $T$ is an $R$-set, the relevant difference
between different tile designs for the purpose of counting tilings is their
stabilizer subgroups.

\subsection{Classifying sets of tile designs}
In order to describe the essential features of a set of tile designs, we
partition it into orbits with respect to $R$. By counting up the number of
orbits and classifying each orbit by the (conjugacy class of the) stabilizer
subgroup of one of its representatives we can understand the combinatorics
of the set of tile designs completely.

\begin{definition}
  Let $R \subseteq D_8$ and let $T$ be a set of tile designs (with respect to $R$).
  Then for each (conjugacy class of) a subgroup $S \leq R$, let
  $\mathcal{O}^R_{S}$ denote the number of orbits that contain a tile
  whose stabilizer subgroup is conjugate to $S$.
\end{definition}

Notice that we classify up to conjugacy class because if $d$ is stable under
$S$, then $d \cdot g$ is stable under $g^{-1}Sg$
since $(d \cdot g)\cdot g^{-1}Sg = d \cdot Sg = d \cdot g$.

\begin{example}
  Suppose we are counting tilings of the grid, cylinder, or torus up to
  horizontal and vertical reflection with a set of tile designs given by
  \[
    T = \left\{
      \tileColor[black!90!white],
      \tileVert,
      \tileRR{3},
      \tileRR{0},
      \tileNW,
      \tileNE,
      \tileSW,
      \tileSE,
      \tileRR{1},
      \tileRR{2},
      \tileU,
      \tileD
    \right\}.
  \]
  Since the symmetry group $D_4 = \langle r^2, f \ |\ (r^2)^2 = f^2 = \id\rangle$
  has $5$ conjugacy classes of subgroups, there are five type of orbits:
  \begin{alignat*}{2}
    \orb{r^2,f}{r^2,f} &= 2 &&\qquad\text{via}\quad\left\{\tileColor[black!90!white]\right\} \text{ and } \left\{\tileVert\right\} \\
    \orb{r^2,f}{f}     &= 1 &&\qquad\text{via}\quad\left\{\tileU, \tileD\right\} \\
    \orb{r^2,f}{r^2}   &= 2 &&\qquad\text{via}\quad\left\{\tileRR{1}, \tileRR{2}\right\} \text{ and } \left\{\tileRR{3}, \tileRR{0}\right\} \\
    \orb{r^2,f}{r^2f}  &= 0 && \\
    \orbid{r^2,f}      &= 1 &&\qquad\text{via}\quad\left\{\tileNW,\tileNE,\tileSW,\tileSE\right\}.
  \end{alignat*}
  \end{example}


\begin{lemma}
The number of tilings for a given $R$-tiling $T$ only depends on the tuple
\[
  \bigoplus_{S \in \operatorname{conj}(R)} \mathcal{O}^R_{S},
\] where $\operatorname{conj}(R)$ is the set of equivalence classes of subgroups
of $R$ up to conjugacy.
\end{lemma}
\begin{proof}
  Suppose that we have two sets of tile designs $T$ and $T'$ with the same
  number of orbits for each stabilizer conjugacy class.
  There exists a bijection
  $f \colon T \to T'$ such that
  $f(d) = d'$ whenever $d$ and $d'$ have the same
  stabilizer subgroup in $R$, that is, $R_d = R_{d'}$.
  Then the induced map of $f$ to the tilings is also bijection of tilings.
\end{proof}
In Appendix \ref{apss:nXnTorusSequences}, we explicitly enumerate all of the
$R$-sets of tile designs that consist of a single orbit
for each subgroup $R \leq D_4$ or $R \leq D_8$.

\subsection{Counting strategy}
In order to count how many tilings exist up to various symmetries, we will use
Burnside's lemma.
\begin{theorem}[Burnside's lemma]
  Let $X$ be a $G$-set. Then the size of $X$ up to the action of $G$ is \[
    |X/G| = \frac{1}{|G|}\sum_{g \in G}|X^g|,
  \] where $|X^g|$ is the number of elements of $X$ that are fixed under the
  action of $g \in G$.
\end{theorem}
We want to understand how many tilings are fixed under various symmetries. To
do this it is necessary to categorize the symmetries of various tiles.
\begin{definition}
  Let $T$ be a set of tile designs, where $R \leq D_8$.
  For each $g \in R$, the set of tiles that are fixed by $g$ is denoted
  \[
    T^g = \{d \in T \mid d \cdot g = d\},
  \]
  and the size of this set is denoted
  \[
    t_g = |T^g|.
  \]
\end{definition}

Note that $T^{\id} = T$, so $t_{\id}$ is the total number of tiles designs.

The following theorem gives us a strategy for counting the number of tilings
that are fixed under a given symmetry, which is
illustrated in Figure \ref{fig:cellOrbits}.

\begin{theorem}
  Suppose that $s = g \in R$,
  $s = (a,g) \in \cylinderSymmetries$, or
  $s = \abg \in \torusSymmetries$.

  Since the set of cells $\grid$ form an $\langle s \rangle$-set, we
  partition the cells into orbits with respect to the cyclic subgroup
  $\langle s \rangle$, which we call $\Theta_s$, so that
  $\bigsqcup_{\vartheta \in \Theta_s} \vartheta = \grid$.

  Then if $X^s$ is the set of tilings of the $n \times m$ grid that are
  stable under $s$,
  \[
    |X^s| = \prod_{\vartheta \in \Theta_s} t_{g^{|\vartheta|}}.
  \]
\end{theorem}
\begin{proof}
  Because $\Theta_s$ partitions the cells into orbits with respect to the cyclic
  subgroup $\langle s \rangle$, the number of tilings that are fixed under $s$
  is equal to the product of the number of tilings of each orbit of cells
  under $\langle s \rangle$.

  The tiling of an orbit of cells can be specified by a single tile $d$, which
  then  determines the rest of the orbit by $\xyd \cdot s^k$ for
  $0 \leq k < |\vartheta|$. The only requirement for a valid tiling of a orbit
  is that \[
    \xyd = \xyd \cdot s^{|\vartheta|} = \ps{x}{y}{d \cdot g^{|\vartheta|}},
  \] thus $d$ must be fixed by $g^{|\vartheta|}$,
  and so $d \in T^{g^{|\vartheta|}}$.
  Therefore there are $t_{g^{|\vartheta|}}$ choices for $d$ and thus for the
  orbit of cells containing $(x,y)$.
\end{proof}

Thus, this reduces the problem to a matter of counting the orbits of cells
under each symmetry $s$ along with counting the sizes of each of these
orbits.

We proceed with Sections \ref{sec:grid}, \ref{sec:cylinder}, and
\ref{sec:torus}, which all implement the above strategy.
Each section consists broadly of fixed point counting theorems,
which count tilings of the grid that are fixed under the actions of
$D_4$ or $D_8$ for arbitrary sets of tile designs.

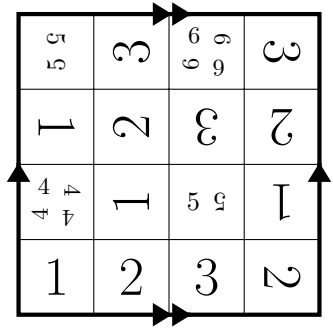
\begin{figure}[ht]
  \centering
  \begin{tikzpicture}
    \draw (0,0) grid (4,4);
    \draw[ultra thick] (0,0) rectangle (4,4);
    \draw[ultra thick, -{Triangle}] (0,2.02)--(0,2.025);
    \draw[ultra thick, -{Triangle}] (4,2.02)--(4,2.025);
    \draw[ultra thick, -{Triangle}{Triangle}] (2.29,0)--(2.3,0);
    \draw[ultra thick, -{Triangle}{Triangle}] (2.29,4)--(2.3,4);
    \node             at (0.5,0.5) {\huge 1};
    \node[rotate=90]  at (1.5,1.5) {\huge 1};
    \node[rotate=180] at (3.5,1.5) {\huge 1};
    \node[rotate=270] at (0.5,2.5) {\huge 1};
    \node             at (1.5,0.5) {\huge 2};
    \node[rotate=90]  at (1.5,2.5) {\huge 2};
    \node[rotate=180] at (3.5,2.5) {\huge 2};
    \node[rotate=270] at (3.5,0.5) {\huge 2};
    \node             at (2.5,0.5) {\huge 3};
    \node[rotate=90]  at (1.5,3.5) {\huge 3};
    \node[rotate=180] at (2.5,2.5) {\huge 3};
    \node[rotate=270] at (3.5,3.5) {\huge 3};
    \node[align=left]             at (0.5,1.5) {\small 4~~~~\\~};
    \node[align=left, rotate=90]  at (0.5,1.5) {\small 4~~~~\\~};
    \node[align=left, rotate=180] at (0.5,1.5) {\small 4~~~~\\~};
    \node[align=left, rotate=270] at (0.5,1.5) {\small 4~~~~\\~};
    \node[align=left]             at (2.5,1.5) {5~~~};
    \node[align=left, rotate=90]  at (0.5,3.5) {5~~~};
    \node[align=left, rotate=180] at (2.5,1.5) {5~~~};
    \node[align=left, rotate=270] at (0.5,3.5) {5~~~};
    \node[align=left]             at (2.5,3.5) {\small 6~~~~\\~};
    \node[align=left, rotate=90]  at (2.5,3.5) {\small 6~~~~\\~};
    \node[align=left, rotate=180] at (2.5,3.5) {\small 6~~~~\\~};
    \node[align=left, rotate=270] at (2.5,3.5) {\small 6~~~~\\~};
  \end{tikzpicture}
  \caption{An illustration showing a tiling of the $4 \times 4$ torus that is
  fixed under $\ps12r$, and the six orbits of its cells with
  respect to the subgroup generated by this symmetry.
  There are three orbits of size $4$ (whose tiles are stable under $r^4=\id$),
  one orbit of size $2$ (whose tiles are stable under $r^2$), and
  two orbits of size $1$ (whose tiles are stable under $r$).
  }
  \label{fig:cellOrbits}
\end{figure}

\section{Grid}
\label{sec:grid}
For counting tilings of the $n \times m$ rectangular grid or $n \times n$ square
grid under subgroups $R \leq D_4$ and $R \leq D_8$ respectively, we
count the number of tilings that are fixed under each element of $R$.

In the following two subsections, we denote the rectangular grid by $\RG$
and the square grid $\SG$.
\subsection{The \texorpdfstring{$n \times m$}{n by m} grid}
We begin by specifying the number of tilings that are fixed under each symmetry.
\begin{definition}
  For a given set of tile designs $T$, and an element
  \(
    g \in R \leq D_4,
  \)
  The number of tilings of the $n \times m$ grid by tile designs in $T$
  that are fixed by $g$ is denoted $\fxpt^\RG_g(n,m)$.
\end{definition}

\begin{theorem}
  For a given set of tile designs $T$ and an element $g \in R \leq D_4$,
  the number of tilings of the $n \times m$ grid by tile designs in $T$
  that are fixed by $g$ is
  \begin{align}
    \fxpt^\RG_{\id}(n,m) &= t_{\id}^{nm}.
    \label{eq:fxpt_R_id}
    \\
    \fxpt^\RG_{r^2}(n,m) &= \begin{cases}
      t_{\id}^{nm/2}            & nm \text{ even} \\
      t_{\id}^{(nm-1)/2}t_{r^2} & nm \text{ odd}.
    \end{cases}
    \label{eq:fxpt_R_rr}
    \\
    \fxpt^\RG_{f}(n,m) &= \begin{cases}
      t_{\id}^{nm/2}            & n \text{ even} \\
      t_{\id}^{(m(n-1))/2}t_f^m & n \text{ odd}.
    \end{cases}
    \label{eq:fxpt_R_f}
    \\
    \fxpt^\RG_{r^2f}(n,m) &= \begin{cases}
      t_{\id}^{nm/2}                 & m \text{ even} \\
      t_{\id}^{(n(m-1))/2}t_{r^2f}^n & m \text{ odd}.
    \end{cases}
    \label{eq:fxpt_R_rrf}
  \end{align}
  Then the number of distinct tilings of the $n \times m$ grid up to action of
  $R$ is given by
  \begin{equation}
    \frac{1}{|R|}\sum_{g \in R} \fxpt^\RG_g(n,m).
    \label{eq:distinctTilingsR}
  \end{equation}
  \label{thm:nXmGridFormulae}
\end{theorem}
\begin{proof} The proof will consist of three cases.

  \begin{description}
    \item[Equation \eqref{eq:fxpt_R_id}.]
    This follows from the fact that $t_{\id}$ is the
    number of distinct tiles, and every tiling is fixed under $\id \in D_4$.
    \item[Equation \eqref{eq:fxpt_R_rr}.]
    This follows from the fact that the (right) action of $r^2$ on the cell
    $(x,y)$ is \[
      (x,y) \cdot r^2 = (n-x-1, m-y-1).
    \] Since $r^2$ has order $2$, each cell is in an orbit of size $1$ or $2$.
    The cell $(x,y)$ is fixed under the action of $r^2$ if and only if
    $n$ and $m$ are both odd and $(x,y) = (\frac{n-1}{2}, \frac{m-1}{2})$.

    Therefore when $nm$ is even, the grid is partitioned into $nm/2$ orbits of
    size $2$, so any fixed tiling can be specified by choosing any tile design
    in $T$ for each orbit.
    When $nm$ is even, the grid is partitioned into one orbit of size $1$
    together with $(nm-1)/2$ orbits of size $2$,
    so any fixed tiling can be specified by choosing a tile design in $T^{r^2}$
    for the fixed point and any tile design in $T$ for each orbit.
    \item[Equations \eqref{eq:fxpt_R_f} and \eqref{eq:fxpt_R_rrf}.]
    These two equations are essentially the
    same, so without loss of generality, we will prove the case of
    Equation \eqref{eq:fxpt_R_f}. The right action of $f$ on $\xyd$ is \[
      \xyd \cdot f = ((n - x - 1, m), d \cdot f).
    \] Because $f$ is order $2$, we can conclude that $(x,y)$ is either a fixed
    point or a $2$-cycle with respect to $f$. It follows that $(x,y)$ is a fixed
    point if and only if $n$ is odd and $x = (n - 1)/2$, therefore when $n$ is
    odd the tiling has $m$ fixed cells.
    The fixed cells can be specified by any tile design in $T^f$, and the cells
    in orbits of size $2$ can be specified by any tile design in $T$.
    \item[Equation \eqref{eq:distinctTilingsR}] Finally, this is a direct
    application of Burnside's lemma.
  \end{description}
\end{proof}

\subsection{The \texorpdfstring{$n \times n$}{n by n} grid}
There are more symmetries and some specializations in the case of the
$n \times n$ grid, which we denote $\SG$ for ``square grid''.

\begin{definition}
  For a given set of tile designs $T$, and an element
  \(
    g \in R \leq D_8,
  \)
  the number of tilings of the $n \times n$ grid by tile designs in $T$
  that are fixed by $g$ is denoted $\fxpt^\SG_g(n)$.
\end{definition}

\begin{theorem}
  For a given set of tile designs $T$ and an element $g \in R \leq D_8$,
  the number of tilings of the $n \times m$ grid by tile designs in $T$
  that are fixed by $g$ is \begin{align}
    \fxpt^{\SG}_{\id}(n)
    &= \fxpt^\RG_{\id}(n,n) = t_{\id}^{n^2}
    \label{eq:fxpt_S_id}
    \\
    \fxpt^{\SG}_{r^2}(n)
    &= \fxpt^\RG_{r^2}(n,n) = \begin{cases}
      t_{\id}^{n^2/2} & n \text{ even} \\
      t_{\id}^{(n^2-1)/2}t_{r^2} & n \text{ odd}
    \end{cases}
    \label{eq:fxpt_S_rr}
    \\
    \fxpt^{\SG}_{f}(n)
    &= \fxpt^\RG_f(n,n) = \begin{cases}
      t_{\id}^{n^2/2} & n \text{ even} \\
      t_{\id}^{(n^2-n)/2}t_{f}^n & n \text{ odd}
    \end{cases}
    \label{eq:fxpt_S_f}
    \\
    \fxpt^{\SG}_{r^2f}(n)
    &= \fxpt^\RG_{r^2f}(n,n) = \begin{cases}
      t_{\id}^{n^2/2} & n \text{ even} \\
      t_{\id}^{(n^2-n)/2}t_{r^2f}^n & n \text{ odd}
    \end{cases}
    \label{eq:fxpt_S_rrf}
    \\
    \fxpt^{\SG}_r(n)
    &= \fxpt^{\SG}_{r^3}(n) = \begin{cases}
      t_{\id}^{n^2/4} & n \text{ even} \\
      t_{\id}^{(n^2-1)/4}t_{r} & n \text{ odd}
    \end{cases}
    \label{eq:fxpt_S_r}
    \\
    \fxpt^{\SG}_{rf}(n)
    &= t_{\id}^{(n^2-n)/2}t_{rf}^n.
    \label{eq:fxpt_S_rf}
    \\
    \fxpt^{\SG}_{r^3f}(n)
    &= t_{\id}^{(n^2-n)/2}t_{r^3f}^n.
    \label{eq:fxpt_S_rrrf}
  \end{align}
  Then the number of distinct tilings of the $n \times n$ grid up to the
  dihedral action of the square is given by
  \begin{equation}
    \frac{1}{|R|}\sum_{g \in R} \fxpt^{\SG}_g(n).
    \label{eq:distinctTilingsS}
  \end{equation}
  \label{thm:nXnGridFormulae}
\end{theorem}
\begin{proof} This proof proceeds with four cases.
  \begin{description}
    \item[Equations \eqref{eq:fxpt_S_id}, \eqref{eq:fxpt_S_rr}, \eqref{eq:fxpt_S_f}, and \eqref{eq:fxpt_S_rrf}.]
    These follow directly from Theorem \ref{thm:nXmGridFormulae}, by specifying $m = n$.
    \item[Equation \eqref{eq:fxpt_S_r}.]
    Firstly, notice that the tilings that are fixed under $r$ are identically
    those that are fixed under $r^{-1} = r^3$.
    The (right) action of $r \in D_8$ on a cell $(x,y)$ is \[
      (x,y) \cdot r = (n - y - 1, x)
    \] therefore $(a,b)$ is a fixed point if and only if it satisfies the
    system of equations \begin{align}
      a &= n - b - 1 \\
      b &= a,
    \end{align} which has an integer solution only when $n$ is odd and when
    $(a,b) = (\frac{n - 1}{2}, \frac{n - 1}{2})$.
    Also, there are no cells that occur in $2$-cycles. This can be seen by
    noticing that cells that occur in $2$-cycles are also fixed points under
    $r^2$, and by the proof of Theorem \ref{thm:nXmGridFormulae}, we know that
    this occurs under the same conditions as the fixed points under $f$.
    Therefore all other cells occur in $4$-cycles.

    Therefore when $n$ is even, the grid is partitioned into $n^2/4$ orbits of
    size $4$, each of which can be specified by any tile design in $T$;
    when $n$ is odd, the grid has one fixed point, which must be tiled with a
    tile design in $T^r$, and the remaining cells can be partitioned
    into $(n^2 - 1)/4$ orbits of size $4$,
    each of which can be specified by any tile design in $T$.
    \item[Equations \eqref{eq:fxpt_S_rf} and \eqref{eq:fxpt_S_rrrf}.]
    Because $rf$ and $r^3f$ are conjugate, these are essentially similar, so
    without loss of generality, we will prove Equation \eqref{eq:fxpt_S_rf}.

    The (right) action of $rf \in D_8$ on a cell $(x,y)$ is \[
      (x,y) \cdot rf = (y,x).
    \] Thus, $(x,y)$ is a fixed point if and only if $x = y$, otherwise it is
    a part of a $2$-cycle. Therefore there are $n$ fixed points, which can be
    specified by a tile design in $T^{rf}$ and $(n^2 - n)/2$ $2$-cycles, which
    can be specified with any tile design in $T$.
    \item[Equation \eqref{eq:distinctTilingsS}.]
    The final equation follows by a direct application of Burnside's lemma.
  \end{description}
\end{proof}

\section{Cylinder}
\label{sec:cylinder}
Here we use the convention that the $n \times m$ cylinder is identified along
its left and right sides, as illustrated in Figure \ref{fig:cylinderIdentification}.

Even in the case of $n \times n$ grids, we only consider tilings of cylinders up
to subgroups of the dihedral group of the rectangle,
because other symmetries of the square would result in swapping
the pair of identified sides (the right and left side) of the grid with the pair
of non-identified sides (the top and bottom).

In both the case of the cylinder and the torus, we will repeatedly use the
following observation.
\begin{lemma}
  For fixed values of $n$ and $a$, the equation \begin{equation}
    x \equiv -1 - x - a \pmod n
    \label{eq:flip_solution}
  \end{equation} has solutions that depend on the parity of $n$ and $a$.

  \noindent
  When $n$ is odd, there is one solution: \begin{equation}
    x \equiv \frac{n+1}{2}(-1-a) \pmod{n}.
    \label{eq:flip_solution_odd_n}
  \end{equation}
  When $n$ is even and $a$ is odd, there are two solutions: \begin{align}
    x &\equiv \frac{-1-a}{2} \pmod n
    \label{eq:flip_solution_even_n1}
    \\[10pt]
    x &\equiv \frac{n-1-a}{2} \pmod n.
    \label{eq:flip_solution_even_n2}
  \end{align}
  When $n$ and $a$ are both even, there are no solutions.
  \label{lem:flip_solution}
\end{lemma}
\begin{proof}
  In both cases, we write equation \eqref{eq:flip_solution} as \[
    2x \equiv -1 - a \pmod{n}.
  \]
  \begin{description}
    \item[Odd $n$.] When $n$ is odd, $2$ has a multiplicative inverse of
    $(n+1)/2$, so multiplying gives the unique solution described in
    equation \eqref{eq:flip_solution_odd_n}.
    \item[Even $n$ and odd $a$.] When $n$ is even and $a$ is odd,
    $-1-a$ is even. Dividing by $2$ gives the solution given in
    equation \eqref{eq:flip_solution_even_n1}, and
    adding $n$ and dividing by $2$ gives the solution in equation
    \eqref{eq:flip_solution_even_n2}.
    \item[Even $n$ and $a$.] When both $n$ and $a$ are even,
    $2x$ is even and $-1-a$ is odd, so there are no solutions.
  \end{description}
\end{proof}
Similarly, we will repeatedly use the following lemma when counting fixed points
for both the cylinder and the torus.
\begin{lemma}[\cite{MacMahon}]
  Given some $a \in \Z/n\Z$, if $d$ is a minimal solution to the equation \[
    da \equiv 0 \pmod n
  \] then $d \mid n$.
  Moreover, when $d$ is a divisor of $n$, there are $\varphi(d)$ choices for
  $a$ that result in $d$ being a minimal solution, where $\varphi$ is
  Euler's totient function.
  \label{lem:eulerPhi}
\end{lemma}
\begin{proof}
  First, we can see that the least value of $da$ will occur when
  $da = \lcm(a,n)$, and thus \[d = \lcm(a,n)/a = n/\gcd(a,n).\]
  Therefore $d$ must be a divisor of $n$.
  The $\varphi(d)$ choices for $a$ such that $d$ is a minimal solution to
  ${da \equiv 0 \pmod n}$ are
  $a \in \left\{kn/d \mid 1 \leq k \leq d\ \text{and}\ \gcd(k,d) = 1\right\}$.
\end{proof}
\subsection{The \texorpdfstring{$n \times m$}{n by m} cylinder}
We denote the $n \times m$ cylinder by the superscript $\CY$.
\begin{definition}
  For a given set of tile designs $T$, and an element
  \(g \in R \leq D_4\)
  the sum over all cyclic shifts of the number of tilings of the $n \times m$
  cylinder by tile designs in $T$ that are fixed by
  $(a,g) \in \cylinderSymmetries$
  is denoted \[
    \fxpt^\CY_g(n,m) = \sum_{a \in \Z/n\Z} X^{(a,g)},
  \] where $X^{(a,g)}$ is the number of tilings fixed by $(a,g)$.
\end{definition}

\begin{theorem}
  The sum over all cyclic shifts of the number of tilings of the $n \times m$
  cylinder by tile designs in $T$ that are fixed by
  $(a,\id) \in \cylinderSymmetries$
  is given by
  \begin{equation}
    \fxpt^\CY_{\id}(n,m) = \sum_{d | n} \varphi(d)t_{\id}^{nm/d}.
    \label{eq:fxpt_C_id}
  \end{equation}
\end{theorem}
\begin{proof}
  For each element $(a, \id) \in \cylinderSymmetries$, the size of the orbits
  is the least $d$ such that \(
      da \equiv 0 \pmod n.
  \)
  By Lemma \ref{lem:eulerPhi}, when $d \mid n$ there are $\varphi(d)$ choices
  for $a$ that result in orbits of size $d$, and each choice partitions the
  $n \times m$ grid into $nm/d$ orbits.
\end{proof}

The next theorem concerns the action of $(a,r^2) \in \cylinderSymmetries$, which is illustrated in Figure \ref{fig:cylinderShiftAndRotate}.
\begin{theorem}
  The sum over all cyclic shifts of the number of tilings of the $n \times m$
  cylinder by tile designs in $T$ that are fixed by
  $(a,r^2) \in \cylinderSymmetries$
  is given by
  \begin{subnumcases}{\fxpt^\CY_{r^2}(n,m) =}
      nt_{\id}^{nm/2}
      & $m$ is even \label{eq:fxpt_C_rr_1}
      \\[10pt]
      \displaystyle
      n\!\left(
        \frac12 t_{\id}^{nm/2} +
        \frac12 t_{\id}^{(nm-2)/2}t_{r^2}^{2}
      \right)
      & $m$ is odd, $n$ is even \label{eq:fxpt_C_rr_2}
      \\[10pt]
      nt_{\id}^{(nm-1)/2}t_{r^2}
      & $m$ and $n$ are odd. \label{eq:fxpt_C_rr_3}
  \end{subnumcases}
\end{theorem}
\begin{proof}
  The right action of $(a,r^2)$ on $\xyd$ is \begin{equation}
    \xyd \cdot (a,r^2) = ((n - 1 - x - a, m - 1 - y), d \cdot r^2).
  \end{equation}
  By applying this map twice, we see that $\xyd^2 = \id$, so each
  orbit is either size $1$ or size $2$. Orbits are size $1$ precisely when
  \begin{align}
    x &\equiv -1 - x - a \pmod n \\
    y &= m - y - 1.
  \end{align}
  \begin{description}
    \item[Equation \eqref{eq:fxpt_C_rr_1}.]
    When $m$ is even, $2y \neq m - 1$, so there are no solutions to this
    system of equations. Thus all orbits have size $2$ totaling $nm/2$ orbits.
    We can choose any tile design in $T$ to start this orbit. Then we sum this
    over all choices $n$ choices of $a \in \Z/n\Z$.
    \item[Equation \eqref{eq:fxpt_C_rr_2}.]
    When $m$ is odd, the second equation has the unique solution of
    $y = (m-1)/2$, which represents the middle row.
    The solutions for the second equation follow directly from Lemma
    \ref{lem:flip_solution}, which states that when $n$ is even, the
    equation has two solutions when $x$ is odd and none otherwise;
    when $n$ is odd the second equation has a single solution.

    Therefore when $m$ is odd and $n$ is even, half of the choices of
    $a \in \Z/n\Z$ result in no orbits of size $1$, and the other half of
    choices of $x$ result in two orbits of size $1$.
    In the former case, the grid decomposes into $nm/2$ orbits of size $2$,
    each of which can be filled with any tile design.
    In the latter case, there are two orbits of size $1$, which must be filled
    with a tile that is fixed under $r^2$, and the rest of the grid decomposes
    into $(nm-2)/2$ orbits all of size $2$, which can be filled with any tile
    design.
    \item[Equation \eqref{eq:fxpt_C_rr_3}.]
    Lastly, when both $m$ and $n$ are odd, there is a single orbit of size $1$
    that must be filled with a tile that is fixed under $r^2$, the remaining
    $nm-1$ cells are partitioned into $(nm-1)/2$ orbits of size $2$ that can
    be filled with any tile design.
  \end{description}
\end{proof}

\begin{figure}[ht]
  \centering
  \begin{tikzpicture}[baseline=35]
    \foreach \x/\y/\a in {
      0/2/A, 1/2/B, 2/2/C, 3/2/D, 4/2/E,
      0/1/F, 1/1/G, 2/1/H, 3/1/I, 4/1/J,
      0/0/K, 1/0/L, 2/0/M, 3/0/N, 4/0/O
    } {
      \node at (\x+0.5,\y+0.5) {\color{lightgray}\Huge \a};
    }
    \draw[thin,lightgray] (0,0) grid (5,3);
    \draw[ultra thick] (0,0) rectangle (5,3);
    \node[opacity=0,yshift=-0.8cm] at (2,0) {\Large $n - a$};
  \end{tikzpicture}
  {\large $\cdot(a,r^2) =$}
  \begin{tikzpicture}[baseline=35]
    \begin{scope}[rotate around={180:(2.5,1.5)}]
    \foreach \x/\y/\a in {
      0/2/D, 1/2/E, 2/2/A, 3/2/B, 4/2/C, %
      0/1/I, 1/1/J, 2/1/F, 3/1/G, 4/1/H, %
      0/0/N, 1/0/O, 2/0/K, 3/0/L, 4/0/M  %
    } {
      \node[rotate around={180:(0,0)},shift={(\x+1/2,\y+1/2)},] at (0,0) {\color{lightgray}\Huge \a};
    }
    \end{scope}
    \draw[thin,lightgray] (0,0) grid (5,3);
    \draw[ultra thick, dashed] (3,0) -- (3,3);
    \draw[ultra thick] (0,0) rectangle (5,3);
    \draw[-{Triangle[length=2mm]}, ultra thick] (2,3/2) arc(0:180:0.5);
    \draw[-{Triangle[length=2mm]}, ultra thick] (4.5,3/2) arc(0:180:0.5);
    \draw[ultra thick, decorate,decoration={brace,amplitude=10pt},xshift=0pt,yshift=-0.15cm] (2.9,0)--(0.1,0);
    \draw[ultra thick, decorate,decoration={brace,amplitude=10pt},xshift=0pt,yshift=-0.15cm] (4.9,0)--(3.1,0);
    \node[yshift=-0.8cm] at (1.5,0) {\Large $n - a$};
    \node[yshift=-0.8cm] at (4,0) {\Large $a$};
  \end{tikzpicture}
  \caption{The action of $(a,r^2) \in \cylinderSymmetries$ on a tiling is
    equivalent to a $180^\circ$ rotation of both the leftmost $n-a \times m$
    sub-grid and
    the rightmost $a \times m$ sub-grid.
  }
  \label{fig:cylinderShiftAndRotate}
\end{figure}
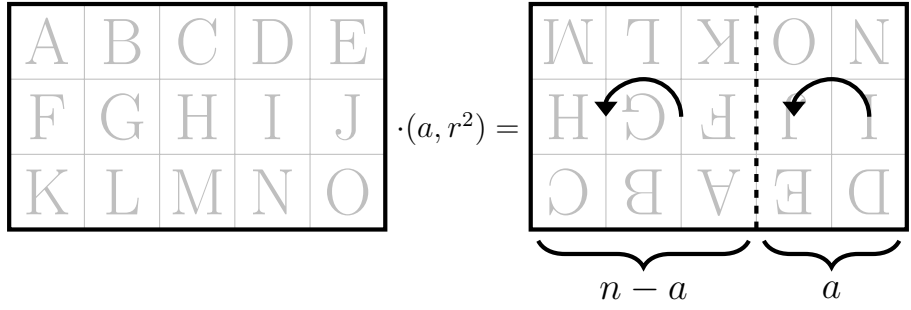

The next theorem concerns the action of $(a,f) \in \cylinderSymmetries$, which is illustrated in Figure \ref{fig:cylinderShiftAndFlip}.
\begin{theorem}
  The sum over all cyclic shifts of the number of tilings of the $n \times m$
  cylinder by tile designs in $T$ that are fixed by
  $(a,f) \in \cylinderSymmetries$
  is given by
  \begin{subnumcases}{\fxpt^\CY_{f}(n,m) =}
    \displaystyle
    n\!\left(
      \frac12 t_{\id}^{nm/2} +
      \frac12 t_{\id}^{(nm-2m)/2}t_f^{2m}
    \right)
    & $n$ is even \label{eq:fxpt_C_f_1}
    \\[10pt]
    nt_{\id}^{(nm-m)/2}t_f^m
    & $n$ is odd. \label{eq:fxpt_C_f_2}
  \end{subnumcases}
\end{theorem}
\begin{proof}
  Since $(a,f)^2 = \id$, every tile is either a fixed point or
  appears in a $2$-cycle under $(a,f)$.
  The right action of $(a,f)$ on $\xyd$ is \[
    \xyd \cdot (a,f) = \ps{-x-a-1}{y}{d \cdot f}
  \] so the tiles that appear as fixed points are those that satisfy \[
    x = - 1 - x - a \pmod n.
  \]
  \begin{description}
    \item[Equation \eqref{eq:fxpt_C_f_1}.]
    When $n$ and $a$ are both even, there are no fixed points.
    When $n$ is even and $a$ is odd, there are two fixed points in each row:
    \begin{align}
      x &\equiv (n - a - 1)/2 \pmod n \qquad \text{and}\\
      x &\equiv (2n - a - 1)/2 \pmod n.
    \end{align}
    \item[Equation \eqref{eq:fxpt_C_f_2}.]
    When $n$ is odd, there is one fixed cell in each row, which occurs when
    $x \equiv (-a-1)(n+1)/2 \pmod n$.
  \end{description}
\end{proof}

\begin{figure}[ht]
  \centering
  \begin{tikzpicture}[baseline=35]
    \foreach \x/\y/\a in {
      0/2/A, 1/2/B, 2/2/C, 3/2/D, 4/2/E,
      0/1/F, 1/1/G, 2/1/H, 3/1/I, 4/1/J,
      0/0/K, 1/0/L, 2/0/M, 3/0/N, 4/0/O
    } {
      \node at (\x+0.5,\y+0.5) {\color{lightgray}\Huge \a};
    }
    \draw[thin,lightgray] (0,0) grid (5,3);
    \draw[ultra thick] (0,0) rectangle (5,3);
    \node[opacity=0,yshift=-0.8cm] at (2,0) {\Large $n - a$};
  \end{tikzpicture}
  {\large $\cdot(a,f) =$}
  \begin{tikzpicture}[baseline=35]
    \begin{scope}[xshift=5cm, xscale=-1]
    \foreach \x/\y/\a in {
      0/2/D, 1/2/E, 2/2/A, 3/2/B, 4/2/C, %
      0/1/I, 1/1/J, 2/1/F, 3/1/G, 4/1/H, %
      0/0/N, 1/0/O, 2/0/K, 3/0/L, 4/0/M  %
    } {
      \node[xscale=-1] at (\x+0.5,\y+0.5) {\color{lightgray}\Huge \a};
    }
    \end{scope}
    \draw[thin,lightgray] (0,0) grid (5,3);
    \draw[ultra thick, dashed] (3,0) -- (3,3);
    \draw[ultra thick] (0,0) rectangle (5,3);
    \draw[{Triangle[length=3mm]}-{Triangle[length=3mm]}, ultra thick] (1,1.5) to[bend right] (2,1.5);
    \draw[{Triangle[length=3mm]}-{Triangle[length=3mm]}, ultra thick] (3.5,1.5) to[bend right] (4.5,1.5);
    \draw[ultra thick, decorate,decoration={brace,amplitude=10pt},xshift=0pt,yshift=-0.15cm] (2.9,0)--(0.1,0);
    \draw[ultra thick, decorate,decoration={brace,amplitude=10pt},xshift=0pt,yshift=-0.15cm] (4.9,0)--(3.1,0);
    \node[yshift=-0.8cm] at (1.5,0) {\Large $n - a$};
    \node[yshift=-0.8cm] at (4,0) {\Large $a$};
  \end{tikzpicture}
  \caption{The action of $(a,f) \in \cylinderSymmetries$ on a tiling is
    equivalent to a horizontal flip of both the leftmost $n-a \times m$
    sub-grid and
    the rightmost $a \times m$ sub-grid.
  }
  \label{fig:cylinderShiftAndFlip}
\end{figure}

\begin{theorem}
  The sum over all cyclic shifts of the number of tilings of the $n \times m$
  cylinder by tile designs in $T$ that are fixed by
  $(a,r^2f) \in \cylinderSymmetries$
  is given by
  \begin{subnumcases}{\fxpt^\CY_{r^2f}(n,m) =}
    \sum_{d | n} \varphi(d) t_{\id}^{nm/\lcm(d,2)}
    & $m$ is even \label{eq:fxpt_C_rrf_1}
    \\[10pt]
    \sum_{d | n} \varphi(d) t_{\id}^{(nm-n)/\lcm(d,2)}t_{(r^2f)^d}^{n/d}
    & $m$ is odd. \label{eq:fxpt_C_rrf_2}
  \end{subnumcases}
\end{theorem}
\begin{proof}

  We can see that $(a,r^2f)$ acts on the coordinates of $(x,y)$ separately,
  that is, \[
    (x,y)\cdot(a,r^2f)^k = \begin{cases}
      (x + ka, m-1-y) & k \text{ odd}\\
      (x + ka, y)     & k \text{ even}.
    \end{cases}
  \]

  Notice that the orbits of the $y$-coordinates have size $2$, so it is enough
  to determine the size of the orbits of the $x$-coordinates.
  By Lemma \ref{lem:eulerPhi}, for each divisor $d \mid n$, there are
  $\varphi(d)$ choices for $a$ such that the size of the orbits of the
  $x$-coordinate is $d$.

  \begin{description}
    \item[Equation \eqref{eq:fxpt_C_rrf_1}.]
    When $m$ is even, we see that $y \neq m - y - 1$ has no solutions,
    the orbit of every $y$-coordinate has size $2$.
    For each $x$-orbit size $d \mid n$,
    each cell must be in an orbit of size $\lcm(d,2)$, and there must
    $nm/\lcm(d,2)$ of them. Each can be specified by any tile design, since all
    tile designs are stable under $(r^2f)^{\lcm(d,2)} = \id$.
    \item[Equation \eqref{eq:fxpt_C_rrf_2}.]
    When $m$ is odd, we see that $y = m - y - 1$ has a solutions precisely when
    $y = (m-1)/2$.
    For each $x$-orbit size $d \mid n$, if $y = (m-1)/2$, then the orbit has
    size $d$, otherwise it has size $\lcm(d,2)$, as in the case above.
    Therefore there are $n$ cells that are in orbits of size $d$ resulting
    in $n/d$ orbits that can be specified by any tile design that is stable
    under $(r^2f)^{\lcm(d,2)}$. The remaining $n^2-n$ cells are partitioned
    into $(n^2-n)/\lcm(d,2)$ orbits of size $\lcm(d,2)$ that can be specified
    by any tile design.
  \end{description}
\end{proof}

\begin{theorem}
  For
  a given set of tile designs $T$,
  a symmetry group $R \leq D_4$, and
  an element $g \in R$,
  the number of distinct tilings of the $n \times m$ cylinder is
  \begin{equation}
    \frac{1}{n|R|}\sum_{g \in R} \fxpt^\CY_g(n,m).
    \label{eq:distinctTilingsC}
  \end{equation}
\end{theorem}
\begin{proof}
  We will use the convention that
  when we index over $g$, implicitly $g \in R$;
  when we index over $a$, implicitly $a \in \Z/n\Z$; and
  when we index over $(a,g)$, implicitly $(a,g) \in \cylinderSymmetries$.

  Since $\fxpt^\CY_g(n,m) = \displaystyle\sum_a X^{(a,g)}$, we can see that
  \begin{align*}
    \frac{1}{n|R|}\sum_{g} \fxpt^\CY_g(n,m)
    &= \frac{1}{n|R|}\sum_{g} \sum_{a} X^{(a,g)} \\
    &= \frac{1}{|\cylinderSymmetries|}\sum_{(a,g)} X^{(a,g)},
  \end{align*}
  which counts the number of distinct tilings by a direct application of
  Burnside's lemma.
\end{proof}

\section{Torus}
\label{sec:torus}
This section builds on the work of
Ethier \cite{Ethier} and
Ethier and Lee\cite{EthierLee}. By
specializing to the set of tile designs
$T = \left\{\tileColor[black!80],\tileColor[black!20]\right\}$, we recover
their work.
Irvine \cite{Irvine} generalized this in the specific context of a set
of tile designs of size $n$ in the specific case that no rotation or reflection
is allowed (only cyclic shifting of rows and columns), which we recover in
Theorem \ref{thm:fxpt_Rt_id}.

We distinguish between the rectangular torus, which we denoted by \RT, and
the square torus, which we denote by \ST.
\subsection{The \texorpdfstring{$n \times m$}{n by m} torus}
\begin{definition}
  For a given set of tile designs $T$, and an element
  \(g \in R \leq D_4\)
  the sum over all cyclic shifts of the number of tilings of the $n \times m$
  cylinder by tile designs in $T$ that are fixed by
  $\abg \in (\Z/n \Z \times \Z/m \Z) \rtimes R$
  is denoted \[
    \fxpt^\RT_g(n,m) = \sum_{(a,b) \in \Z/n\Z\times\Z/m\Z} X^{\abg},
  \] where $X^{\abg}$ is the number of tilings fixed by $\abg$.
\end{definition}
\begin{theorem}
  The sum over all cyclic shifts of the number of tilings of the $n \times m$
  torus by tile designs in $T$ that are fixed by
  $(a,\id) \in \torusSymmetries$
  is given by
  \begin{equation}
  \fxpt^\RT_{\id}(n,m) = \sum_{c \mid m}\sum_{d \mid n}\varphi(c)\varphi(d) t_{\id}^{mn/\lcm(c,d)}.
  \label{eq:fxpt_TR_id}
  \end{equation}
  \label{thm:fxpt_Rt_id}
\end{theorem}
\begin{proof}
  The size of an orbit of a tile under $\abg^k$ is the set of solutions to
  \begin{align*}
    ka &\equiv 0 \pmod n \\
    kb &\equiv 0 \pmod m.
  \end{align*}

  For each individual equation, the minimal choice for $k$ must be a divisor of
  $n$.
  For a given divisor $d \mid n$, there are $\varphi(d)$ choices for $a \in \Z/n\Z$ so
  that $da \equiv 0 \pmod n$. Namely if $i$ is coprime to $d$, then
  $a = i\left(n/d\right)$ will be a minimal solution.
  An analogous argument holds for the second equation.

  Therefore if $d \mid n$ and $c \mid m$, there are $\varphi(d)\varphi(c)$ pairs
  $(a,b) \in \grid$ where $a$ has order $d$ and $b$ has order
  $c$, and thus $(a,b)$ has order $\lcm(d,c)$. Therefore, each orbit of cells
  has size $\lcm(d,c)$, and so the number of orbits of cells is $nm/\lcm(d,c)$.

  Thus, the sum of the number of orbits over each pair
  $(a,b) \in \grid$ gives equation \eqref{eq:fxpt_TR_id}, as desired.
\end{proof}

The next theorem concerns the action of $\ps ab{r^2} \in \torusSymmetries$, which is illustrated in Figure \ref{fig:torusShiftAndRotate180}.

\begin{theorem}
  The sum over all cyclic shifts of the number of tilings of the $n \times m$
  torus by tile designs in $T$ that are fixed by
  $(a,r^2) \in \torusSymmetries$
  is given by
  \begin{subnumcases}{\fxpt^\RT_{r^2}(n,m) =}
    \displaystyle
    nm\left(\frac34t_{\id}^{nm/2} + \frac14t_{\id}^{nm/2-2}t_{r^2}^4\right)
    & $n$ and $m$ even \label{eq:fxpt_TR_rr_a}
    \\[10pt]
    \displaystyle
    nm t_{\id}^{(nm-1)/2}t_{r^2}
    & $n$ and $m$ odd \label{eq:fxpt_TR_rr_b}
    \\[10pt]
    \displaystyle \label{eq:fxpt_TR_rr_c}
    nm\left(\frac12 t_{\id}^{nm/2} + \frac12 t_{\id}^{nm/2-1}t_{r^2}^2\right)
    & otherwise.
  \end{subnumcases}
  \label{thm:fxpt_Rt_rr}
\end{theorem}
\begin{proof}[Proof of Theorem \ref{thm:fxpt_Rt_rr}]
  The orbits of cells under the group generated by $\ps ab{r^2}$ have size
  either $1$ or $2$, because by how we defined the semidirect product,
  $\ps ab{r^2}$ has order $2$: \begin{equation}
    \ps ab{r^2}^2
    = \ps ab{r^2}\ps ab{r^2}
    = \left((a,b)+(-a,-b),r^2r^2\right)
    = \ps00{\id}.
    \label{eq:abrr_order_2}
  \end{equation}
  Therefore, it is enough to count how many cells are stable under
  $\ps ab{r^2}$, which depends on the parity of $n$, $m$, $a$, and $b$.

  The element $\ps ab{r^2}$ fixes a cell $(x,y)$ when \[
    (x,y)\cdot\ps ab{r^2} = (n-1-(x+a),n-1-(y+b)).
  \]
  This corresponds to the system of equations \begin{align}
    x \equiv - 1 - x - a &\pmod n \label{eq:fxpt_TR_rr_d}
    \\
    y \equiv - 1 - y - b &\pmod m, \label{eq:fxpt_TR_rr_e}
  \end{align} whose solutions are given by Lemma \ref{lem:flip_solution}.

  Therefore we proceed by each case
  \begin{description}
    \item[Equation \eqref{eq:fxpt_TR_rr_a}.]
    When $n$ and $m$ are even, there are fixed cells only when both $a$ and $b$
    are odd, by Lemma \ref{lem:flip_solution}; in this case, there are exactly
    $4$ fixed cells, because each equation in the system of equations has two
    solutions. When this occurs, it partitions the cells into
    $4$ orbits of size $1$,
    which can be filled with tile designs that are fixed under $r^2$, and
    $(nm-4)/2$ orbits of size $2$,
    which can be filled with any tile design.

    When either $a$ or $b$ is even, there are no fixed cells, which partitions
    the cells into $nm/2$ orbits of size $2$, each of which can be filled with
    any choice of tile design.

    Since the $4$ fixed cells occur for exactly one quarter of the pairs
    $(a,b) \in \grid$, this results in the desired equation.
    \item[Equation \eqref{eq:fxpt_TR_rr_b}.]
    When $n$ and $m$ are both odd, Lemma \ref{lem:flip_solution} states that
    equations \eqref{eq:fxpt_TR_rr_d} and \eqref{eq:fxpt_TR_rr_e} have one
    solution.

    When this occurs, it partitions the cells into $1$ orbit of size $1$, which
    can be filled only with a tile design that is fixed by $r^2$,
    and
    $(nm-1)/2$ orbits of size $2$, which can be specified with any tile design.
    \item[Equation \eqref{eq:fxpt_TR_rr_c}.]
    Without loss of generality, we can assume that $n$ is even and $m$ is odd,
    because the proof is essentially similar in the opposite case.
    Lemma \ref{lem:flip_solution} states
    that equation \eqref{eq:fxpt_TR_rr_d} has no solutions when $a$ is even and
    $2$ solutions when $a$ is odd;
    it also states that \eqref{eq:fxpt_TR_rr_e} has one solution.

    Thus for half of the pairs $(a,b)$, there are no fixed cells, and so there
    are $nm/2$ orbits, each of which can be specified by any tile design.

    For the other half of the pairs, there are $2$ fixed cells, which can be
    specified by any tile design that is fixed under $r^2$ and $(nm-2)/2$ orbits
    of size $2$ that can be specified by any tile design.
  \end{description}
\end{proof}

\begin{figure}[ht]
  \centering
  \begin{tikzpicture}[baseline=35]
    \foreach \x/\y/\a in {
      0/2/A, 1/2/B, 2/2/C, 3/2/D, 4/2/E,
      0/1/F, 1/1/G, 2/1/H, 3/1/I, 4/1/J,
      0/0/K, 1/0/L, 2/0/M, 3/0/N, 4/0/O
    } {
      \node at (\x+0.5,\y+0.5) {\color{lightgray}\Huge \a};
    }
    \draw[thin,lightgray] (0,0) grid (5,3);
    \draw[ultra thick] (0,0) rectangle (5,3);
    \node[opacity=0,yshift=-0.8cm] at (2,0) {\Large $n - a$};
  \end{tikzpicture}
  {\large $\ps ab{r^2} =$}
  \begin{tikzpicture}[baseline=35]
    \begin{scope}[rotate around={180:(2.5,1.5)}]
    \foreach \x/\y/\a in {
      0/2/N, 1/2/O, 2/2/K, 3/2/L, 4/2/M,  %
      0/1/D, 1/1/E, 2/1/A, 3/1/B, 4/1/C, %
      0/0/I, 1/0/J, 2/0/F, 3/0/G, 4/0/H %
    } {
      \node[rotate around={180:(0,0)},shift={(\x+1/2,\y+1/2)},] at (0,0) {\color{lightgray}\Huge \a};
    }
    \end{scope}
    \draw[thin,lightgray] (0,0) grid (5,3);
    \draw[ultra thick, dashed] (3,0) -- (3,3);
    \draw[ultra thick, dashed] (0,1) -- (5,1);
    \draw[ultra thick] (0,0) rectangle (5,3);
    \draw[-{Triangle[length=2mm]}, ultra thick] (2.0,5/3) arc(0:180:0.5);
    \draw[-{Triangle[length=2mm]}, ultra thick] (4.5,5/3) arc(0:180:0.5);
    \draw[-{Triangle[length=2mm]}, ultra thick] (2.0,1/3) arc(0:180:0.5);
    \draw[-{Triangle[length=2mm]}, ultra thick] (4.5,1/3) arc(0:180:0.5);
    \draw[ultra thick, decorate,decoration={brace,amplitude=10pt},xshift=0pt,yshift=-0.15cm] (2.9,0)--(0.1,0);
    \draw[ultra thick, decorate,decoration={brace,amplitude=10pt},xshift=0pt,yshift=-0.15cm] (4.9,0)--(3.1,0);
    \draw[ultra thick, decorate,decoration={brace,amplitude=5pt},xshift=0.15cm,yshift=0] (5.0,2.9)--(5.0,1.1);
    \draw[ultra thick, decorate,decoration={brace,amplitude=5pt},xshift=0.15cm,yshift=0] (5.0,0.9)--(5.0,0.1);
    \node[yshift=-0.8cm] at (1.5,0) {\Large $n - a$};
    \node[yshift=-0.8cm] at (4,0) {\Large $a$};
    \node[xshift=1.2cm] at (5,0.5) {\Large $m - b$};
    \node[xshift=0.7cm] at (5,2.0) {\Large $b$};
  \end{tikzpicture}
  \caption{The action of $((a,b),r^2) \in \torusSymmetries$ on a tiling is
    equivalent to a $180^\circ$ rotations of
    the lower left  $(n-a) \times (m-b)$ sub-grid,
    the lower right $    a \times (m-b)$ sub-grid,
    the upper left  $(n-a) \times b$     sub-grid, and
    the upper right $    a \times b$     sub-grid.
  }
  \label{fig:torusShiftAndRotate180}
\end{figure}

\begin{theorem}
  The sum over all cyclic shifts of the number of tilings of the $n \times m$
  torus by tile designs in $T$ that are fixed by
  $(a,f) \in \torusSymmetries$
  is given by
  \begin{subnumcases}{\fxpt^\RT_{f}(n,m) =}
      \displaystyle
      n \sum_{c \mid m} \varphi(c) \left(
        \frac12 t_{\id}^{nm/\lcm(2,c)} +
        \frac12 t_{\id}^{(n-2)m/\lcm(2,c)}t_{f^c}^{2m/c}
      \right)
      & $n$ even \label{eq:fxpt_Rt_f_even}
      \\[10pt]
      \displaystyle
      n \sum_{c \mid m} \varphi(c) t_{\id}^{(n-1)m/\lcm(2,c)}t_{f^c}^{m/c}
      & $n$ odd. \label{eq:fxpt_Rt_f_odd}
  \end{subnumcases}
  and
  \begin{equation*}
    \fxpt^\RT_{r^2f}(n,m) = \begin{cases}
      \displaystyle
      m \sum_{d \mid n} \varphi(d) \left(
        \frac12 t_{\id}^{nm/\lcm(d,2)} +
        \frac12 t_{\id}^{n(m-2)/\lcm(d,2)}t_{(r^2f)^d}^{2n/d}
      \right)                                                      & m \text{ even} \\[10pt]
      \displaystyle
      m \sum_{d \mid n} \varphi(d) t_{\id}^{n(m-1)/\lcm(d,2)}t_{(r^2f)^d}^{n/d}  & m \text{ odd.}
    \end{cases}
    \label{eq:fxpt_TR_rrf}
  \end{equation*}
  \label{thm:fxpt_Rt_f}
\end{theorem}
\begin{proof}
  Since $f$ and $r^2f$ are conjugate as elements of $D_8$, these fixed point
  formulas have essentially the same proof, so we will prove
  equations \eqref{eq:fxpt_Rt_f_even} and \eqref{eq:fxpt_Rt_f_odd} specifically.

  Since $(x,y) \cdot \ps abf = (n-1-(x+a), y+b)$, we can view
  $\ps abf$ as acting on each coordinate separately. Since
  $\ps abf^2 = ((0,2b),\id)$, we can see that the orbits
  \textit{of the first coordinate} have either size $1$ or $2$.
  Moreover, by Lemma \ref{lem:flip_solution},
  there are no  fixed cells when $n$ is even and $a$ is even,
  there are $2$ fixed cells when $n$ is even and $a$ is odd, and
  there is $1$  fixed cell when $n$ is odd.

  Since $\ps abf$ acts on the second coordinate by shifting by $b$, we see
  that Lemma \ref{lem:eulerPhi} applies. Thus for each divisor $c \mid m$,
  there are $\varphi(c)$ choices of $b$ that produce orbits
  \textit{of the second coordinate} with size $c$.

  \begin{description}
    \item[Equation \eqref{eq:fxpt_Rt_f_even}.]
    When $n$ is even, then half of the values of $a \in \Z/n\Z$ are even, and each
    orbit has size $\lcm(2,c)$ and these can be specified by any tile design.

    The other half of values of $a$ are odd, which results in $2$ fixed points for
    the first coordinate, each of which results in an orbit of size $c$
    that can be specified by any tile design that is fixed by $f^c$.
    The remaining $(n-2)m$ cells then are partitioned into orbits of size
    $\lcm(2,c)$, which can be specified by any tile design.
    \item[Equation \eqref{eq:fxpt_Rt_f_odd}.]
    When $n$ is odd, then there is one fixed point for the first coordinate,
    resulting in $m$ cells where the first coordinate is fixed under the action
    of $\ps abf$. For each divisor $c \mid m$, there is a partition these $m$
    cells into orbits of size $c$. The resulting $m/c$ orbits
    can be specified by any tile design that fixes $f^c$.
    The remaining $n(m-1)$ cells are partitioned into orbits of size
    $\lcm(2,c)$, resulting in $n(m-1)/\lcm(2,c)$ orbits
    which can be specified by any tile design.
  \end{description}
\end{proof}
\begin{theorem}
  Then the number of distinct tilings of the $n \times m$ torus up to
  $R \subseteq D_4$ is given by
  \begin{equation}
    \frac{1}{nm|R|}\sum_{g \in R} \fxpt^\RT_g(n,m).
    \label{eq:distinctTilingsTR}
  \end{equation}
  \label{thm:nXmTorusFormula}
\end{theorem}
\begin{proof}
  We will use the convention that
  when we index over $g$, implicitly $g \in R$;
  when we index over $(a,b)$, implicitly $(a,b) \in \Z/n\Z\times\Z/m\Z$; and
  when we index over $\abg$, implicitly $\abg \in \torusSymmetries$.

  Since $\fxpt^\RT_g(n,m) = \displaystyle\sum_{(a,b)} X^{\abg}$, we can see that
  \begin{align*}
    \frac{1}{nm|R|}\sum_{g} \fxpt^\RT_g(n,m)
    &= \frac{1}{nm|R|}\sum_{g} \sum_{(a,b)} X^{\abg} \\
    &= \frac{1}{|\torusSymmetries|}\sum_{\abg} X^{\abg},
  \end{align*}
  which counts the number of distinct tilings by a direct application of
  Burnside's lemma.
\end{proof}

\subsection{The \texorpdfstring{$n \times n$}{n by n} torus}
\begin{theorem}
  The sum over all cyclic shifts of the number of tilings of the $n \times n$
  square torus (denoted \ST) by tile designs in $T$ that are fixed by
  $\abg \in \torusSymmetries$ is given by \(\fxpt^\ST_g(n)\) where
  \begin{align}
    \fxpt^\ST_{\id}(n)    &= \fxpt^\RT_{\id}(n,n) = \sum_{d_1|n}\sum_{d_2|n}\varphi(d_1)\varphi(d_2)t_{\id}^{n^2/\lcm(d_1,d_2)},
    \\
    \fxpt^\ST_{r^2}(n)  &= \fxpt^\RT_{r^2}(n,n) = \begin{cases}
      n^2t_{\id}^{(n^2-1)/2}t_{r^2} & n \text{ odd} \\
    \displaystyle
    n^2\!\left(
      \frac34 t_{\id}^{n^2/2} +
      \frac14 t_{\id}^{n^2/2-2}t_{r^2}^4
    \right) & n \text{ even.}
    \end{cases}
    \\
    \fxpt^\ST_{f}(n)    &= \fxpt^\RT_{f}(n,n) = \begin{cases}
      \displaystyle
      n \sum_{d \mid n} \varphi(d) \left(
        \frac12 t_{\id}^{n^2/\lcm(2,d)} +
        \frac12 t_{\id}^{(n^2-2n)/\lcm(2,d)}t_{f^d}^{2n/d}
      \right)                                                      & n \text{ even} \\[10pt]
      \displaystyle
      n \sum_{d \mid n} \varphi(d) t_{\id}^{(n^2-n)/\lcm(2,d)}t_{f^d}^{n/d}  & n \text{ odd.}
    \end{cases}
    \\
    \fxpt^\ST_{r^2f}(n) &= \fxpt^\RT_{r^2f}(n,n) = \begin{cases}
      \displaystyle
      n \sum_{d \mid n} \varphi(d) \left(
        \frac12 t_{\id}^{n^2/\lcm(2,d)} +
        \frac12 t_{\id}^{(n^2-2n)/\lcm(2,d)}t_{(r^2f)^d}^{2n/d}
      \right)                                                      & n \text{ even} \\[10pt]
      \displaystyle
      n \sum_{d \mid n} \varphi(d) t_{\id}^{(n^2-n)/\lcm(2,d)}t_{(r^2f)^d}^{n/d}  & n \text{ odd.}
    \end{cases}
  \end{align}
\end{theorem}
\begin{proof}
  These equations follow directly from Theorems \ref{thm:fxpt_Rt_id},
  \ref{thm:fxpt_Rt_rr} and \ref{thm:fxpt_Rt_f} by specifying $m = n$.
\end{proof}
\begin{theorem}
  The sum over all cyclic shifts of the number of tilings of the $n \times n$
  torus by tile designs in $T$ that are fixed by
  $(a,r) \in \torusSymmetries$
  is given by
  \begin{subnumcases}{\fxpt^\ST_{r}(n) = \fxpt^\ST_{r^3}(n) =}
    n^2 t_{\id}^{(n^2-1)/4}t_r
    & $n$ odd \label{eq:fxpt_St_r_odd}
    \\
    \displaystyle
    n^2\!\left(
      \frac12 t_{\id}^{n^2/4} +
      \frac12 t_{\id}^{(n^2-4)/4}t_r^2t_{r^2}
    \right)
    & $n$ even. \label{eq:fxpt_St_r_even}
  \end{subnumcases}
\end{theorem}
\begin{proof}
  First, note that the first equality comes from the fact that tilings that are
  stable under $g$ are stable under $g^{-1}$.

  Next, note that $\ps abr$ is an element of order $4$, which follows from
  observing that $\ps abr = ((a',b'),r^2)$ together with
  equation \eqref{eq:abrr_order_2}.
  Therefore cells appear in orbits of size $1$, $2$, or $4$ under $\ps abr$,
  and we will count how many cells appear in each.

  We begin by counting cells $(x,y)$ that are fixed by $\ps abr$, that is
  they satisfy the system of equations
  \begin{alignat}{2}
    x &\equiv -y - b - 1 &\pmod n \label{eq:fxpt1} \\
    y &\equiv x + a      &\pmod n \label{eq:fxpt2},
  \end{alignat}
  where we can substitute $y$ with $x+a$ in the first equation to get
  \begin{equation}
    x \equiv -x - a - b - 1 \pmod n.
    \label{eq:fxpt3}
  \end{equation}

  Next we count cells $(x,y)$ that are fixed by $\ps abr^2$, but are not
  solutions to the above system of equations. These cells satisfy the system of
  equations \begin{alignat}{2}
    x &\equiv -1 - x - a - b &\pmod n \label{eq:twoCycle1} \\
    y &\equiv -1 - y - b + a &\pmod n \label{eq:twoCycle2}.
  \end{alignat}
  \begin{description}
    \item[Equation \eqref{eq:fxpt_St_r_odd}.]
    When $n$ is odd, we can solve equation \eqref{eq:fxpt3} using
    Lemma \ref{lem:flip_solution}.
    We can see that this has one solution when $n$ is odd,
    so in this case there is one fixed cell, which can be
    specified by any tile design that is stable under $r$.

    We can add equations \eqref{eq:twoCycle1} and \eqref{eq:twoCycle2} and use
    Lemma \ref{lem:flip_solution} to see that this system has a single solution
    when $n$ is odd. However, this is identically the solution that specifies
    the fixed point, so this does not describe an orbit of size $2$.

    Thus there are $n^2 - 1$ cells that are partitioned into $(n^2-1)/4$ orbits
    of size $4$, which can be specified by any tile design.

    \item[Equation \eqref{eq:fxpt_St_r_even}] When $n$ is even, we can see again
    by Lemma \ref{lem:flip_solution}, that there are two fixed cells when
    $a + b$ is odd and none when $a + b$ is even.

    When we check the number of orbits of size $2$, Lemma
    \ref{lem:flip_solution} shows that we have $4$ solutions when $a + b$ is odd
    and none when $a + b$ is even, $2$ of which were the fixed cells, resulting
    in a single orbit of size $2$.

    Therefore, we can specify a fixed tiling by
    specifying tile designs that are fixed under $r$ for each of the two fixed cells,
    specifying a tile design that is fixed under $r^2$ for the orbit of size $2$, and
    specifying any tile designs for each of the $(n^2 - 4)/4$ orbits of length $4$.
  \end{description}
\end{proof}

\begin{theorem}
  \begin{subnumcases}{\fxpt^\ST_{rf}(n) = n\sum_{d \mid n}}
    \varphi(d) t_{\id}^{(n^2-n)/(2d)}t_{rf}^{n/d}
    & $d$ odd \label{eq:fxpt_St_rf_odd}
    \\
    \varphi(d) t_{\id}^{n^2/(2d)}
    & $d$ even. \label{eq:fxpt_St_rf_even}
  \end{subnumcases}
  and
  \begin{align}
    \fxpt^\ST_{r^3f}(n)   &= n\sum_{d \mid n}\begin{cases}
      \varphi(d) t_{\id}^{(n^2-n)/(2d)}t_{r^3f}^{n/d} & d \text{ odd} \\
      \varphi(d) t_{\id}^{n^2/(2d)}                   & d \text{ even}.
    \end{cases}
  \end{align}
\end{theorem}
\begin{proof}
  These situations are essentially the same because $rf$ and $r^3f$ are
  conjugate in $D_8$, so we prove the case for $\fxpt^\ST_{rf}(n)$.

  We can see that \begin{align*}
    (x,y) &\cdot \ps ab{rf}^{2k}   = (x+k(a+b), y+k(a+b)) \\
    (x,y) &\cdot \ps ab{rf}^{2k+1} = (y+k(a+b)+b, x+k(a+b)+a)
  \end{align*} and we can ask: what is the least $k$ such that either
  \begin{align}
    x &\equiv x + k(a+b) \pmod n \label{syseq:fxpt_St_rf_even_1}
    \\
    y &\equiv y + k(a+b) \pmod n \label{syseq:fxpt_St_rf_even_2}
  \end{align} or
  \begin{align}
    x &\equiv y + k(a+b) + b \pmod n \label{syseq:fxpt_St_rf_odd_1}
    \\
    y &\equiv x + k(a+b) + a \pmod n. \label{syseq:fxpt_St_rf_odd_2}
  \end{align}
  In the first case, we want to know when $k(a+b) \equiv 0 \pmod n$, which
  occurs first when $k = n/\gcd(a+b,n)$. We call this $d$ and note that
  $d \mid n$.
  In the second case, we can add equations \eqref{syseq:fxpt_St_rf_odd_1} and
  \eqref{syseq:fxpt_St_rf_odd_2} to get \begin{equation}
    (2k + 1)(a + b) = 0 \pmod n,
  \end{equation} which occurs when $2k+1 = n/\gcd(a+b,n)$.
  Again, this is a divisor of $n$, so we say $2k+1 = d$
  and note that this solution occurs only when $d$ is odd.

  \begin{description}
    \item[Equation \eqref{eq:fxpt_St_rf_odd}] Thus, when $d$ is odd,
    we have the system of equations \begin{align*}
      x &\equiv y + \frac{d-1}{2}(a+b) + b \pmod n \\
      y &\equiv x + \frac{d-1}{2}(a+b) + a \pmod n,
    \end{align*} which has $n$ solutions:
    for each choice of $x$, there is a unique choice of $y$ that satisfies both
    equations.
    Each of these solutions correspond to a cell in one of the $n/d$ orbits
    of length $d$, each of which can be specified by any tile design
    that is stabilized by $(rf)^d = rf$, since $d$ is odd.

    The other $n^2 - n$ cells occur in one of the $(n^2-n)/(2d)$ orbits of length
    $2d$ that are solutions to the first system of equations. Each of these
    orbits can be specified by any tile design at all.

    The $\varphi(d)$ comes from the fact that for any choice of $a$ there are
    precisely $\varphi(d)$ choices for $b$ such that $n/\gcd(a+b,n) = d$.
    \item[Equation \eqref{eq:fxpt_St_rf_even}] When $d$ is even, there are no
    choices of $(x,y)$ that simultaneously satisfy equations
    \eqref{syseq:fxpt_St_rf_odd_1} and \eqref{syseq:fxpt_St_rf_odd_2},
    so all $n^2$ of the tiles $(x,y)$ occur in orbits of size $2d$.
    Each of these $n^2/(2d)$ orbits can be specified by any tile design.
  \end{description}
\end{proof}

\begin{theorem}
  The number of distinct tilings of the $n \times n$ torus up to
  $R \leq D_8$ is given by
  \begin{equation}
    \frac{1}{n^2|R|}\sum_{g \in R} \fxpt^\ST_g(n).
  \end{equation}
  \label{thm:nXnTorusFormula}
\end{theorem}
\begin{proof}
  The proof of this theorem is essentially identical to the proof of
  Theorem \ref{thm:nXmTorusFormula}, which follows by definition together with
  Burnside's lemma.
\end{proof}

Thus for any arbitrary $R \subseteq D_8$ and set of tile designs, we have
a formula to count the number of tilings of the $n \times n$ torus up to $R$.
A formula for each choice of $R$ together with each $R$-set generated by a
single tile design can be found in Appendix \ref{apss:nXnTorusSequences};
the corresponding illustrations can be found in
Appendix \ref{apss:nXnTorusIllustrations}.

\section{Next steps}
In this section, we propose several different settings for studying similar
kinds of problems. Many of these may be subtle research problems,
many may be good undergraduate research problems, and many may be good homework
problems for a combinatorics class.
Many of them would make for interesting additions to the On-Line Encyclopedia
of Integer Sequences.
\subsection{Rectangular tori under \texorpdfstring{$90^\circ$}{90 degree} rotation}
We have used the $n \times m$ torus as a model for a repeating tiling of the
plane. However, in the case that $n \neq m$, we have only analyzed the case
where we count tilings up to $D_4$, the dihedral group of the rectangle.
However, for a given tile set, it is possible that a tiling of a
$n \times m$ and a tiling of a $m \times n$ torus describe the same
tiling of the plane; an example of this is given in Figure
\ref{fig:rectangularTorusRotation}.

\begin{conjecture}
  If a plane tiling described by an $n \times m$ torus is fixed under
  $\ps abr$,
  $\ps ab{r^3}$,
  $\ps ab{rf}$, or
  $\ps ab{r^3f}$,
  then it is equivalent to the tiling of a
  $\gcd(n,m) \times \gcd(n,m)$ torus.
\end{conjecture}

Similarly, it might be interesting to count \textit{irreducible} plane tilings:
tilings of the plane corresponding to a tiling of the $n \times m$ torus
that don't correspond to a smaller torus.

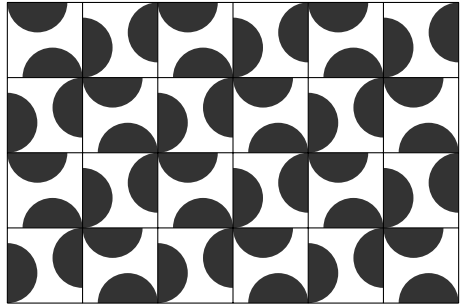
\begin{figure}[ht]
  \centering
  \begin{tikzpicture}
    \newcommand{\tile}[3]{
      \fill[black!80, xshift=#1 cm,yshift=#2 cm,rotate around={#3:(0.5,0.5)}] (0,0) arc(-90:90:1/2.507) (1,1) arc(90:270:1/2.507)
    }
    \foreach \x in {0,...,5} {
      \foreach \y in {0,...,3} {
        \tile{\x}{\y}{270*mod(\x,2) + 90*mod(\y,2) - 180*mod(\x*\y,2)};
      }
    }
    \draw (0,0) grid (6,4);
  \end{tikzpicture}
  \caption{
    A periodic tiling of the plane arising from
    a $6 \times 4$ torus tiling that is fixed under $90^\circ$ rotation.
    Notice that this tiling of the plane can also come from a $2 \times 2$
    torus.
  }
  \label{fig:rectangularTorusRotation}
\end{figure}


\subsection{Other regions of the square grid.}
We also are interested in counting the number of ways of tiling shapes like
Aztec diamonds or centered square numbers, as shown in
Figure \ref{fig:centeredSquare}.

\begin{figure}[ht]
  \centering
  \begin{tikzpicture}
    \pgfmathsetmacro{\n}{1}
    \foreach \y in {-\n,...,\n} {
      \pgfmathtruncatemacro{\startValue}{\n - abs(\y)}
      \pgfmathtruncatemacro{\endValue}{-\n + abs(\y)}
      \foreach \x in {\startValue,...,\endValue} {
        \pgfmathtruncatemacro{\r}{rand*4}
        \fill[gray, shift={(\x,\y)},rotate around={\r*90:(1/2,1/2)}] (0,0) -- (0,1) -- (1,1) -- cycle;
        \draw[shift={(\x,\y)}] (0,0) rectangle (1,1);
      }
    }
  \end{tikzpicture}
  \qquad
  \begin{tikzpicture}[scale=3/5]
    \pgfmathsetmacro{\n}{2}
    \foreach \y in {-\n,...,\n} {
      \pgfmathtruncatemacro{\startValue}{\n - abs(\y)}
      \pgfmathtruncatemacro{\endValue}{-\n + abs(\y)}
      \foreach \x in {\startValue,...,\endValue} {
        \pgfmathtruncatemacro{\r}{rand*4}
        \fill[gray, shift={(\x,\y)},rotate around={\r*90:(1/2,1/2)}] (0,0) -- (0,1) -- (1,1) -- cycle;
        \draw[shift={(\x,\y)}] (0,0) rectangle (1,1);
      }
    }
  \end{tikzpicture}
  \qquad
  \begin{tikzpicture}[scale=3/7]
    \pgfmathsetmacro{\n}{3}
    \foreach \y in {-\n,...,\n} {
      \pgfmathtruncatemacro{\startValue}{\n - abs(\y)}
      \pgfmathtruncatemacro{\endValue}{-\n + abs(\y)}
      \foreach \x in {\startValue,...,\endValue} {
        \pgfmathtruncatemacro{\r}{rand*4}
        \fill[gray, shift={(\x,\y)},rotate around={\r*90:(1/2,1/2)}] (0,0) -- (0,1) -- (1,1) -- cycle;
        \draw[shift={(\x,\y)}] (0,0) rectangle (1,1);
      }
    }
  \end{tikzpicture}
  \qquad
  \begin{tikzpicture}[scale=3/9]
    \pgfmathsetmacro{\n}{4}
    \foreach \y in {-\n,...,\n} {
      \pgfmathtruncatemacro{\startValue}{\n - abs(\y)}
      \pgfmathtruncatemacro{\endValue}{-\n + abs(\y)}
      \foreach \x in {\startValue,...,\endValue} {
        \pgfmathtruncatemacro{\r}{rand*4}
        \fill[gray, shift={(\x,\y)},rotate around={\r*90:(1/2,1/2)}] (0,0) -- (0,1) -- (1,1) -- cycle;
        \draw[shift={(\x,\y)}] (0,0) rectangle (1,1);
      }
    }
  \end{tikzpicture}
  \caption{Order $1$, $2$, $3$, and $4$ centered square figures in the square
  tiling of the plane.}
  \label{fig:centeredSquare}
\end{figure}
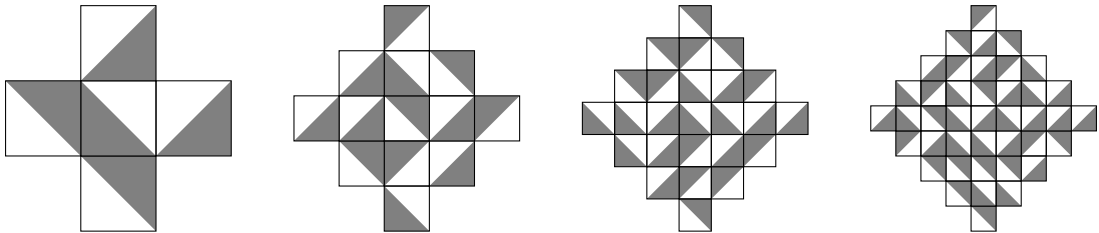

\subsection{The M\"obius strip and Klein bottle}
Since we've looked at the orientable identifications of the grid, we are also
interested in the non-orientable identifications. The M\"obius strip has a
universal cover that is $[0,1] \times \R$, and the Klein bottle has a universal
cover of $\R \times \R$, so we can visualize them analogously to how we
visualized the cylinder and torus respectively.
An illustration of a tiling of the M\"obius strip in Figure \ref{fig:mobiusStrip}.
An illustration of a tiling of the Klein bottle in Figure \ref{fig:kleinBottle}.

We are also interested counting tilings of the real projective plane, but
because the universal cover is not the Euclidean plane, it cannot be
illustrated in the same manner as the Klein bottle and the torus.

\todo[inline]{Let's include examples for each of these two situations,
together with brute-forced numbers.}

\begin{figure}[ht]
  \centering
  \begin{subfigure}[b]{0.56\textwidth}
    \centering
    \begin{tikzpicture}[scale=1]
      \fill[gray] (0+1,0+1)--(0,0+1)--(0+1,0);
      \fill[gray] (2+1,1)  --(2+1,1+1)--(2,1);
      \fill[gray] (4+1,0+1)--(4,0+1)--(4+1,0);
      \fill[gray] (6+1,1)  --(6+1,1+1)--(6,1);
      \fill[gray] (0+1,1+1)--(0,1+1)--(0+1,1);
      \fill[gray] (2+1,0)  --(2+1,0+1)--(2,0);
      \fill[gray] (4+1,1+1)--(4,1+1)--(4+1,1);
      \fill[gray] (6+1,0)  --(6+1,0+1)--(6,0);
      \fill[gray] (1,0)    --(1+1,0)--(1,0+1);
      \fill[gray] (3,1+1)  --(3,1)--(3+1,1+1);
      \fill[gray] (5,0)    --(5+1,0)--(5,0+1);
      \fill[gray] (7,1+1)  --(7,1)--(7+1,1+1);
      \fill[gray] (1+1,1)  --(1+1,1+1)--(1,1);
      \fill[gray] (3+1,0+1)--(3,0+1)--(3+1,0);
      \fill[gray] (5+1,1)  --(5+1,1+1)--(5,1);
      \fill[gray] (7+1,0+1)--(7,0+1)--(7+1,0);
      \draw (0,0) grid (8,2);
      \draw[red, ultra thick] (0,0) rectangle (2,2);
      \draw[white, ultra thick] (3,0) rectangle (5,2);
      \draw[green!50!black, ultra thick, dotted] (3,0) rectangle (5,2);
      \draw[white, ultra thick] (6,0) rectangle (8,2);
      \draw[blue, ultra thick, dashed] (6,0) rectangle (8,2);
    \end{tikzpicture}
    \caption{}
  \end{subfigure}
  \begin{subfigure}[b]{0.14\textwidth}
    \centering
    \begin{tikzpicture}[scale=1]
      \fill[gray] (0+1,0+1)--(0,0+1)--(0+1,0);
      \fill[gray] (0+1,1+1)--(0,1+1)--(0+1,1);
      \fill[gray] (1,0)    --(1+1,0)--(1,0+1);
      \fill[gray] (1+1,1)  --(1+1,1+1)--(1,1);
      \draw (0,0) grid (2,2);
      \draw[red, ultra thick] (0,0) rectangle (2,2);
      \draw[red, -{Triangle[length=3mm, width=3mm]}] (0,1)--(0,1.3);
      \draw[red, -{Triangle[length=3mm, width=3mm]}] (2,1.3)--(2,1);
    \end{tikzpicture}
    \caption{}
  \end{subfigure}
  \begin{subfigure}[b]{0.14\textwidth}
    \centering
    \begin{tikzpicture}[scale=1]
      \fill[gray] (3,1+1)  --(3,1)--(3+1,1+1);
      \fill[gray] (3+1,0+1)--(3,0+1)--(3+1,0);
      \fill[gray] (4+1,0+1)--(4,0+1)--(4+1,0);
      \fill[gray] (4+1,1+1)--(4,1+1)--(4+1,1);
      \draw (3,0) grid (5,2);
      \draw[white, ultra thick] (3,0) rectangle (5,2);
      \draw[green!50!black, ultra thick, dotted] (3,0) rectangle (5,2);
      \draw[green!50!black, -{Triangle[length=3mm, width=3mm]}] (3,1)--(3,1.3);
      \draw[green!50!black, -{Triangle[length=3mm, width=3mm]}] (5,1.3)--(5,1);
    \end{tikzpicture}
    \caption{}
  \end{subfigure}
  \begin{subfigure}[b]{0.14\textwidth}
    \centering
    \begin{tikzpicture}[scale=1]
      \fill[gray] (6+1,1)  --(6+1,1+1)--(6,1);
      \fill[gray] (6+1,0)  --(6+1,0+1)--(6,0);
      \fill[gray] (7,1+1)  --(7,1)--(7+1,1+1);
      \fill[gray] (7+1,0+1)--(7,0+1)--(7+1,0);
      \draw (6,0) grid (8,2);
      \draw[white, ultra thick] (6,0) rectangle (8,2);
      \draw[blue, ultra thick, dashed] (6,0) rectangle (8,2);
      \draw[blue, -{Triangle[length=3mm, width=3mm]}] (6,1)--(6,1.3);
      \draw[blue, -{Triangle[length=3mm, width=3mm]}] (8,1.3)--(8,1);

    \end{tikzpicture}
    \caption{}
  \end{subfigure}
  \caption{
    (a) A $2 \times 2$ M\"obius strip repeated four times horizontally.
    Parts (b), (c), and (d) show equivalent tilings under this symmetry.}
  \label{fig:mobiusStrip}
\end{figure}
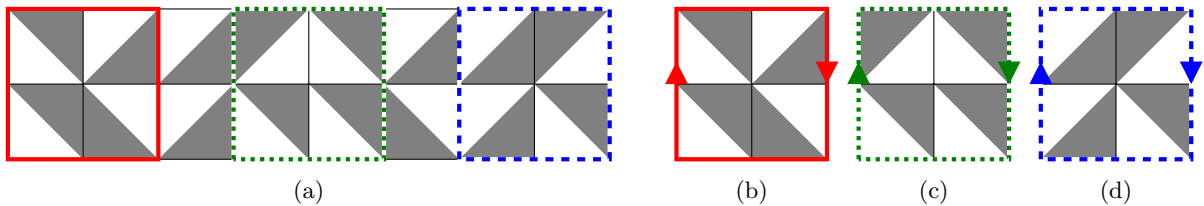
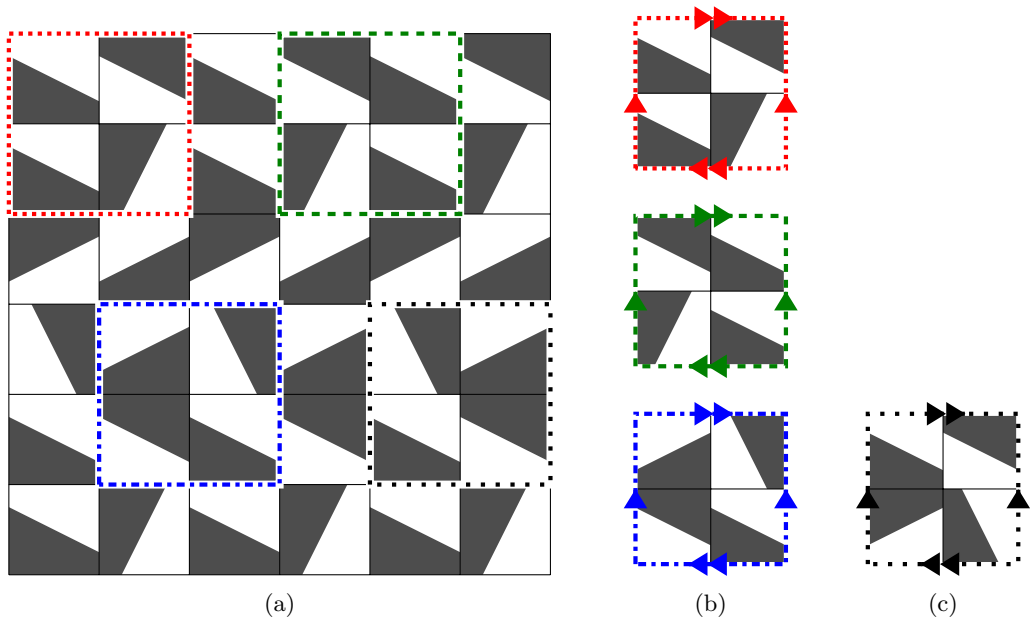
\begin{figure}[ht]
  \centering
  \begin{subfigure}[b]{0.5\textwidth}
    \centering
    \begin{tikzpicture}[scale=1.2]
      \foreach \x in {0, 2, 4} { \foreach \y in {0, 2, 4} {
        \begin{scope}[shift={({1+\x-(-1)^(\y/2)},\y)}, xscale={(-1)^(\y/2)}]
          \fill[black!70!white] (1,0)--(0,0)--(0,0.75)--(1,0.25);
          \fill[black!70!white] (1,1)--(1,0)--(1.25,0)--(1.75,1);
          \fill[black!70!white] (1,1)--(0,1)--(0,1.75)--(1,1.25);
          \fill[black!70!white] (1,2)--(2,2)--(2,1.25)--(1,1.75);
        \end{scope}
      }}
      \draw (0,0) grid (6,6);
      \draw[line width=3, white]      (0,4) rectangle (2,6);
      \draw[ultra thick, red, dotted] (0,4) rectangle (2,6);
      \draw[line width=3, white]                 (3,4) rectangle (5,6);
      \draw[ultra thick, green!50!black, dashed] (3,4) rectangle (5,6);
      \draw[line width=3, white]         (1,1) rectangle (3,3);
      \draw[ultra thick, blue, dash dot] (1,1) rectangle (3,3);
      \draw[line width=3, white] (4,1) rectangle (6,3);
      \draw[ultra thick, black, loosely dotted] (4,1) rectangle (6,3);
    \end{tikzpicture}
    \caption{}
  \end{subfigure}
  \begin{subfigure}[b]{0.18\textwidth}
    \centering
    \begin{tikzpicture}
      \fill[black!70!white] (1,0)--(0,0)--(0,0.75)--(1,0.25);
      \fill[black!70!white] (1,1)--(0,1)--(0,1.75)--(1,1.25);
      \fill[black!70!white] (1,2)--(2,2)--(2,1.25)--(1,1.75);
      \fill[black!70!white] (1,0)--(1,1)--(1.75,1)--(1.25,0);
      \draw (0,0) grid (2,2);
      \draw[ultra thick, white] (0,0) rectangle (2,2);
      \draw[ultra thick, red, dotted] (0,0) rectangle (2,2);
      \draw[ultra thick, red, -{Triangle}] (0,0.99)--(0,1);
      \draw[ultra thick, red, -{Triangle}] (2,0.99)--(2,1);
      \draw[ultra thick, red, -{Triangle}{Triangle}] (0.71,0)--(0.7,0);
      \draw[ultra thick, red, -{Triangle}{Triangle}] (1.29,2)--(1.3,2);
    \end{tikzpicture}\vspace{0.3cm}
    \begin{tikzpicture}
      \fill[black!70!white] (2,0)--(1,0)--(1,0.75)--(2,0.25);
      \fill[black!70!white] (2,1)--(1,1)--(1,1.75)--(2,1.25);
      \fill[black!70!white] (0,2)--(1,2)--(1,1.25)--(0,1.75);
      \fill[black!70!white] (0,0)--(0,1)--(0.75,1)--(0.25,0);
      \draw (0,0) grid (2,2);
      \draw[ultra thick, white] (0,0) rectangle (2,2);
      \draw[ultra thick, green!50!black, dashed] (0,0) rectangle (2,2);
      \draw[ultra thick, green!50!black, -{Triangle}] (0,0.99)--(0,1);
      \draw[ultra thick, green!50!black, -{Triangle}] (2,0.99)--(2,1);
      \draw[ultra thick, green!50!black, -{Triangle}{Triangle}] (0.71,0)--(0.7,0);
      \draw[ultra thick, green!50!black, -{Triangle}{Triangle}] (1.29,2)--(1.3,2);
    \end{tikzpicture}\vspace{0.3cm}
    \begin{tikzpicture}
      \fill[black!70!white] (2,1)--(2,2)--(1.25,2)--(1.75,1);
      \fill[black!70!white] (0,1)--(1,1)--(1,0.25)--(0,0.75);
      \fill[black!70!white] (2,0)--(1,0)--(1,0.75)--(2,0.25);
      \fill[black!70!white] (0,1)--(1,1)--(1,1.75)--(0,1.25);
      \draw (0,0) grid (2,2);
      \draw[ultra thick, white] (0,0) rectangle (2,2);
      \draw[ultra thick, blue, dash dot] (0,0) rectangle (2,2);
      \draw[ultra thick, blue, -{Triangle}] (0,0.99)--(0,1);
      \draw[ultra thick, blue, -{Triangle}] (2,0.99)--(2,1);
      \draw[ultra thick, blue, -{Triangle}{Triangle}] (0.71,0)--(0.7,0);
      \draw[ultra thick, blue, -{Triangle}{Triangle}] (1.29,2)--(1.3,2);
    \end{tikzpicture}
    \caption{}
  \end{subfigure}
  \begin{subfigure}[b]{0.18\textwidth}
    \centering
    \begin{tikzpicture}
      \fill[black!70!white] (0,0.25)--(0,1)--(1,1)--(1,0.75)--cycle;
      \fill[black!70!white] (0,1.75)--(0,1)--(1,1)--(1,1.25)--cycle;
      \fill[black!70!white] (1,2)--(2,2)--(2,1.25)--(1,1.75)--cycle;
      \fill[black!70!white] (1,0)--(1,1)--(1.25,1)--(1.75,0)--cycle;
      \draw (0,0) grid (2,2);
      \draw[ultra thick, white] (0,0) rectangle (2,2);
      \draw[ultra thick, black, loosely dotted] (0,0) rectangle (2,2);
      \draw[ultra thick, black, -{Triangle}] (0,0.99)--(0,1);
      \draw[ultra thick, black, -{Triangle}] (2,0.99)--(2,1);
      \draw[ultra thick, black, -{Triangle}{Triangle}] (0.71,0)--(0.7,0);
      \draw[ultra thick, black, -{Triangle}{Triangle}] (1.29,2)--(1.3,2);
    \end{tikzpicture}
    \caption{}
  \end{subfigure}
  \caption{
    Part (a) shows a $2 \times 2$ Klein bottle repeated three
  times horizontally and three times vertically, with three $2 \times 2$ regions
  selected. Part (b) shows three tilings of the $2 \times 2$ grid that are
  equivalent under the torus action $\Z/2\Z \times \Z/2\Z$.
  Part (c) shows a $2 \times 2$ Klein that is equivalent to the other Klein
  bottles under $180^\circ$ rotation.
  }
  \label{fig:kleinBottle}
\end{figure}

\subsection{Tilings of the triangular and hexagonal grids}
While this paper explored tilings of the square grid and related settings,
it is equally natural to ask about tilings of the triangular grid and hexagonal
grid.
In particular, it would be interesting to explore the number of tilings of
(1) triangular regions of the triangular grid,
(2) hexagonal regions of the triangular grid,
(3) triangular regions of the hexagonal grid, or
(4) hexagonal regions of the hexagonal grid,
all of which are illustrated in Figure \ref{fig:hexagonTriangle}.

In the cases of tiling hexagons and triangles in the triangular grid, each can
be extended to a tiling of the plane, as described in
Figure \ref{fig:triangleHexagonTiling}.

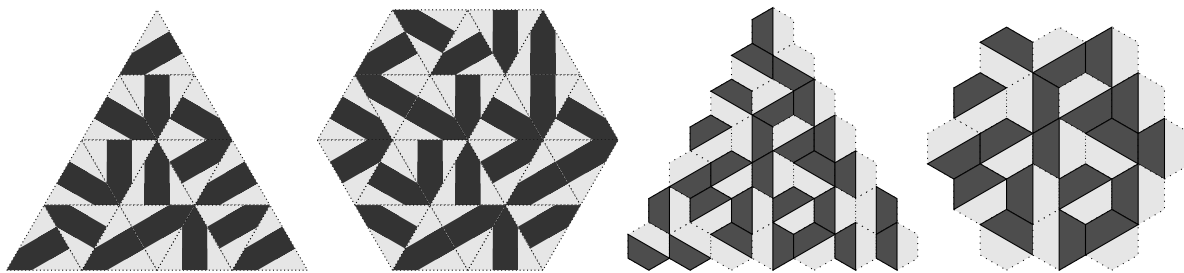
\begin{figure}[ht]
  \begin{tikzpicture}
    \foreach \a/\b/\r [evaluate=\r as \rr using {mod(\r + 1, 3)}] in {
                 0/3/3,
             0/2/1, 1/2/4,
           0/1/1, 1/1/5, 2/1/9,
        0/0/0, 1/0/6, 2/0/2, 3/0/0
    }{
      \pgfmathsetmacro{\x}{\a + \b/2}
      \pgfmathsetmacro{\y}{\b * sqrt(3)/2}
      \draw[xshift=\x cm, yshift=\y cm, densely dotted, fill=black!10] (0,0)--(1,0)--(1/2,{sqrt(3)/2})--cycle;
      \begin{scope}[xshift=\x cm, yshift=\y cm]
        \clip (0,0)--(1,0)--(1/2,{sqrt(3)/2})--cycle;
        \draw[line width=9.5, black!80] (1/2,{sqrt(3)/6}) -- ++({120*\r+30}:{sqrt(3)/6}) -- ++({120*\r-150}:{sqrt(3)/2});
      \end{scope}
    }

    \foreach \a/\b/\r [evaluate=\r as \rr using {mod(\r + 1, 3)}] in {
                 0/2/5,
              0/1/8, 1/1/9,
          0/0/7, 1/0/9, 2/0/3
    }{
      \pgfmathsetmacro{\x}{1 + \a + \b/2}
      \pgfmathsetmacro{\y}{\b * sqrt(3)/2}
      \draw[xshift=\x cm, yshift=\y cm, densely dotted, fill=black!10] (0,0)--(-1/2,{sqrt(3)/2})--(1/2,{sqrt(3)/2})--cycle;
      \begin{scope}[xshift=\x cm, yshift=\y cm]
        \clip (0,0)--(-1/2,{sqrt(3)/2})--(1/2,{sqrt(3)/2})--cycle;
        \draw[line width=9.5, black!80] (0,{sqrt(3)/3}) -- ++({120*\r+30}:{sqrt(3)/3}) -- ++({120*\r-150}:{sqrt(3)/2});
      \end{scope}
    }
  \end{tikzpicture}
  \begin{tikzpicture}
    \foreach \a/\b/\r [evaluate=\r as \rr using {mod(\r + 1, 3)}] in {
         -1/3/3, 0/3/3, 1/3/2,
      -1/2/1, 0/2/1, 1/2/4, 2/2/4,
           0/1/1, 1/1/5, 2/1/9,
             1/0/6, 2/0/5
    }{
      \pgfmathsetmacro{\x}{\a + \b/2}
      \pgfmathsetmacro{\y}{\b * sqrt(3)/2}
      \draw[xshift=\x cm, yshift=\y cm, densely dotted, fill=black!10] (0,0)--(1,0)--(1/2,{sqrt(3)/2})--cycle;
      \begin{scope}[xshift=\x cm, yshift=\y cm]
        \clip (0,0)--(1,0)--(1/2,{sqrt(3)/2})--cycle;
        \draw[line width=9.5, black!80] (1/2,{sqrt(3)/6}) -- ++({120*\r+30}:{sqrt(3)/6}) -- ++({120*\r-150}:{sqrt(3)/2});
      \end{scope}
    }

    \foreach \a/\b/\r [evaluate=\r as \rr using {mod(\r + 1, 3)}] in {
              -1/3/4, 0/3/5,
          -1/2/4, 0/2/5, 1/2/5,
      -1/1/3, 0/1/8, 1/1/9, 2/1/0,
          0/0/7, 1/0/9, 2/0/3
    }{
      \pgfmathsetmacro{\x}{1 + \a + \b/2}
      \pgfmathsetmacro{\y}{\b * sqrt(3)/2}
      \draw[xshift=\x cm, yshift=\y cm, densely dotted, fill=black!10] (0,0)--(-1/2,{sqrt(3)/2})--(1/2,{sqrt(3)/2})--cycle;
      \begin{scope}[xshift=\x cm, yshift=\y cm]
        \clip (0,0)--(-1/2,{sqrt(3)/2})--(1/2,{sqrt(3)/2})--cycle;
        \draw[line width=9.5, black!80] (0,{sqrt(3)/3}) -- ++({120*\r+30}:{sqrt(3)/3}) -- ++({120*\r-150}:{sqrt(3)/2});
      \end{scope}
    }
  \end{tikzpicture}
  \begin{tikzpicture}[scale=0.55]
  \foreach \a/\b\r in {
    0/3/3,
    0/2/1,  1/2/4,
    0/1/1,  1/1/5,  2/1/9,
    0/0/2,  1/0/6,  2/0/5,  3/0/3,
    0/-1/5, 1/-1/8, 2/-1/9, 3/-1/7, 4/-1/9,
    0/-2/3, 1/-2/2, 2/-2/3, 3/-2/8, 4/-2/3, 5/-2/6,
    0/-3/4, 1/-3/4, 2/-3/1, 3/-3/6, 4/-3/2, 5/-3/3, 6/-3/3
  } {
    \pgfmathsetmacro{\x}{\a+\b/2}
    \pgfmathsetmacro{\y}{sqrt(3)/2*\b}
    \draw[xshift=\x cm, yshift=\y cm, fill=white!90!black, dotted] (0, 0)
      foreach \i in {0,...,5} {
        -- ++({30 + 60 * \i}:{sqrt(3)/3})
      };
    \fill[xshift=\x cm, yshift=\y cm, white!30!black, draw=black]
      (0, {sqrt(3)/3}) -- ++({30 + 60 * (\r-2)}:{sqrt(3)/3})
      foreach \i in {0,...,2} {
        -- ++({30 + 60 * (\i + \r)}:{sqrt(3)/3})
      } -- cycle;
  }
  \end{tikzpicture}
  \begin{tikzpicture}[scale=0.7]
    \foreach \a/\b\r in {
      0/1/1,  1/1/5,  2/1/9,
      0/0/2,  1/0/6,  2/0/5,  3/0/3,
      0/-1/5, 1/-1/8, 2/-1/9, 3/-1/7, 4/-1/9,
              1/-2/2, 2/-2/3, 3/-2/8, 4/-2/3,
                      2/-3/1, 3/-3/6, 4/-3/2
    } {
      \pgfmathsetmacro{\x}{\a+\b/2}
      \pgfmathsetmacro{\y}{sqrt(3)/2*\b}
      \draw[xshift=\x cm, yshift=\y cm, fill=white!90!black, dotted] (0, 0)
        foreach \i in {0,...,5} {
          -- ++({30 + 60 * \i}:{sqrt(3)/3})
        };
      \fill[xshift=\x cm, yshift=\y cm, white!30!black, draw=black]
        (0, {sqrt(3)/3}) -- ++({30 + 60 * (\r-2)}:{sqrt(3)/3})
        foreach \i in {0,...,2} {
          -- ++({30 + 60 * (\i + \r)}:{sqrt(3)/3})
        } -- cycle;
    }
  \end{tikzpicture}
  \caption{
    An illustration of a triangle in a triangular grid, a hexagon in a
    triangular grid, a triangle in a hexagonal grid, and a hexagon in a
    hexagonal grid.
  }
  \label{fig:hexagonTriangle}
\end{figure}

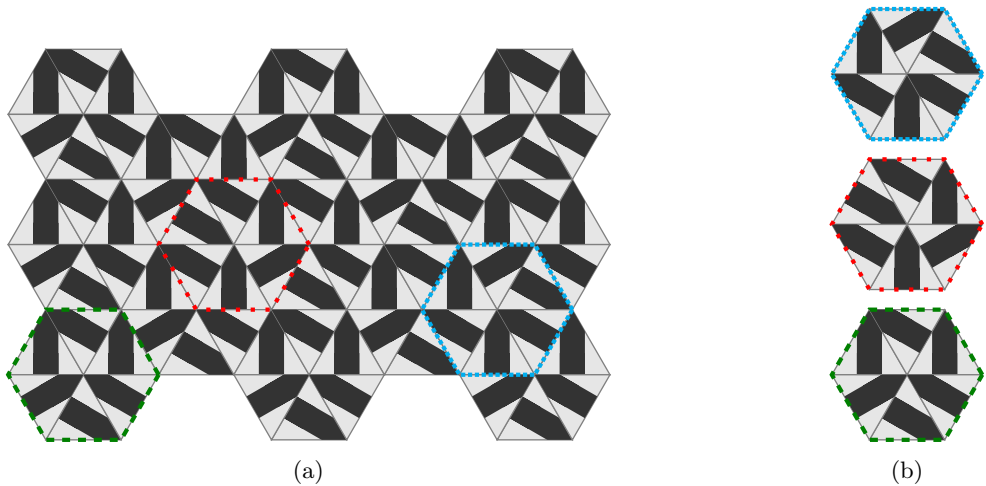
\begin{figure}[ht]
  \centering
  \begin{subfigure}[b]{0.7\textwidth}
    \centering
    \begin{tikzpicture}
      \foreach \i/\j in {
        -2/4, -1/4,  0/4,
        -1/3,  0/3,
        -1/2,  0/2,  1/2,
          0/1,  1/1,
          0/0,  1/0,  2/0
      } {\begin{scope}[shift={({3*(\i+\j/2)},{sqrt(3)/2*\j})}]
        \foreach \a/\b/\r [evaluate=\r as \rr using {mod(\r + 1, 3)}] in {
            -1/3/2, 0/3/2,
                  0/2/1
        }{
          \pgfmathsetmacro{\x}{\a + \b/2}
          \pgfmathsetmacro{\y}{\b * sqrt(3)/2}
          \fill[shift={(\x,\y)}, black!10, line width=0.5, draw=gray] (0,0)--(1,0)--(1/2,{sqrt(3)/2})--cycle;
          \begin{scope}[shift={(\x,\y)}]
            \clip (0,0)--(1,0)--(1/2,{sqrt(3)/2})--cycle;
            \draw[line width=9.5, black!80] (1/2,{sqrt(3)/6}) -- ++({120*\r+30}:{sqrt(3)/6}) -- ++({120*\r-150}:{sqrt(3)/2});
          \end{scope}

        }

        \foreach \a/\b/\r [evaluate=\r as \rr using {mod(\r + 1, 3)}] in {
                  -1/3/1,
              -1/2/0, 0/2/1
        }{
          \pgfmathsetmacro{\x}{1 + \a + \b/2}
          \pgfmathsetmacro{\y}{\b * sqrt(3)/2}
          \fill[shift={(\x,\y)}, black!10, line width=0.5, draw=gray] (0,0)--(-1/2,{sqrt(3)/2})--(1/2,{sqrt(3)/2})--cycle;
          \begin{scope}[shift={(\x,\y)}]
            \clip (0,0)--(-1/2,{sqrt(3)/2})--(1/2,{sqrt(3)/2})--cycle;
            \draw[line width=9.5, black!80] (0,{sqrt(3)/3}) -- ++({120*\r+30}:{sqrt(3)/3}) -- ++({120*\r-150}:{sqrt(3)/2});
          \end{scope}
        }
        \end{scope}}
        \draw[green!50!black, ultra thick, dashed] (1.5+1,{3*sqrt(3)/2}) \foreach \v in {0,60,120,180,240} { -- ++(120+\v:1cm) } -- cycle;
        \draw[red, ultra thick, loosely dotted] (3.5+1,{5*sqrt(3)/2}) \foreach \v in {0,60,120,180,240} { -- ++(120+\v:1cm) } -- cycle;
        \draw[cyan, ultra thick, densely dotted] (7+1,{4*sqrt(3)/2}) \foreach \v in {0,60,120,180,240} { -- ++(120+\v:1cm) } -- cycle;
    \end{tikzpicture}
    \caption{}
    \end{subfigure}
    \begin{subfigure}[b]{0.25\textwidth}
    \centering
    \begin{tikzpicture}
      \begin{scope}[yshift=0cm]
        \foreach \a/\b/\r [evaluate=\r as \rr using {mod(\r + 1, 3)}] in {
            -1/3/2, 0/3/2,
                  0/2/1
        }{
          \pgfmathsetmacro{\x}{\a + \b/2}
          \pgfmathsetmacro{\y}{\b * sqrt(3)/2}
          \fill[shift={(\x,\y)}, black!10, line width=0.5, draw=gray] (0,0)--(1,0)--(1/2,{sqrt(3)/2})--cycle;
          \begin{scope}[shift={(\x,\y)}]
            \clip (0,0)--(1,0)--(1/2,{sqrt(3)/2})--cycle;
            \draw[line width=9.5, black!80] (1/2,{sqrt(3)/6}) -- ++({120*\r+30}:{sqrt(3)/6}) -- ++({120*\r-150}:{sqrt(3)/2});
          \end{scope}

        }

        \foreach \a/\b/\r [evaluate=\r as \rr using {mod(\r + 1, 3)}] in {
                  -1/3/1,
              -1/2/0, 0/2/1
        }{
          \pgfmathsetmacro{\x}{1 + \a + \b/2}
          \pgfmathsetmacro{\y}{\b * sqrt(3)/2}
          \fill[shift={(\x,\y)}, black!10, line width=0.5, draw=gray] (0,0)--(-1/2,{sqrt(3)/2})--(1/2,{sqrt(3)/2})--cycle;
          \begin{scope}[shift={(\x,\y)}]
            \clip (0,0)--(-1/2,{sqrt(3)/2})--(1/2,{sqrt(3)/2})--cycle;
            \draw[line width=9.5, black!80] (0,{sqrt(3)/3}) -- ++({120*\r+30}:{sqrt(3)/3}) -- ++({120*\r-150}:{sqrt(3)/2});
          \end{scope}
        }
        \draw[green!50!black, ultra thick, dashed] (1.5+1,{3*sqrt(3)/2}) \foreach \v in {0,60,120,180,240} { -- ++(120+\v:1cm) } -- cycle;
      \end{scope}
      \begin{scope}[yshift=2cm]
      \foreach \a/\b/\r [evaluate=\r as \rr using {mod(\r + 1, 3)}] in {
          -1/3/1, 0/3/2,
                0/2/2
      }{
        \pgfmathsetmacro{\x}{\a + \b/2}
        \pgfmathsetmacro{\y}{\b * sqrt(3)/2}
        \fill[shift={(\x,\y)}, black!10, line width=0.5, draw=gray] (0,0)--(1,0)--(1/2,{sqrt(3)/2})--cycle;
        \begin{scope}[shift={(\x,\y)}]
          \clip (0,0)--(1,0)--(1/2,{sqrt(3)/2})--cycle;
          \draw[line width=9.5, black!80] (1/2,{sqrt(3)/6}) -- ++({120*\r+30}:{sqrt(3)/6}) -- ++({120*\r-150}:{sqrt(3)/2});
        \end{scope}

      }

      \foreach \a/\b/\r [evaluate=\r as \rr using {mod(\r + 1, 3)}] in {
                -1/3/1,
            -1/2/1, 0/2/0
      }{
        \pgfmathsetmacro{\x}{1 + \a + \b/2}
        \pgfmathsetmacro{\y}{\b * sqrt(3)/2}
        \fill[shift={(\x,\y)}, black!10, line width=0.5, draw=gray] (0,0)--(-1/2,{sqrt(3)/2})--(1/2,{sqrt(3)/2})--cycle;
        \begin{scope}[shift={(\x,\y)}]
          \clip (0,0)--(-1/2,{sqrt(3)/2})--(1/2,{sqrt(3)/2})--cycle;
          \draw[line width=9.5, black!80] (0,{sqrt(3)/3}) -- ++({120*\r+30}:{sqrt(3)/3}) -- ++({120*\r-150}:{sqrt(3)/2});
        \end{scope}
      }
      \draw[red, ultra thick, loosely dotted] (1.5+1,{3*sqrt(3)/2}) \foreach \v in {0,60,120,180,240} { -- ++(120+\v:1cm) } -- cycle;
    \end{scope}
    \begin{scope}[yshift=4cm]
      \foreach \a/\b/\r [evaluate=\r as \rr using {mod(\r + 1, 3)}] in {
          -1/3/2, 0/3/1,
                0/2/2
      }{
        \pgfmathsetmacro{\x}{\a + \b/2}
        \pgfmathsetmacro{\y}{\b * sqrt(3)/2}
        \fill[shift={(\x,\y)}, black!10, line width=0.5, draw=gray] (0,0)--(1,0)--(1/2,{sqrt(3)/2})--cycle;
        \begin{scope}[shift={(\x,\y)}]
          \clip (0,0)--(1,0)--(1/2,{sqrt(3)/2})--cycle;
          \draw[line width=9.5, black!80] (1/2,{sqrt(3)/6}) -- ++({120*\r+30}:{sqrt(3)/6}) -- ++({120*\r-150}:{sqrt(3)/2});
        \end{scope}

      }

      \foreach \a/\b/\r [evaluate=\r as \rr using {mod(\r + 1, 3)}] in {
                -1/3/0,
            -1/2/1, 0/2/1
      }{
        \pgfmathsetmacro{\x}{1 + \a + \b/2}
        \pgfmathsetmacro{\y}{\b * sqrt(3)/2}
        \fill[shift={(\x,\y)}, black!10, line width=0.5, draw=gray] (0,0)--(-1/2,{sqrt(3)/2})--(1/2,{sqrt(3)/2})--cycle;
        \begin{scope}[shift={(\x,\y)}]
          \clip (0,0)--(-1/2,{sqrt(3)/2})--(1/2,{sqrt(3)/2})--cycle;
          \draw[line width=9.5, black!80] (0,{sqrt(3)/3}) -- ++({120*\r+30}:{sqrt(3)/3}) -- ++({120*\r-150}:{sqrt(3)/2});
        \end{scope}
      }
      \draw[cyan, ultra thick, densely dotted] (1.5+1,{3*sqrt(3)/2}) \foreach \v in {0,60,120,180,240} { -- ++(120+\v:1cm) } -- cycle;
    \end{scope}
    \end{tikzpicture}
    \caption{}
  \end{subfigure}
  \caption{
    (a) A triangular tiling of the plane tiled with repeating patterns of size
    $1$ hexagons. (b) Three equivalent tilings of the triangular hexagon under
    this symmetry.
  }
  \label{fig:triangleHexagonTiling}
\end{figure}

\subsection{Other tilings of the Euclidean plane}
Extending this idea even further, we might be interested in ways of placing
multiple shapes of tiles on various tilings of the Euclidean plane by convex
polygons, such as the
truncated trihexagonal tiling,
the snub square tiling, or the
triakis triangular tiling.

Figure \ref{fig:truncatedSquareTiling} shows an example of such a setup on a
region of the truncated square grid.
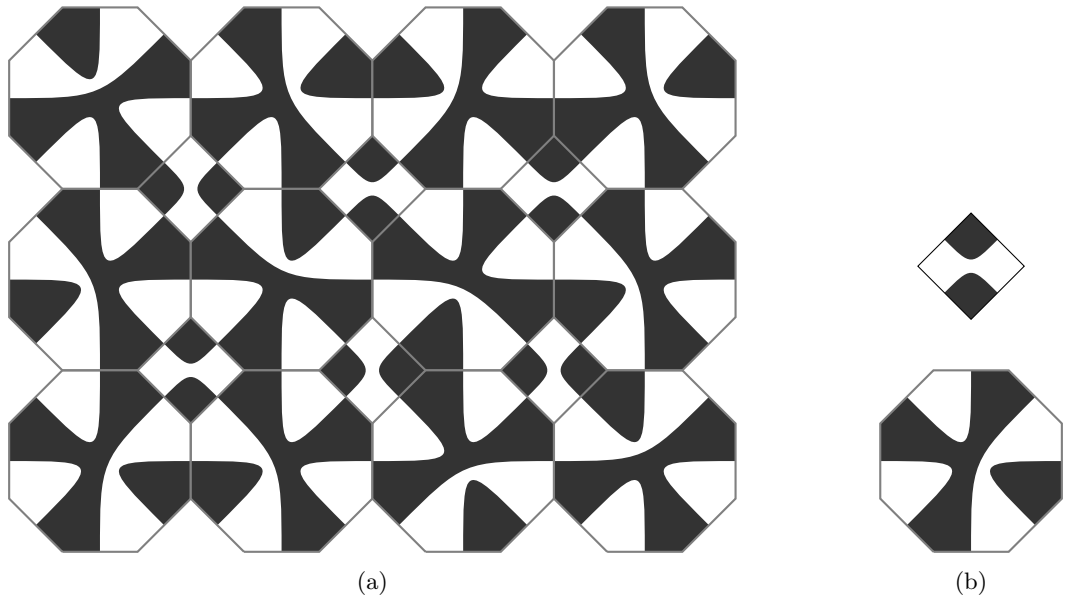
\begin{figure}[ht]
  \centering
  \begin{subfigure}[b]{0.7\textwidth}
    \centering
    \begin{tikzpicture}
      \foreach \x/\y/\j in {
        0/1/5, 1/1/2, 2/1/6,
        0/0/0, 1/0/5, 2/0/7
      }{
        \fill[
          black!80!white,
          shift={({(\x+1)*(1+sqrt(2))},{(\y+1)*(1+sqrt(2))})},
          rotate around={90*\j:({-sqrt(2)/2},0)}
        ]
          ({-sqrt(2)/4},{sqrt(2)/4}) .. controls ({-sqrt(2)/2},0) .. ({-3*sqrt(2)/4},{sqrt(2)/4}) -- ({-sqrt(2)/2},{sqrt(2)/2})
          ({-3*sqrt(2)/4},{-sqrt(2)/4}) .. controls ({-sqrt(2)/2},0) .. ({-sqrt(2)/4},{-sqrt(2)/4}) -- ({-sqrt(2)/2},{-sqrt(2)/2})
        ;
      }
      \foreach \x/\y/\j in {
        0/2/3, 1/2/1, 2/2/4, 3/2/1,
        0/1/5, 1/1/2, 2/1/6, 3/1/5,
        0/0/0, 1/0/5, 2/0/7, 3/0/3
      } {
        \fill[
          black!80!white,
          ultra thick,
          shift={({\x+\x*sqrt(2)},{\y+\y*sqrt(2)})},
          rotate around={45*\j:(1/2,{1/2 + 1/2*sqrt(2)})}
        ]
          ({1+sqrt(2)/4},{sqrt(2)/4}) .. controls (1/2,{1/2 + 1/2*sqrt(2)}) .. ({1+sqrt(2)/2},{1/2+sqrt(2)/2}) -- ({1+sqrt(2)/2},{sqrt(2)/2})
          (1/2,0) .. controls (1/2,{1/2 + 1/2*sqrt(2)}) .. ({1+sqrt(2)/4},{1+3*sqrt(2)/4})
          -- (1,{1+sqrt(2)}) --
          (1/2,{1+sqrt(2)}) .. controls (1/2,{1/2 + 1/2*sqrt(2)}) .. ({-sqrt(2)/4},{1+3*sqrt(2)/4})
          -- ({-sqrt(2)/2},{1+sqrt(2)/2}) --
          ({-sqrt(2)/2},{1/2+sqrt(2)/2}) .. controls (1/2,{1/2 + 1/2*sqrt(2)}) .. ({-sqrt(2)/4},{sqrt(2)/4})
          -- (0,0) -- cycle
        ;
        \draw[gray, thick, shift={({\x+\x*sqrt(2)},{\y+\y*sqrt(2)})}] (0,0) \foreach \v in {0,45,...,315} { -- ++(0+\v:1cm) };
      }
    \end{tikzpicture}
    \caption{}
  \end{subfigure}
  \begin{subfigure}[b]{0.25\textwidth}
    \centering
    \begin{tikzpicture}
      \begin{scope}[shift={({1/2+sqrt(2)/2},3.8)}]
        \fill[black!80!white]
          ({-sqrt(2)/4},{sqrt(2)/4}) .. controls ({-sqrt(2)/2},0) .. ({-3*sqrt(2)/4},{sqrt(2)/4}) -- ({-sqrt(2)/2},{sqrt(2)/2})
          ({-3*sqrt(2)/4},{-sqrt(2)/4}) .. controls ({-sqrt(2)/2},0) .. ({-sqrt(2)/4},{-sqrt(2)/4}) -- ({-sqrt(2)/2},{-sqrt(2)/2})
        ;
        \draw (0,0) -- ({-sqrt(2)/2},{sqrt(2)/2}) -- ({-sqrt(2)},0) -- ({-sqrt(2)/2},{-sqrt(2)/2}) -- cycle;

      \end{scope}
      \fill[
        black!80!white
      ]
        ({1+sqrt(2)/4},{sqrt(2)/4}) .. controls (1/2,{1/2 + 1/2*sqrt(2)}) .. ({1+sqrt(2)/2},{1/2+sqrt(2)/2}) -- ({1+sqrt(2)/2},{sqrt(2)/2})
        (1/2,0) .. controls (1/2,{1/2 + 1/2*sqrt(2)}) .. ({1+sqrt(2)/4},{1+3*sqrt(2)/4})
        -- (1,{1+sqrt(2)}) --
        (1/2,{1+sqrt(2)}) .. controls (1/2,{1/2 + 1/2*sqrt(2)}) .. ({-sqrt(2)/4},{1+3*sqrt(2)/4})
        -- ({-sqrt(2)/2},{1+sqrt(2)/2}) --
        ({-sqrt(2)/2},{1/2+sqrt(2)/2}) .. controls (1/2,{1/2 + 1/2*sqrt(2)}) .. ({-sqrt(2)/4},{sqrt(2)/4})
        -- (0,0) -- cycle
      ;
      \draw[gray, thick] (0,0) \foreach \v in {0,45,...,315} { -- ++(0+\v:1cm) };
    \end{tikzpicture}
    \caption{}
  \end{subfigure}
  \caption{(a) A $4 \times 3$ section of the truncated square tiling, and (b) the square and octagonal tile designs.}
  \label{fig:truncatedSquareTiling}
\end{figure}
\subsection{Polyhedra}
We are also interested in settings related to polyhedra.
For instance, one could use various tile designs to count the number of distinct
tilings of a $2 \times 2 \times 2$ Rubik's cube-like object, as illustrated in
Figure \ref{fig:TruchetCube}.

Similarly, one could do this analysis on other polyhedra, not limited to the
Platonic solids, Archimedean solids, Johnson solids, prisms, antiprisms, and
even polyhedra whose faces are not regular polygons, such as the
rhombic dodecahedron.

\begin{figure}[ht]
  \begin{tikzpicture}[line join=bevel]
  \fill[black!50]
  (2,2,2)--
  (2,2,1)--
  (2,1,2)--
  cycle;
  \fill[black!50]
  (2,1,1)--
  (2,0,1)--
  (2,1,0)--
  cycle;
  \fill[black!50]
  (2,1,1)--
  (2,2,0)--
  (2,2,1)--
  cycle;
  \fill[black!50]
  (2,1,1)--
  (2,0,2)--
  (2,0,1)--
  cycle;

  \fill[black!50]
  (1,2,1)--
  (2,2,2)--
  (1,2,2)--
  cycle;
  \fill[black!50]
  (0,2,0)--
  (1,2,0)--
  (0,2,1)--
  cycle;
  \fill[black!50]
  (2,2,0)--
  (1,2,0)--
  (2,2,1)--
  cycle;
  \fill[black!50]
  (0,2,2)--
  (1,2,2)--
  (0,2,1)--
  cycle;

  \fill[black!50]
  (2,2,2)--
  (1,2,2)--
  (2,1,2)--
  cycle;
  \fill[black!50]
  (0,0,2)--
  (1,0,2)--
  (0,1,2)--
  cycle;
  \fill[black!50]
  (1,1,2)--
  (1,0,2)--
  (2,1,2)--
  cycle;
  \fill[black!50]
  (1,1,2)--
  (0,2,2)--
  (1,2,2)--
  cycle;

  \draw (2,1,2)--(2,1,0);
  \draw (2,2,1)--(2,0,1);

  \draw (1,2,2)--(1,2,0);
  \draw (2,2,1)--(0,2,1);

  \draw (1,2,2)--(1,0,2);
  \draw (2,1,2)--(0,1,2);

  \draw (2,2,2) -- (2,2,0) -- (2,0,0) -- (2,0,2) -- cycle;
  \draw (2,2,2) -- (2,2,0) -- (0,2,0) -- (0,2,2) -- cycle;
  \draw (2,2,2) -- (2,0,2) -- (0,0,2) -- (0,2,2) -- cycle;
\end{tikzpicture}
  \hfill
  \begin{tikzpicture}[line join=bevel]
  \fill[black!50] (2,1,1)--(2,2,2)--
  (2,2,1)
  --cycle;
  \fill[black!50] (2,1,1)--(2,0,0)--
  (2,0,1)
  --cycle;
  \fill[black!50] (2,1,1)--(2,2,0)--
  (2,1,0)
  --cycle;
  \fill[black!50] (2,1,1)--(2,0,2)--
  (2,1,2)
  --cycle;

  \fill[black!50] (1,2,1)--(2,2,2)--
  (1,2,2)
  --cycle;
  \fill[black!50] (1,2,1)--(0,2,0)--
  (1,2,0)
  --cycle;
  \fill[black!50] (1,2,1)--(2,2,0)--
  (2,2,1)
  --cycle;
  \fill[black!50] (1,2,1)--(0,2,2)--
  (0,2,1)
  --cycle;

  \fill[black!50] (1,1,2)--(2,2,2)--
  (2,1,2)
  --cycle;
  \fill[black!50] (1,1,2)--(0,0,2)--
  (0,1,2)
  --cycle;
  \fill[black!50] (1,1,2)--(2,0,2)--
  (1,0,2)
  --cycle;
  \fill[black!50] (1,1,2)--(0,2,2)--
  (1,2,2)
  --cycle;

  \draw (2,1,2)--(2,1,0);
  \draw (2,2,1)--(2,0,1);

  \draw (1,2,2)--(1,2,0);
  \draw (2,2,1)--(0,2,1);

  \draw (1,2,2)--(1,0,2);
  \draw (2,1,2)--(0,1,2);

  \draw (2,2,2) -- (2,2,0) -- (2,0,0) -- (2,0,2) -- cycle;
  \draw (2,2,2) -- (2,2,0) -- (0,2,0) -- (0,2,2) -- cycle;
  \draw (2,2,2) -- (2,0,2) -- (0,0,2) -- (0,2,2) -- cycle;
\end{tikzpicture}
  \hfill
  \begin{tikzpicture}[line join=bevel]
  \fill[black!50]
  (2,2,2)--
  (2,2,1)--
  (2,1,2)
  --cycle;
  \fill[black!50]
  (2,0,0)--
  (2,0,1)--
  (2,1,0)
  --cycle;
  \fill[black!50]
  (2,2,0)--
  (2,1,0)--
  (2,2,1)
  --cycle;
  \fill[black!50]
  (2,0,2)--
  (2,0,1)--
  (2,1,2)
  --cycle;

  \fill[black!50]
  (2,2,2)--
  (1,2,2)--
  (2,2,1)
  --cycle;
  \fill[black!50]
  (0,2,0)--
  (1,2,0)--
  (0,2,1)
  --cycle;
  \fill[black!50]
  (2,2,0)--
  (1,2,0)--
  (2,2,1)
  --cycle;
  \fill[black!50]
  (0,2,2)--
  (1,2,2)--
  (0,2,1)
  --cycle;

  \fill[black!50]
  (2,2,2)--
  (1,2,2)--
  (2,1,2)
  --cycle;
  \fill[black!50]
  (0,0,2)--
  (1,0,2)--
  (0,1,2)
  --cycle;
  \fill[black!50]
  (2,0,2)--
  (1,0,2)--
  (2,1,2)
  --cycle;
  \fill[black!50]
  (0,2,2)--
  (0,1,2)--
  (1,2,2)
  --cycle;

  \draw (2,1,2)--(2,1,0);
  \draw (2,2,1)--(2,0,1);

  \draw (1,2,2)--(1,2,0);
  \draw (2,2,1)--(0,2,1);

  \draw (1,2,2)--(1,0,2);
  \draw (2,1,2)--(0,1,2);

  \draw (2,2,2) -- (2,2,0) -- (2,0,0) -- (2,0,2) -- cycle;
  \draw (2,2,2) -- (2,2,0) -- (0,2,0) -- (0,2,2) -- cycle;
  \draw (2,2,2) -- (2,0,2) -- (0,0,2) -- (0,2,2) -- cycle;
\end{tikzpicture}
  \hfill
  \begin{tikzpicture}[line join=bevel]
  \fill[black!50] (2,1,1)--(2,2,2)--
  (2,1,2)
  --cycle;
  \fill[black!50] (2,1,1)--(2,0,0)--
  (2,1,0)
  --cycle;
  \fill[black!50] (2,1,1)--(2,2,0)--
  (2,1,0)
  --cycle;
  \fill[black!50] (2,1,1)--(2,0,2)--
  (2,1,2)
  --cycle;

  \fill[black!50] (1,2,1)--(2,2,2)--
  (2,2,1)
  --cycle;
  \fill[black!50] (1,2,1)--(0,2,0)--
  (0,2,1)
  --cycle;
  \fill[black!50] (1,2,1)--(2,2,0)--
  (2,2,1)
  --cycle;
  \fill[black!50] (1,2,1)--(0,2,2)--
  (0,2,1)
  --cycle;

  \fill[black!50] (1,1,2)--(2,2,2)--
  (1,2,2)
  --cycle;
  \fill[black!50] (1,1,2)--(0,0,2)--
  (1,0,2)
  --cycle;
  \fill[black!50] (1,1,2)--(2,0,2)--
  (1,0,2)
  --cycle;
  \fill[black!50] (1,1,2)--(0,2,2)--
  (1,2,2)
  --cycle;

  \draw (2,1,2)--(2,1,0);
  \draw (2,2,1)--(2,0,1);

  \draw (1,2,2)--(1,2,0);
  \draw (2,2,1)--(0,2,1);

  \draw (1,2,2)--(1,0,2);
  \draw (2,1,2)--(0,1,2);

  \draw (2,2,2) -- (2,2,0) -- (2,0,0) -- (2,0,2) -- cycle;
  \draw (2,2,2) -- (2,2,0) -- (0,2,0) -- (0,2,2) -- cycle;
  \draw (2,2,2) -- (2,0,2) -- (0,0,2) -- (0,2,2) -- cycle;
\end{tikzpicture}
  \hfill
  \begin{tikzpicture}[line join=bevel]
  \fill[black!50]
  (2,1,1)--
  (2,2,1)--
  (2,1,2)--
  cycle;
  \fill[black!50]
  (2,1,1)--
  (2,0,0)--
  (2,1,0)--
  cycle;
  \fill[black!50]
  (2,1,1)--
  (2,2,0)--
  (2,1,0)--
  cycle;
  \fill[black!50]
  (2,1,1)--
  (2,0,2)--
  (2,1,2)--
  cycle;

  \fill[black!50]
  (1,2,1)--
  (2,2,2)--
  (2,2,1)--
  cycle;
  \fill[black!50]
  (1,2,1)--
  (1,2,0)--
  (0,2,1)--
  cycle;
  \fill[black!50]
  (1,2,1)--
  (1,2,0)--
  (2,2,1)--
  cycle;
  \fill[black!50]
  (1,2,1)--
  (1,2,2)--
  (0,2,1)--
  cycle;

  \fill[black!50]
  (1,1,2)--
  (1,2,2)--
  (2,1,2)--
  cycle;
  \fill[black!50]
  (1,1,2)--
  (1,0,2)--
  (0,1,2)--
  cycle;
  \fill[black!50]
  (1,1,2)--
  (2,0,2)--
  (2,1,2)--
  cycle;
  \fill[black!50]
  (1,1,2)--
  (0,2,2)--
  (0,1,2)--
  cycle;

  \draw (2,1,2)--(2,1,0);
  \draw (2,2,1)--(2,0,1);

  \draw (1,2,2)--(1,2,0);
  \draw (2,2,1)--(0,2,1);

  \draw (1,2,2)--(1,0,2);
  \draw (2,1,2)--(0,1,2);

  \draw (2,2,2) -- (2,2,0) -- (2,0,0) -- (2,0,2) -- cycle;
  \draw (2,2,2) -- (2,2,0) -- (0,2,0) -- (0,2,2) -- cycle;
  \draw (2,2,2) -- (2,0,2) -- (0,0,2) -- (0,2,2) -- cycle;
\end{tikzpicture}
  \caption{Five illustrations of $2 \times 2 \times 2$ cubes tiled with
  Truchet tiles.}
  \label{fig:TruchetCube}
\end{figure}
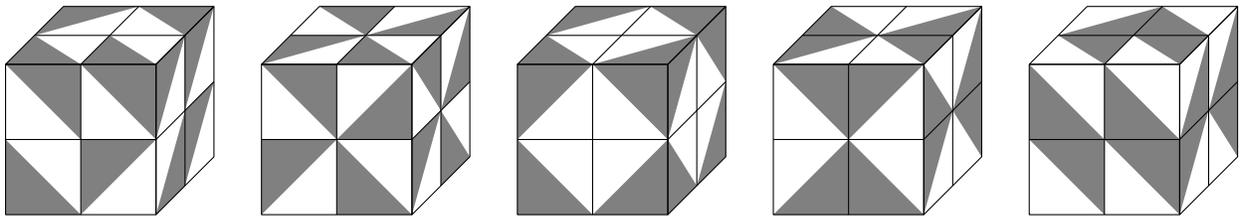

\subsection{Hyperbolic plane}
In addition to the settings with no curvature (the plane)
and positive curvature (polyhedra)
it is also interesting to look at this in the negative curvature setting of the
hyperbolic plane, as illustrated in Figure \ref{fig:hyperbolicTiling}.

\begin{figure}[ht]
  \includegraphics[width=0.45\textwidth]{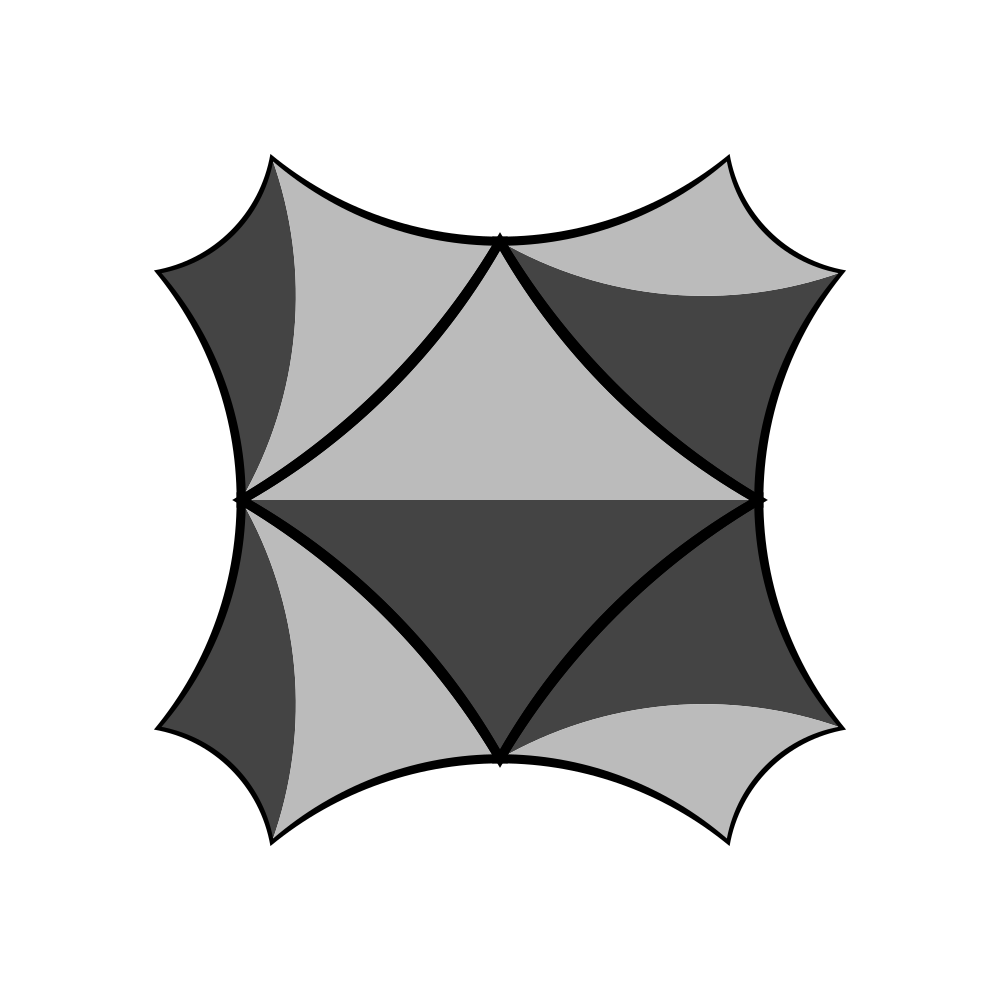}
  \hfill
  \includegraphics[width=0.45\textwidth]{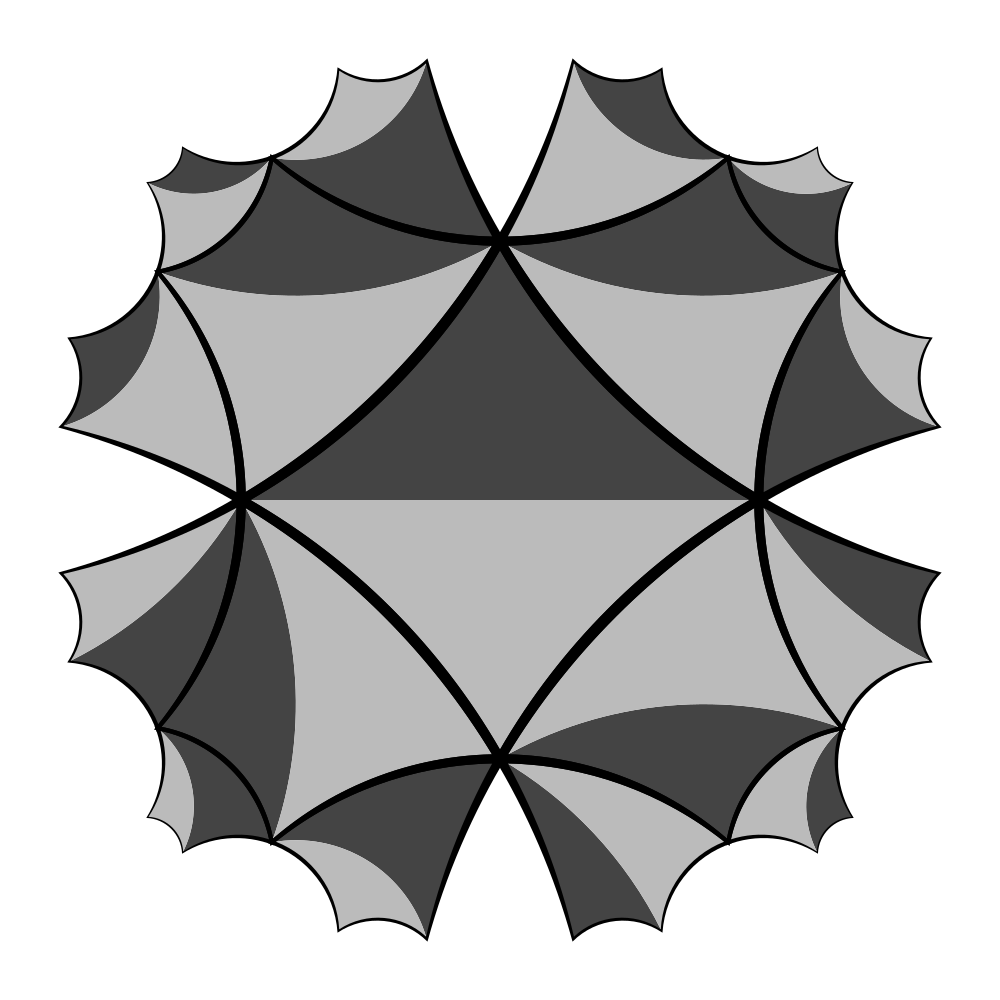}
  \caption{An illustration showing tilings of the size $2$ and size $3$
  iterations of the order-$5$ square tiling of the hyperbolic plane.}
  \label{fig:hyperbolicTiling}
\end{figure}

\subsection{Higher dimensional objects}
We are also interested in computing higher-dimensional analogs, such as where
the tile designs are space-filling polyhedra.

\subsection{Permuting tile colors}
In some illustrations, one may observe that two tilings are equivalent up to
swapping the colors of the tiles, as illustrated in
Figure \ref{fig:PermuteColors}. In Appendix \ref{apss:nXnTorusIllustrations},
you may notice that for each
tiling in Figures \ref{fig:nXnTorus_r_f__rf} and
\ref{fig:nXnTorus_r_f__id}, swapping the colors of the tiling is equivalent to
a $180^\circ$ rotation, a property that we would like to understand better.

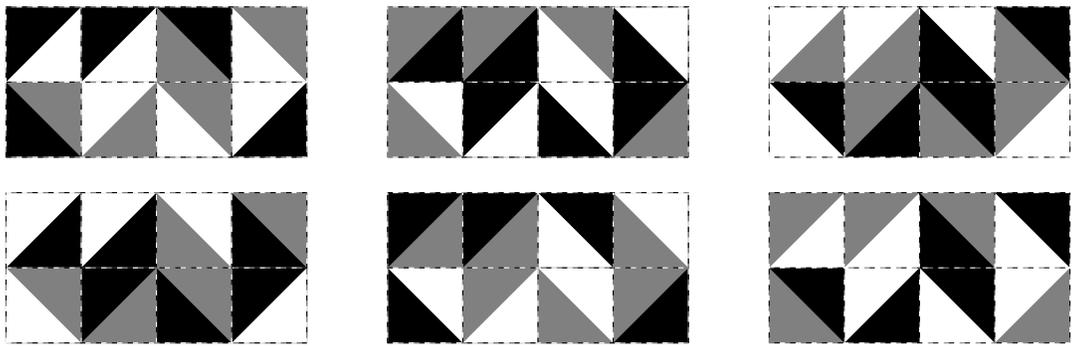
\begin{figure}[ht]
  \noindent
  ~ \hfill
  \begin{tikzpicture}
    \fill[white] (0,0) rectangle (4,2);
    \fill[gray] (0,0) rectangle (1,1);
    \fill[black] (0,1) -- (0,2) -- (1,2);
    \fill[black] (1,0) -- (0,0) -- (0,1);

    \fill[black] (1,2) -- (2,2) -- (1,1);
    \fill[gray] (1,0) -- (2,0) -- (2,1);

    \fill[gray] (2,1) rectangle (3,2);
    \fill[black] (2,2) -- (3,2) -- (3,1);
    \fill[gray] (2,1) -- (3,0) -- (3,1);

    \fill[gray] (3,2) -- (4,2) -- (4,1);
    \fill[black] (3,0) -- (4,0) -- (4,1);
    \draw[gray] (0,0) grid (4,2);
    \draw[black, dash pattern= on 2pt off 4pt,dash phase=0pt] (0,0) grid (4,2);
    \draw[white, dash pattern= on 2pt off 4pt,dash phase=2pt] (0,0) grid (4,2);
    \draw[gray, dash pattern= on 2pt off 4pt,dash phase=4pt] (0,0) grid (4,2);
  \end{tikzpicture}
  \hfill
  \begin{tikzpicture}
    \fill[black] (0,0) rectangle (4,2);
    \fill[white] (0,0) rectangle (1,1);
    \fill[gray] (0,1) -- (0,2) -- (1,2);
    \fill[gray] (1,0) -- (0,0) -- (0,1);

    \fill[gray] (1,2) -- (2,2) -- (1,1);
    \fill[white] (1,0) -- (2,0) -- (2,1);

    \fill[white] (2,1) rectangle (3,2);
    \fill[gray] (2,2) -- (3,2) -- (3,1);
    \fill[white] (2,1) -- (3,0) -- (3,1);

    \fill[white] (3,2) -- (4,2) -- (4,1);
    \fill[gray] (3,0) -- (4,0) -- (4,1);
    \draw[gray] (0,0) grid (4,2);
    \draw[gray, dash pattern= on 2pt off 4pt,dash phase=0pt] (0,0) grid (4,2);
    \draw[black, dash pattern= on 2pt off 4pt,dash phase=2pt] (0,0) grid (4,2);
    \draw[white, dash pattern= on 2pt off 4pt,dash phase=4pt] (0,0) grid (4,2);
  \end{tikzpicture}
  \hfill
  \begin{tikzpicture}
    \fill[gray] (0,0) rectangle (4,2);
    \fill[black] (0,0) rectangle (1,1);
    \fill[white] (0,1) -- (0,2) -- (1,2);
    \fill[white] (1,0) -- (0,0) -- (0,1);

    \fill[white] (1,2) -- (2,2) -- (1,1);
    \fill[black] (1,0) -- (2,0) -- (2,1);

    \fill[black] (2,1) rectangle (3,2);
    \fill[white] (2,2) -- (3,2) -- (3,1);
    \fill[black] (2,1) -- (3,0) -- (3,1);

    \fill[black] (3,2) -- (4,2) -- (4,1);
    \fill[white] (3,0) -- (4,0) -- (4,1);
    \draw[white, dash pattern= on 2pt off 4pt,dash phase=0pt] (0,0) grid (4,2);
    \draw[gray, dash pattern= on 2pt off 4pt,dash phase=2pt] (0,0) grid (4,2);
    \draw[black, dash pattern= on 2pt off 4pt,dash phase=4pt] (0,0) grid (4,2);
  \end{tikzpicture}
  \hfill ~

  ~

  \noindent
  ~ \hfill
  \begin{tikzpicture}
    \fill[black] (0,0) rectangle (4,2);
    \fill[gray] (0,0) rectangle (1,1);
    \fill[white] (0,1) -- (0,2) -- (1,2);
    \fill[white] (1,0) -- (0,0) -- (0,1);

    \fill[white] (1,2) -- (2,2) -- (1,1);
    \fill[gray] (1,0) -- (2,0) -- (2,1);

    \fill[gray] (2,1) rectangle (3,2);
    \fill[white] (2,2) -- (3,2) -- (3,1);
    \fill[gray] (2,1) -- (3,0) -- (3,1);

    \fill[gray] (3,2) -- (4,2) -- (4,1);
    \fill[white] (3,0) -- (4,0) -- (4,1);
    \draw[gray] (0,0) grid (4,2);
    \draw[white, dash pattern= on 2pt off 4pt,dash phase=0pt] (0,0) grid (4,2);
    \draw[black, dash pattern= on 2pt off 4pt,dash phase=2pt] (0,0) grid (4,2);
    \draw[gray, dash pattern= on 2pt off 4pt,dash phase=4pt] (0,0) grid (4,2);
  \end{tikzpicture}
  \hfill
  \begin{tikzpicture}
    \fill[gray] (0,0) rectangle (4,2);
    \fill[white] (0,0) rectangle (1,1);
    \fill[black] (0,1) -- (0,2) -- (1,2);
    \fill[black] (1,0) -- (0,0) -- (0,1);

    \fill[black] (1,2) -- (2,2) -- (1,1);
    \fill[white] (1,0) -- (2,0) -- (2,1);

    \fill[white] (2,1) rectangle (3,2);
    \fill[black] (2,2) -- (3,2) -- (3,1);
    \fill[white] (2,1) -- (3,0) -- (3,1);

    \fill[white] (3,2) -- (4,2) -- (4,1);
    \fill[black] (3,0) -- (4,0) -- (4,1);
    \draw[gray] (0,0) grid (4,2);
    \draw[black, dash pattern= on 2pt off 4pt,dash phase=0pt] (0,0) grid (4,2);
    \draw[gray, dash pattern= on 2pt off 4pt,dash phase=2pt] (0,0) grid (4,2);
    \draw[white, dash pattern= on 2pt off 4pt,dash phase=4pt] (0,0) grid (4,2);
  \end{tikzpicture}
  \hfill
  \begin{tikzpicture}
    \fill[white] (0,0) rectangle (4,2);
    \fill[black] (0,0) rectangle (1,1);
    \fill[gray] (0,1) -- (0,2) -- (1,2);
    \fill[gray] (1,0) -- (0,0) -- (0,1);

    \fill[gray] (1,2) -- (2,2) -- (1,1);
    \fill[black] (1,0) -- (2,0) -- (2,1);

    \fill[black] (2,1) rectangle (3,2);
    \fill[gray] (2,2) -- (3,2) -- (3,1);
    \fill[black] (2,1) -- (3,0) -- (3,1);

    \fill[black] (3,2) -- (4,2) -- (4,1);
    \fill[gray] (3,0) -- (4,0) -- (4,1);
    \draw[gray, dash pattern= on 2pt off 4pt,dash phase=0pt] (0,0) grid (4,2);
    \draw[white, dash pattern= on 2pt off 4pt,dash phase=2pt] (0,0) grid (4,2);
    \draw[black, dash pattern= on 2pt off 4pt,dash phase=4pt] (0,0) grid (4,2);
  \end{tikzpicture}
  \hfill ~
  \caption{The following six tilings would be considered equivalent under
  permuting colors.}
  \label{fig:PermuteColors}
  \end{figure}

\printbibliography

@article{Ethier,
  Author = {S. N. Ethier},
  Title = {Counting toroidal binary arrays},
  Year = {2013},
  Eprint = {arXiv:1301.2352},
  journal = {Journal of Integer Sequences},
  volume = {16}
}

@article{EthierLee,
  author = "S. N. Ethier and Jiyeon Lee",
  title = {Counting toroidal binary arrays, {II}},
  year = {2015},
  Eprint = {arXiv:1502.03792},
  journal = {Journal of Integer Sequences},
  volume = {18}
}

@phdthesis{Irvine,
  author       = {Veronika Irvine},
  title        = {Lace tessellations: a mathematical model for bobbin lace and an exhaustive combinatorial search for patterns},
  school       = {University of Victoria},
  year         = {2016}
}

@article{MacMahon,
author = {MacMahon, P. A.},
title = {Applications of a Theory of Permutations in Circular Procession to the Theory of Numbers},
journal = {Proceedings of the London Mathematical Society},
volume = {s1-23},
number = {1},
pages = {305-318},
doi = {https://doi.org/10.1112/plms/s1-23.1.305},
year = {1891}
}

@misc{OEIS,
  author={{OEIS Foundation Inc.}},
  title={{The On-Line Encyclopedia of Integer Sequences}},
  year={2023},
  url={https://oeis.org/},
}

\appendix
\appendixpagenumbering
\section{Sequences}
\label{apn:sequencesAndTables}
This section of the appendix gives examples of all of the different sequences
and tables of
integers that count tilings of the $n\times m$ grid, cylinder, and torus
for all valid choices of $R \leq D_8$ and all sets of tile designs
consisting of a single orbit.
\subsection{The \texorpdfstring{$n \times m$}{n by m} grid}     %
This section gives examples of every choice of symmetry of the $n \times m$
grid together with every essentially different set of tile designs that
consists of a single orbit (or two orbits, in the case of a fully symmetric tile).
Each sequence is annotated with its corresponding entry in the
On-Line Encyclopedia of Integer Sequences.
A table of all such sequences is given in Table \ref{tabl:nXmGridIndex}.
\begin{table}[ht]
\centering
\begin{tabular}{l|l|l|l|}
                                      & $\langle r^2, f \rangle$                      & $\langle f \rangle \cong C_2$            & $\langle r^2 \rangle \cong C_2$ \\
\hline
$\mathcal O_{\langle r^2, f \rangle}$ & \tableTableEntry{nXmGrid_rr_f__rr_f}{A225910} & \NA                                      & \NA                                       \\[10pt]
$\mathcal O_{\langle f \rangle}$      & \tableTableEntry{nXmGrid_rr_f__f}{A368218}    & \tableTableEntry{nXmGrid_f__f}{A368221}  & \NA                                       \\[10pt]
$\mathcal O_{\langle r^2 \rangle}$    & \tableTableEntry{nXmGrid_rr_f__rr}{A368219}   & \NA                                      & \tableTableEntry{nXmGrid_rr__rr}{A368223} \\[10pt]
$\mathcal O_\mathbbm{1}$              & \tableTableEntry{nXmGrid_rr_f__id}{A368220}   & \tableTableEntry{nXmGrid_f__id}{A368222} & \tableTableEntry{nXmGrid_rr__id}{A368224}
\end{tabular}
\caption{An index of tables that describe the number of tilings of the $n \times m$ grid.}
\label{tabl:nXmGridIndex}
\end{table}
\subsubsection{Under horizontal and vertical reflection}
When counting tilings of the grid up to ${\langle r^2, f \rangle}$, we have that
\begin{align}
  t_{\id} &=
        \orb{r^2,f}{r^2, f}
    + 2 \orb{r^2,f}{f}
    + 2 \orb{r^2,f}{r^2f}
    + 2 \orb{r^2,f}{r^2}
    + 4 \orbid{r^2,f} \\
  t_f      &= \orb{r^2,f}{r^2, f} + 2 \orb{r^2,f}{f} \\
  t_{r^2f} &= \orb{r^2,f}{r^2, f} + 2 \orb{r^2,f}{r^2f} \\
  t_{r^2}  &= \orb{r^2,f}{r^2, f} + 2 \orb{r^2,f}{r^2}
\end{align}
\begin{proposition}
When $\orb{r^2, f}{r^2, f} = 2$, such as when
\[
T = \left\{\tileColor[black!90!white], \tileColor[white!90!black]\right\},
\]
the number of tilings of the $n \times m$ grid up to horizontal/vertical
reflection
by tile designs that are fixed horizontal/vertical reflection
is given by the following table:
\[
\begin{array}{r}
n=1 \\ n=2 \\ n=3 \\ n=4 \\ n=5 \\ n=6
\end{array}
\left|
\begin{array}{rrrrrr}
2 & 3 & 6 & 10 & 20 & 36 \\
3 & 7 & 24 & 76 & 288 & 1072 \\
6 & 24 & 168 & 1120 & 8640 & 66816 \\
10 & 76 & 1120 & 16576 & 263680 & 4197376 \\
20 & 288 & 8640 & 263680 & 8407040 & 268517376 \\
36 & 1072 & 66816 & 4197376 & 268517376 & 17180065792 \\
\end{array}
\right.
\label{tabl:nXmGrid_rr_f__rr_f}
\]
\end{proposition}
This is OEIS sequence A225910.
\begin{proposition}
When $\orb{r^2, f}{f} = 1$, such as when
\[
T = \left\{\tileU, \tileD\right\},
\]
the number of tilings of the $n \times m$ grid up to horizontal and vertical
reflection by tiles that are fixed under horizontal reflection but not
vertical reflection is given by the following table:
\[
\begin{array}{r}
n=1 \\ n=2 \\ n=3 \\ n=4 \\ n=5 \\ n=6
\end{array}
\left|
\begin{array}{rrrrrr}
1  & 3   & 4      & 10      & 16        & 36 \\
2  & 7   & 20     & 76      & 272       & 1072 \\
3  & 24  & 144    & 1120    & 8448      & 66816 \\
6  & 76  & 1056   & 16576   & 262656    & 4197376 \\
10 & 288 & 8320   & 263680  & 8396800   & 268517376 \\
20 & 1072 & 65792 & 4197376 & 268451840 & 17180065792 \\
\end{array}
\right.
\label{tabl:nXmGrid_rr_f__f}
\]
The transpose of this table is the number of tilings by tiles fixed under
vertical reflection but not horizontal reflection.
\end{proposition}
This has been added to the OEIS as sequence A368218.
\begin{proposition}
When $\orb{r^2, f}{r^2} = 1$, such as when
\[
T = \left\{\tileDiag, \tileAdiag\right\},
\]
the number of tilings of the $n \times m$ grid up to horizontal and vertical
reflection by tiles that are fixed under $180^\circ$ rotation, but not
horizontal or vertical reflection is given by the following table:
\[
\begin{array}{r}
n=1 \\ n=2 \\ n=3 \\ n=4 \\ n=5 \\ n=6
\end{array}
\left|
\begin{array}{rrrrrr}
1 & 2 & 3 & 6 & 10 & 20 \\
2 & 7 & 20 & 76 & 272 & 1072 \\
3 & 20 & 136 & 1056 & 8256 & 65792 \\
6 & 76 & 1056 & 16576 & 262656 & 4197376 \\
10 & 272 & 8256 & 262656 & 8390656 & 268451840 \\
20 & 1072 & 65792 & 4197376 & 268451840 & 17180065792 \\
\end{array}
\right.
\label{tabl:nXmGrid_rr_f__rr}
\]
This table is symmetric across its main diagonal.
\end{proposition}
This has been added to the OEIS as sequence A368219.
\begin{proposition}
When $\orbid{r^2, f} = 1$, such as when
\[
T = \left\{\tileNW, \tileNE, \tileSW, \tileSE\right\},
\]
the number of tilings of the $n \times m$ grid up to horizontal and vertical
reflection by tiles that are fixed only under $\id \in D_4$ is given by the
following table:
\[
\begin{array}{r}
n=1 \\ n=2 \\ n=3 \\ n=4 \\ n=5 \\ n=6
\end{array}
\left|
\begin{array}{rrrrr}
1 & 6 & 16 & 72 & 256 \\
6 & 76 & 1056 & 16576 & 262656 \\
16 & 1056 & 65536 & 4196352 & 268435456 \\
72 & 16576 & 4196352 & 1073790976 & 274878431232 \\
256 & 262656 & 268435456 & 274878431232 & 281474976710656 \\
1056 & 4197376 & 17180000256 & 70368756760576 & 288230376688582656 \\
\end{array}
\right.
\label{tabl:nXmGrid_rr_f__id}
\]
\end{proposition}
This has been added to the OEIS as sequence A368220.
\subsubsection{Under horizontal (equivalently vertical) reflection}
When counting tilings of the grid up to ${\langle f \rangle}$
(equivalently ${\langle r^2f \rangle}$), we have that
\begin{align}
  t_{\id}  &= \orb{f}{f} + 2\orbid{f} \\
  t_f      &= \orb{f}{f}
\end{align}
\begin{proposition}
When $\orb{f}{f} = 2$, such as when
\[
T = \left\{\tileD[white!90!black], \tileD[black!90!white]\right\},
\]
the number of tilings of the $n \times m$ grid up to horizontal reflection by
two tiles that are fixed under horizontal reflection
is given by the following table:
\[
\begin{array}{r}
n=1 \\ n=2 \\ n=3 \\ n=4 \\ n=5 \\ n=6
\end{array}
\left|
\begin{array}{rrrrrr}
2 & 3 & 6 & 10 & 20 & 36 \\
4 & 10 & 40 & 136 & 544 & 2080 \\
8 & 36 & 288 & 2080 & 16640 & 131328 \\
16 & 136 & 2176 & 32896 & 526336 & 8390656 \\
32 & 528 & 16896 & 524800 & 16793600 & 536887296 \\
64 & 2080 & 133120 & 8390656 & 537001984 & 34359869440 \\
\end{array}
\right.
\label{tabl:nXmGrid_f__f}
\]
\end{proposition}
This has been added to the OEIS as sequence A368221.
\begin{proposition}
When $\orbid{f} = 1$, such as when
\[
T = \left\{\tileSW, \tileSE\right\},
\]
the number of tilings of the $n \times m$ grid up to horizontal reflection by
tiles that are fixed only under $\id \in \langle f \rangle$ is given by the
following table:
\[
\begin{array}{r}
n=1 \\ n=2 \\ n=3 \\ n=4 \\ n=5 \\ n=6
\end{array}
\left|
\begin{array}{rrrrrr}
1 & 2 & 4 & 8 & 16 & 32 \\
3 & 10 & 36 & 136 & 528 & 2080 \\
4 & 32 & 256 & 2048 & 16384 & 131072 \\
10 & 136 & 2080 & 32896 & 524800 & 8390656 \\
16 & 512 & 16384 & 524288 & 16777216 & 536870912 \\
36 & 2080 & 131328 & 8390656 & 536887296 & 34359869440 \\
\end{array}
\right.
\label{tabl:nXmGrid_f__id}
\]
\end{proposition}
This has been added to the OEIS as sequence A368222.
\subsubsection{Under \texorpdfstring{$180^\circ$}{180 degree} rotation}
When counting tilings of the grid up to $180^\circ$ rotation ($S = \langle r^2 \rangle$),
\begin{align}
  t_{\id} &= \orb{r^2}{r^2} + 2 \orbid{r^2} \\
  t_{r^2} &= \orb{r^2}{r^2}
\end{align}
\begin{proposition}
When $\orb{r^2}{r^2} = 2$ such as when
\[
T = \left\{\tileDiag[white!90!black], \tileDiag[black!90!white]\right\},
\]
the number of tilings of the $n \times m$ grid up to $180^\circ$ rotation by
tiles that are fixed under $180^\circ$ rotation
is given by the following table:
\[
\begin{array}{r}
n=1 \\ n=2 \\ n=3 \\ n=4 \\ n=5 \\ n=6
\end{array}
\left|
\begin{array}{rrrrrr}
2 & 3 & 6 & 10 & 20 & 36 \\
3 & 10 & 36 & 136 & 528 & 2080 \\
6 & 36 & 272 & 2080 & 16512 & 131328 \\
10 & 136 & 2080 & 32896 & 524800 & 8390656 \\
20 & 528 & 16512 & 524800 & 16781312 & 536887296 \\
36 & 2080 & 131328 & 8390656 & 536887296 & 34359869440 \\
\end{array}
\right.
\label{tabl:nXmGrid_rr__rr}
\]
\end{proposition}
This has been added to the OEIS as sequence A368223.
\begin{proposition}
When $\orbid{r^2} = 1$, such as when
\[
T = \left\{\tileNW, \tileSE\right\},
\]
the number of tilings of the $n \times m$ grid up to $180^\circ$ rotation by
tiles that are fixed only under $\id \in \langle r^2 \rangle$.
\[
\begin{array}{r}
n=1 \\ n=2 \\ n=3 \\ n=4 \\ n=5 \\ n=6
\end{array}
\left|
\begin{array}{rrrrrr}
1 & 3 & 4 & 10 & 16 & 36 \\
3 & 10 & 36 & 136 & 528 & 2080 \\
4 & 36 & 256 & 2080 & 16384 & 131328 \\
10 & 136 & 2080 & 32896 & 524800 & 8390656 \\
16 & 528 & 16384 & 524800 & 16777216 & 536887296 \\
36 & 2080 & 131328 & 8390656 & 536887296 & 34359869440 \\
\end{array}
\right.
\label{tabl:nXmGrid_rr__id}
\]
\end{proposition}
This has been added to the OEIS as sequence A368224.

\subsection{The \texorpdfstring{$n \times n$}{n by n} grid}     %
This section gives examples of every choice of symmetry of the $n \times n$
grid together with every essentially different set of tile designs that
consists of a single orbit (or two orbits, in the case of a fully symmetric tile).
Each sequence is annotated with its corresponding entry in the
On-Line Encyclopedia of Integer Sequences.
A table of all such sequences is given in Table \ref{tabl:nXnGridIndex}.
\begin{table}[ht]
\begin{tabular}{l|l|l|l|l|}
  & $\langle r, f \rangle$ & $\langle r^2, rf \rangle$     & $\langle r \rangle$             & $\langle rf \rangle$        \\
\hline
$\mathcal O_{\langle r, f \rangle}$    & \seqTableEntry{nXnGrid_r_f__r_f}{A054247}   & \NA                                           & \NA                                    & \NA                                     \\[10pt]
$\mathcal O_{\langle r^2, f \rangle}$  & \seqTableEntry{nXnGrid_r_f__rr_f}{A367522}  & \NA                                           & \NA                                    & \NA                                     \\[10pt]
$\mathcal O_{\langle r^2, rf \rangle}$ & \seqTableEntry{nXnGrid_r_f__rr_rf}{A295229} & \seqTableEntry{nXnGrid_rr_rf__rr_rf}{A367526} & \NA                                    & \NA                                     \\[10pt]
$\mathcal O_{\langle r \rangle}$       & \seqTableEntry{nXnGrid_r_f__r}{A367523}     & \NA                                           & \seqTableEntry{nXnGrid_r__r}{A047937}  & \NA                                     \\[10pt]
$\mathcal O_{\langle f \rangle}$       & \seqTableEntry{nXnGrid_r_f__f}{A367524}     & \NA                                           & \NA                                    & \NA                                     \\[10pt]
$\mathcal O_{\langle rf \rangle}$      & \seqTableEntry{nXnGrid_r_f__rf}{A302484}    & \seqTableEntry{nXnGrid_rr_rf__rf}{A367527}    & \NA                                    & \seqTableEntry{nXnGrid_rf__rf}{A200564} \\[10pt]
$\mathcal O_{\langle r^2 \rangle}$     & \seqTableEntry{nXnGrid_r_f__rr}{A367524}    & \seqTableEntry{nXnGrid_rr_rf__rr}{A367528}    & \seqTableEntry{nXnGrid_r__rr}{A367531} & \NA                                     \\[10pt]
$\mathcal O_{\mathbbm 1}$              & \seqTableEntry{nXnGrid_r_f__id}{A367525}    & \seqTableEntry{nXnGrid_rr_rf__id}{A367529}    & \seqTableEntry{nXnGrid_r__id}{A367532} & \seqTableEntry{nXnGrid_rf__id}{A103488}
\end{tabular}
\caption{An index of tables that describe the number of tilings of the $n \times n$ grid.}
\label{tabl:nXnGridIndex}
\end{table}
\subsubsection{Under symmetries of the square}
When counting tilings of the grid up to ${\langle r, f \rangle}$, we have that
\begin{align}
  t_{\id}  &= \orb{r,f}{r, f}
          + 2 \orb{r,f}{r^2, f}
          + 2 \orb{r,f}{r^2, rf}
          + 2 \orb{r,f}{r}
          + 4 \orb{r,f}{f}
          + 4 \orb{r,f}{rf}
          + 4 \orb{r,f}{r^2}
          + 8 \orbid{r,f} \\
  t_f = t_{r^2f} &= \orb{r,f}{r, f} + 2 \orb{r,f}{r^2, f} + 4 \orb{r,f}{f} \\
  t_{r^2}        &= \orb{r,f}{r, f} + 2 \orb{r,f}{r^2, f} + 2 \orb{r,f}{r^2, rf} + 2 \orb{r,f}{r} + 4 \orb{r,f}{r^2}.
\end{align}
\begin{proposition}
When $\orb{r, f}{r, f} = 2$, such as when
\[
  T = \left\{\tileColor[black!90!white], \tileColor[white!90!black]\right\},
\]
the number of tilings of the $n \times n$ grid up to $D_8$ action by
two distinct tile designs which are fixed under all elements of $D_8$ is given
by
\[
  2, 6, 102, 8548, 4211744, 8590557312, 70368882591744, 2305843028004192256, \dots
  \label{seq:nXnGrid_r_f__r_f}
\]
\end{proposition}
This is OEIS sequence A054247.
\begin{proposition}
When $\orb{r, f}{r^2, f} = 1$, such as when
\[
  T = \left\{\tileVert, \tileHor\right\},
\]
the number of tilings of the $n \times n$ grid up to $D_8$ action by
tiles that are stable under horizontal and vertical reflections is given by
\[
  1, 4, 84, 8292, 4203520, 8590033024, 70368815480832, 2305843010824323072, \dots
  \label{seq:nXnGrid_r_f__rr_f}
\]
\end{proposition}
This has been added to the OEIS as sequence A367522.
\begin{proposition}
When $\orb{r, f}{r^2, rf} = 1$, such as when
\[
  T = \left\{\tileDiag, \tileAdiag\right\},
\]
the number of tilings of the $n \times n$ grid up to $D_8$ action by
tiles that are stable under diagonal and antidiagonal reflections is given by
\[
  1, 6, 84, 8548, 4203520, 8590557312, 70368815480832, 2305843028004192256, \dots
  \label{seq:nXnGrid_r_f__rr_rf}
\]
\end{proposition}
This is OEIS sequence A295229.
\begin{proposition}
When $\orb{r, f}{r} = 1$, such as when
\[
  T = \left\{\tileRot{0}, \tileRot{1}\right\},
\]
the number of tilings of the $n \times n$ grid up to $D_8$ action by
tiles that are stable under $90^\circ$ rotations is given by
\[
  1, 4, 70, 8292, 4195360, 8590033024, 70368748374016, 2305843010824323072, \dots
  \label{seq:nXnGrid_r_f__r}
\]
\end{proposition}
This has been added to the OEIS as sequence A367523.
\begin{proposition}
When $\orb{r, f}{f} = 1$, (resp, $\orb{r, f}{r^2f} = 1$) such as when
\[
  T = \left\{\tileU, \tileD, \tileL, \tileR\right\},
\]
the number of tilings of the $n \times n$ grid up to $D_8$ action by
tiles that are stable under horizontal (resp. vertical) reflections is given
by
\[
  1, 39, 32896, 536895552, 140737496743936, 590295810384475521024, \dots
  \label{seq:nXnGrid_r_f__f}
\]
\end{proposition}
This has been added to the OEIS as sequence A367524.
\begin{proposition}
When $\orb{r, f}{rf} = 1$, (resp. $\orb{r, f}{r^3f} = 1$) such as when
\[
  T = \left\{\tileNW, \tileSW, \tileNE, \tileSE\right\},
\]
the number of tilings of the $n \times n$ grid up to $D_8$ action by
tiles that are stable under antidiagonal (resp. diagonal) reflections is given
by
\[
  1, 43, 32896, 536911936, 140737496743936, 590295810401655390208, \dots
  \label{seq:nXnGrid_r_f__rf}
\]
\end{proposition}
This is OEIS Sequence A302484.
\begin{proposition}
When $\orb{r, f}{r^2} = 1$, such as when
\[
  T = \left\{\tileRR{0}, \tileRR{1}, \tileRR{2}, \tileRR{3}\right\},
\]
the number of tilings of the $n \times n$ grid up to $D_8$ action by
tiles that are stable under $180^\circ$ rotation is given
by
\[
  1, 39, 32896, 536895552, 140737496743936, 590295810384475521024, \dots
  \label{seq:nXnGrid_r_f__rr}
\]
\end{proposition}
Note that the above sequence agrees with sequence in
Proposition \ref{seq:nXnGrid_r_f__f},
which has been added to the OEIS as sequence A367524.
\begin{proposition}
When $\orbid{r, f} = 1$, such as when
\[
  T = \left\{
    \tileAsym{0},
    \tileAsym{1},
    \tileAsym{2},
    \tileAsym{3},
    \tileAsym{4},
    \tileAsym{5},
    \tileAsym{6},
    \tileAsym{7}
  \right\},
\]
the number of tilings of the $n \times n$ grid up to $D_8$ action by
tiles that are stable under $180^\circ$ rotation is given
by
\[
  1, 538, 16777216, 35184378381312, 4722366482869645213696, \dots
  \label{seq:nXnGrid_r_f__id}
\]
\end{proposition}
This has been added to the OEIS as sequence A367525.
\subsubsection{Under diagonal and antidiagonal reflection}
When counting tilings of the grid up to ${\langle r^2, f \rangle}$, we have that
\begin{align}
  t_{\id}  &= \orb{r^2,f}{r^2,f}
          + 2 \orb{r^2,f}{f}
          + 2 \orb{r^2,f}{r^2f}
          + 2 \orb{r^2,f}{r^2} \\
  t_f      &= \orb{r^2,f}{r^2,f} + 2 \orb{r^2,f}{f} \\
  t_{r^2f} &= \orb{r^2,f}{r^2,f} + 2 \orb{r^2,f}{r^2f} \\
  t_{r^2}  &= \orb{r^2,f}{r^2,f} + 2 \orb{r^2,f}{r^2}
\end{align}
\begin{proposition}
When $\orb{r^2, rf}{r^2, rf} = 2$, such as when
\[
  T = \left\{
    \tileDiag[black!90!white],
    \tileDiag[black!10!white]
  \right\},
\]
the number of tilings of the $n \times n$ grid up to diagonal and antidiagonal
flipping by two colors of tiles that are stable under this symmetry is given
by
\[
  2, 9, 168, 16960, 8407040, 17180983296, 140737630961664, 4611686053860868096, \dots
  \label{seq:nXnGrid_rr_rf__rr_rf}
\]
\end{proposition}
This has been added to the OEIS as sequence A367526.
\begin{proposition}
When $\orb{r^2, rf}{rf} = 1$, (resp $\orb{r^2, rf}{r^3f} = 1$) such as when
\[
  T = \left\{
    \tileNW,
    \tileSE
  \right\},
\]
the number of tilings of the $n \times n$ grid up to diagonal and antidiagonal
flipping by the orbit of a tile that is stable under antidiagonal
(resp. diagonal) flipping is given by
\[
  1, 7, 144, 16704, 8396800, 17180459008, 140737555464192, 4611686036680998912, \dots
  \label{seq:nXnGrid_rr_rf__rf}
\]
\end{proposition}
This has been added to the OEIS as sequence A367527.
\begin{proposition}
When $\orb{r^2, rf}{r^2} = 1$, such as when
\[
  T = \left\{
    \tileVert,
    \tileHor
  \right\},
\]
the number of tilings of the $n \times n$ grid up to diagonal and antidiagonal
flipping by the orbit of a tile that is stable under $180^\circ$ rotation
is given by
\[
  1, 5, 136, 16448, 8390656, 17179934720, 140737496743936, 4611686019501129728, \dots
  \label{seq:nXnGrid_rr_rf__rr}
\]
\end{proposition}
This has been added to the OEIS as sequence A367528.
\begin{proposition}
When $\orbid{r^2, rf} = 1$, such as when
\[
  T = \left\{
    \tileAsym{0},
    \tileAsym{1},
    \tileAsym{2},
    \tileAsym{3}
  \right\},
\]
the number of tilings of the $n \times n$ grid up to diagonal and antidiagonal
flipping by the orbit of a tile that is not stable under any of these symmetries
is given by
\[
  1, 68, 65536, 1073758208, 281474976710656, 1180591620734591172608, \dots
  \label{seq:nXnGrid_rr_rf__id}
\]
\end{proposition}
This has been added to the OEIS as sequence A367529.
\subsubsection{Under \texorpdfstring{$90^\circ$}{90 degree} rotation}
\begin{proposition}
When $\orb{r}{r} = 2$, such as when
\[
  T = \left\{
    \tileRot[black!90!white]{0},
    \tileRot[black!10!white]{0}
  \right\},
\]
the number of tilings of the $n \times n$ grid up to $90^\circ$ rotation by
two colors of tiles that are fixed under this symmetry are
\[
  2, 6, 140, 16456, 8390720, 17179934976, 140737496748032, 4611686019501162496, \dots
  \label{seq:nXnGrid_r__r}
\]
This is in the OEIS as A047937, which is column $2$ of A343095.
\end{proposition}
\begin{proposition}
When $\orb{r}{r^2} = 1$, such as when
\[
  T = \left\{
    \tileRR{0},
    \tileRR{1}
  \right\},
\]
the number of tilings of the $n \times n$ grid up to $90^\circ$ rotation by
tiles that are fixed under $180^\circ$ rotations is given by
\[
  1, 6, 136, 16456, 8390656, 17179934976, 140737496743936, 4611686019501162496, \dots
  \label{seq:nXnGrid_r__rr}
\]
\end{proposition}
This has been added to the OEIS as sequence A367531.
\begin{proposition}
When $\orbid{r} = 1$, such as when
\[
  T = \left\{
    \tileAsym{0},
    \tileAsym{4},
    \tileAsym{2},
    \tileAsym{6}
  \right\},
\]
the number of tilings of the $n \times n$ grid up to $90^\circ$ rotation by
an asymmetric tile
\[
  1, 70, 65536, 1073758336, 281474976710656, 1180591620734591303680, \dots
  \label{seq:nXnGrid_r__id}
\]
\end{proposition}
This has been added to the OEIS as sequence A367532.
\subsubsection{Under diagonal (equivalently antidiagonal) reflection}
\begin{proposition}
When $\orb{rf}{rf} = 2$, such as when
\[
  T = \left\{ \tileSW[black!90!white], \tileSW[black!10!white] \right\}
  \qquad\text{or}\qquad
  T = \left\{ \tileSW, \tileNE \right\}
\]
the number of tilings of the $n \times n$ grid up to flipping over the
antidiagonal by
tiles that are fixed under that symmetry is given by
\[
  2, 12, 288, 33280, 16793600, 34360786944, 281475110928384, 9223372071214514176, \dots
  \label{seq:nXnGrid_rf__rf}
\]
\end{proposition}
This is OEIS sequence A200564.
\begin{proposition}
When $\orbid{rf} = 1$, such as when
\[
  T = \left\{
    \tileAsym{0},
    \tileAsym{1}
  \right\},
\]
the number of tilings of the $n \times n$ grid up to flipping over the
antidiagonal by asymmetric tiles is given by
\[
  1, 8, 256, 32768, 16777216, 34359738368, 281474976710656, 9223372036854775808, \dots
  \label{seq:nXnGrid_rf__id}
\]
\end{proposition}
This is OEIS sequence A103488.

\subsection{The \texorpdfstring{$n \times m$}{n by m} cylinder} %
This section gives examples of every choice of symmetry of the $n \times m$
cylinder together with every essentially different set of tile designs that
consists of a single orbit (or two orbits, in the case of a fully symmetric tile).
Each sequence is annotated with its corresponding entry in the
On-Line Encyclopedia of Integer Sequences.
A table of all such sequences is given in Table \ref{tabl:nXmCylIndex}.
\begin{table}[ht]
\begin{tabular}{l|l|l|l|l|l|}
                                    & $\langle r^2, f \rangle $                    & $\langle f \rangle$                     & $\langle r^2f \rangle$                     & $\langle r^2 \rangle$                    & $\mathbbm{1}$                            \\
\hline
$\mathcal O_V$                      & \tableTableEntry{nXmCyl_rr_f__rr_f}{A368253} & \NA                                     & \NA                                        & \NA                                      & \NA                                      \\[10pt]
$\mathcal O_{\langle f \rangle}$    & \tableTableEntry{nXmCyl_rr_f__f}{A368254}    & \tableTableEntry{nXmCyl_f__f}{A368258}  & \NA                                        & \NA                                      & \NA                                      \\[10pt]
$\mathcal O_{\langle r^2f \rangle}$ & \tableTableEntry{nXmCyl_rr_f__rrf}{A368255}  & \NA                                     & \tableTableEntry{nXmCyl_rrf__rrf}{A368260} & \NA                                      & \NA                                      \\[10pt]
$\mathcal O_{\langle r^2 \rangle}$  & \tableTableEntry{nXmCyl_rr_f__rr}{A368256}   & \NA                                     & \NA                                        & \tableTableEntry{nXmCyl_rr__rr}{A368262} & \NA                                      \\[10pt]
$\mathcal O_\mathbbm{1}$            & \tableTableEntry{nXmCyl_rr_f__id}{A368257}   & \tableTableEntry{nXmCyl_f__id}{A368259} & \tableTableEntry{nXmCyl_rrf__id}{A368261}  & \tableTableEntry{nXmCyl_rr__id}{A368263} & \tableTableEntry{nXmCyl_id__id}{A368264}
\end{tabular}
\caption{An index of tables that describe the number of tilings of the $n \times m$ cylinder.}
\label{tabl:nXmCylIndex}
\end{table}
\subsubsection{Under horizontal and vertical reflection}
\begin{proposition}
  When $\orb{r^2, f}{r^2, f} = 2$, such as when
  \[
    T = \left\{\tileVert[black!90!white], \tileVert[white!90!black]\right\}
    \qquad\text{or}\qquad
    T = \left\{\tileVert[black!90!white], \tileHor[black!90!white]\right\}
  \]
  the number of tilings of the $n \times m$ cylinder up to horizontal and vertical
  reflection by tiles that are fixed under those actions is given by
  \[
    \begin{array}{r}
      n=1 \\ n=2 \\ n=3 \\ n=4 \\ n=5 \\ n=6 \\ n=7
    \end{array}
    \left|
    \begin{array}{rrrrrrr}
      2 & 3 & 6 & 10 & 20 & 36 & 72\\
      3 & 7 & 24 & 76 & 288 & 1072 & 4224\\
      4 & 13 & 74 & 430 & 3100 & 23052 & 179736\\
      6 & 34 & 378 & 4756 & 70536 & 1083664 & 17053728\\
      8 & 78 & 1884 & 53764 & 1689608 & 53762472 & 1718629200\\
      13 & 237 & 11912 & 709316 & 44900448 & 2865540112 & 183287416192\\
      18 & 687 & 77022 & 9608050 & 1227536100 & 157077883188 & 20105440563816\\
    \end{array}
    \right.
    \label{tabl:nXmCyl_rr_f__rr_f}
  \]
\end{proposition}
This has been added to the OEIS as sequence A368253.
\begin{proposition}
When $\orb{r^2, f}{f} = 1$, such as when
\[
  T = \left\{\tileU, \tileD\right\}
\]
the number of tilings of the $n \times m$ cylinder up to horizontal and vertical
reflection by tiles that are fixed under horizontal reflection is given by \[
  \begin{array}{r}
    n=1 \\ n=2 \\ n=3 \\ n=4 \\ n=5 \\ n=6 \\ n=7
  \end{array}
  \left|
  \begin{array}{rrrrrrr}
    1 & 3 & 4 & 10 & 16 & 36 & 64 \\
    2 & 7 & 20 & 76 & 272 & 1072 & 4160 \\
    2 & 13 & 60 & 430 & 2992 & 23052 & 178880 \\
    4 & 34 & 346 & 4756 & 70024 & 1083664 & 17045536 \\
    4 & 78 & 1768 & 53764 & 1685920 & 53762472 & 1718511232 \\
    8 & 237 & 11612 & 709316 & 44881328 & 2865540112 & 183286192832 \\
    9 & 687 & 75924 & 9608050 & 1227395664 & 157077883188 & 20105422588224
  \end{array}
  \right.
  \label{tabl:nXmCyl_rr_f__f}
\]
\end{proposition}
This has been added to the OEIS as sequence A368254.
\begin{proposition}
  When $\orb{r^2, f}{r^2f} = 1$, such as when
  \[
    T = \left\{\tileL, \tileR\right\}
  \]
  the number of tilings of the $n \times m$ cylinder up to horizontal and vertical
  reflection by tiles that are fixed under vertical reflection is given by \[
    \begin{array}{r}
      n=1 \\ n=2 \\ n=3 \\ n=4 \\ n=5 \\ n=6 \\ n=7
    \end{array}
    \left|
    \begin{array}{rrrrrrr}
      1 & 2 & 3 & 6 & 10 & 20 & 36 \\
      2 & 5 & 14 & 44 & 152 & 560 & 2144 \\
      2 & 9 & 50 & 366 & 2780 & 22028 & 175128 \\
      4 & 26 & 298 & 4244 & 66184 & 1050896 & 16787488 \\
      4 & 62 & 1692 & 52740 & 1679368 & 53696936 & 1718039376 \\
      9 & 205 & 11272 & 701124 & 44761184 & 2863442960 & 183253337472 \\
      10 & 623 & 75486 & 9591666 & 1227208420 & 157073688884 & 20105365066344
    \end{array}
    \right.
    \label{tabl:nXmCyl_rr_f__rrf}
  \]
\end{proposition}
This has been added to the OEIS as sequence A368255.
\begin{proposition}
  When $\orb{r^2, f}{r^2} = 1$, such as when
  \[
    T = \left\{\tileDiag, \tileAdiag\right\}
  \]
  the number of tilings of the $n \times m$ cylinder up to horizontal and vertical
  reflection by tiles that are fixed under $180^\circ$ rotation is given by \[
    \begin{array}{r}
      n=1 \\ n=2 \\ n=3 \\ n=4 \\ n=5 \\ n=6 \\ n=7
    \end{array}
    \left|
    \begin{array}{rrrrrrr}
      1 & 2 & 3 & 6 & 10 & 20 & 36 \\
      2 & 5 & 14 & 44 & 152 & 560 & 2144 \\
      2 & 9 & 52 & 366 & 2800 & 22028 & 175296 \\
      4 & 26 & 298 & 4244 & 66184 & 1050896 & 16787488 \\
      4 & 62 & 1704 & 52740 & 1679776 & 53696936 & 1718052480 \\
      8 & 205 & 11228 & 701124 & 44758448 & 2863442960 & 183253162688 \\
      9 & 623 & 75412 & 9591666 & 1227199056 & 157073688884 & 20105363867968
    \end{array}
    \right.
    \label{tabl:nXmCyl_rr_f__rr}
  \]
\end{proposition}
This has been added to the OEIS as sequence A368256.
\begin{proposition}
  When $\orbid{r^2, f} = 1$, such as when
  \[
    T = \left\{\tileNW, \tileSW, \tileNE, \tileSE\right\}
  \]
  the number of tilings of the $n \times m$ cylinder up to horizontal and vertical
  reflection by asymmetric tiles is given by \[
    \begin{array}{r}
      n=1 \\ n=2 \\ n=3 \\ n=4 \\ n=5 \\ n=6
    \end{array}
    \left|
    \begin{array}{rrrrrr}
      1 & 6 & 16 & 72 & 256 & 1056 \\
      4 & 44 & 544 & 8384 & 131584 & 2100224 \\
      6 & 366 & 21856 & 1399512 & 89478656 & 5726711136 \\
      23 & 4244 & 1050128 & 268472384 & 68719870208 & 17592195482624 \\
      52 & 52740 & 53687104 & 54975896016 & 56294995342336 & 57646075552465728 \\
      194 & 701124 & 2863399264 & 11728132423744 & 48038396383286784 & 196765270153929688064
    \end{array}
    \right.
    \label{tabl:nXmCyl_rr_f__id}
  \]
\end{proposition}
This has been added to the OEIS as sequence A368257.
\subsubsection{Under horizontal reflection}
\begin{proposition}
  When $\orb{f}{f} = 2$, such as when
  \[
    T = \left\{\tileColor[black!90!white], \tileColor[black!10!white]\right\}
    \qquad \text{or} \qquad
    T = \left\{\tileU, \tileD\right\}
  \]
  the number of tilings of the $n \times m$ cylinder up to horizontal
  reflection two distinct tiles that are stable under horizontal reflection is
  given by
  \[
    \begin{array}{r}
      n=1 \\ n=2 \\ n=3 \\ n=4 \\ n=5 \\ n=6 \\ n=7
    \end{array}
    \left|
      \begin{array}{rrrrrrr}
         2 &    4 &      8 &       16 &         32 &           64 &            128 \\
         3 &   10 &     36 &      136 &        528 &         2080 &           8256 \\
         4 &   20 &    120 &      816 &       5984 &        45760 &         357760 \\
         6 &   55 &    666 &     9316 &     139656 &      2164240 &       34084896 \\
         8 &  136 &   3536 &   106912 &    3371840 &    107505280 &     3437022464 \\
        13 &  430 &  23052 &  1415896 &   89751728 &   5730905440 &   366571686592 \\
        18 & 1300 & 151848 & 19206736 & 2454791328 & 314154568000 & 40210845176448
      \end{array}
    \right.
    \label{tabl:nXmCyl_f__f}
  \]
\end{proposition}
This has been added to the OEIS as sequence A368258.
\begin{proposition}
  When $\orbid{f} = 1$, such as when
  \[
    T = \left\{\tileAsym{0}, \tileAsym{7}\right\}
    \qquad \text{or} \qquad
    T = \left\{\tileL, \tileR\right\}
  \]
  the number of tilings of the $n \times m$ cylinder up to horizontal
  reflection by a tile that is not stable under horizontal reflection is
  given by
  \[
    \begin{array}{r}
      n=1 \\ n=2 \\ n=3 \\ n=4 \\ n=5 \\ n=6 \\ n=7
    \end{array}
    \left|
      \begin{array}{rrrrrrr}
        1 & 2 & 4 & 8 & 16 & 32 & 64 \\
        2 & 6 & 20 & 72 & 272 & 1056 & 4160 \\
        2 & 12 & 88 & 688 & 5472 & 43712 & 349568 \\
        4 & 39 & 538 & 8292 & 131464 & 2098704 & 33560608 \\
        4 & 104 & 3280 & 104864 & 3355456 & 107374208 & 3435973888 \\
        9 & 366 & 22028 & 1399512 & 89489584 & 5726711136 & 366504577728 \\
        10 & 1172 & 149800 & 19173968 & 2454267040 & 314146179392 & 40210710958720
      \end{array}
    \right.
    \label{tabl:nXmCyl_f__id}
  \]
\end{proposition}
This has been added to the OEIS as sequence A368259.
\subsubsection{Under vertical reflection}
\begin{proposition}
  When $\orb{r^2f}{r^2f} = 2$, such as when
  \[
    T = \left\{\tileColor[black!90!white], \tileColor[black!10!white]\right\}
    \qquad\text{or}\qquad
    T = \left\{\tileL, \tileR\right\}
  \]
  the number of tilings of the $n \times m$ cylinder up to vertical
  reflection by two distinct tiles that are stable under vertical reflection is
  given by \[
    \begin{array}{r}
      n=1 \\ n=2 \\ n=3 \\ n=4 \\ n=5 \\ n=6 \\ n=7
    \end{array}
    \left|
      \begin{array}{rrrrrrr}
        2 & 3 & 6 & 10 & 20 & 36 & 72 \\
        3 & 7 & 24 & 76 & 288 & 1072 & 4224 \\
        4 & 14 & 100 & 700 & 5560 & 43800 & 350256 \\
        6 & 40 & 564 & 8296 & 131856 & 2098720 & 33566784 \\
        8 & 108 & 3384 & 104968 & 3358736 & 107377488 & 3436078752 \\
        14 & 362 & 22288 & 1399176 & 89505984 & 5726689312 & 366505626368 \\
        20 & 1182 & 150972 & 19175140 & 2454416840 & 314146329192 & 40210730132688
      \end{array}
    \right.
    \label{tabl:nXmCyl_rrf__rrf}
  \]
\end{proposition}
This has been added to the OEIS as sequence A368260.
\begin{proposition}
  When $\orbid{r^2f} = 1$, such as when
  \[
    T = \left\{\tileAsym{0}, \tileAsym{5}\right\}
    \qquad\text{or}\qquad
    T = \left\{\tileU, \tileD\right\}
  \]
  the number of tilings of the $n \times m$ cylinder up to vertical
  reflection by a tile that is not stable under vertical reflection is
  given by \[
    \begin{array}{r}
      n=1 \\ n=2 \\ n=3 \\ n=4 \\ n=5 \\ n=6 \\ n=7
    \end{array}
    \left|
      \begin{array}{rrrrrrr}
        1 & 3 & 4 & 10 & 16 & 36 & 64 \\
        2 & 7 & 20 & 76 & 272 & 1072 & 4160 \\
        2 & 14 & 88 & 700 & 5472 & 43800 & 349568 \\
        4 & 40 & 532 & 8296 & 131344 & 2098720 & 33558592 \\
        4 & 108 & 3280 & 104968 & 3355456 & 107377488 & 3435973888 \\
        8 & 362 & 21944 & 1399176 & 89484128 & 5726689312 & 366504228224 \\
        10 & 1182 & 149800 & 19175140 & 2454267040 & 314146329192 & 40210710958720
      \end{array}
    \right.
    \label{tabl:nXmCyl_rrf__id}
  \]
\end{proposition}
This has been added to the OEIS as sequence A368261.
\subsubsection{Under \texorpdfstring{$180^\circ$}{180 degree} rotation}
\begin{proposition}
  When $\orb{r^2}{r^2} = 2$, such as when
  \[
    T = \left\{\tileColor[black!90!white], \tileColor[black!10!white]\right\}
    \qquad\text{or}\qquad
    T = \left\{\tileRR{0}, \tileRR{1}\right\}
  \]
  the number of tilings of the $n \times m$ cylinder up to $180^\circ$ rotation
  by two distinct tiles that are stable under $180^\circ$ rotation is
  given by \[
    \begin{array}{r}
      n=1 \\ n=2 \\ n=3 \\ n=4 \\ n=5 \\ n=6 \\ n=7
    \end{array}
    \left|
      \begin{array}{rrrrrrr}
        2 & 3 & 6 & 10 & 20 & 36 & 72 \\
      3 & 7 & 24 & 76 & 288 & 1072 & 4224 \\
      4 & 16 & 104 & 720 & 5600 & 43968 & 350592 \\
      6 & 43 & 570 & 8356 & 131976 & 2099728 & 33568800 \\
      8 & 120 & 3408 & 105376 & 3359552 & 107390592 & 3436104960 \\
      13 & 382 & 22284 & 1400536 & 89505968 & 5726776672 & 366505626304 \\
      18 & 1236 & 150824 & 19182160 & 2454398112 & 314147227968 & 40210727735936
      \end{array}
    \right.
    \label{tabl:nXmCyl_rr__rr}
  \]
\end{proposition}
This has been added to the OEIS as sequence A368262.
\begin{proposition}
  When $\orbid{r^2} = 1$, such as when
  \[
    T = \left\{\tileAsym{0}, \tileAsym{2}\right\}
    \qquad\text{or}\qquad
    T = \left\{\tileU, \tileD\right\}
  \]
  the number of tilings of the $n \times m$ cylinder up to $180^\circ$ rotation
  by a tiles that is not stable under $180^\circ$ rotation is
  given by \[
    \begin{array}{r}
      n=1 \\ n=2 \\ n=3 \\ n=4 \\ n=5 \\ n=6 \\ n=7
    \end{array}
    \left|
      \begin{array}{rrrrrrr}
        1 & 3 & 4 & 10 & 16 & 36 & 64 \\
        2 & 7 & 20 & 76 & 272 & 1072 & 4160 \\
        2 & 16 & 88 & 720 & 5472 & 43968 & 349568 \\
        4 & 43 & 538 & 8356 & 131464 & 2099728 & 33560608 \\
        4 & 120 & 3280 & 105376 & 3355456 & 107390592 & 3435973888 \\
        9 & 382 & 22028 & 1400536 & 89489584 & 5726776672 & 366504577728 \\
        10 & 1236 & 149800 & 19182160 & 2454267040 & 314147227968 & 40210710958720
      \end{array}
    \right.
    \label{tabl:nXmCyl_rr__id}
  \]
\end{proposition}
This has been added to the OEIS as sequence A368263.
\subsubsection{Under cylindrical action only}
\begin{proposition}
  When $\mathcal{O}^{\mathbbm 1}_{\mathbbm 1} = 2$, such as when
  \[
    T = \left\{\tileColor[black!80!white], \tileColor[black!20!white]\right\}
  \]
  the number of tilings of the $n \times m$ cylinder
  by two distinct tiles is given by \[
    \begin{array}{r}
      n=1 \\ n=2 \\ n=3 \\ n=4 \\ n=5 \\ n=6 \\ n=7
    \end{array}
    \left|
      \begin{array}{rrrrrrr}
        2 & 4 & 8 & 16 & 32 & 64 & 128 \\
        3 & 10 & 36 & 136 & 528 & 2080 & 8256 \\
        4 & 24 & 176 & 1376 & 10944 & 87424 & 699136 \\
        6 & 70 & 1044 & 16456 & 262416 & 4195360 & 67113024 \\
        8 & 208 & 6560 & 209728 & 6710912 & 214748416 & 6871947776 \\
        14 & 700 & 43800 & 2796976 & 178962784 & 11453291200 & 733008106880 \\
        20 & 2344 & 299600 & 38347936 & 4908534080 & 628292358784 & 80421421917440
      \end{array}
    \right.
    \label{tabl:nXmCyl_id__id}
  \]
\end{proposition}
This has been added to the OEIS as sequence A368264.
                           %

\subsection{The \texorpdfstring{$n \times m$}{n by m} torus}    %
This section gives examples of every choice of symmetry of the $n \times m$
torus together with every essentially different set of tile designs that
consists of a single orbit (or two orbits, in the case of a fully symmetric tile).
Each sequence is annotated with its corresponding entry in the
On-Line Encyclopedia of Integer Sequences.
A table of all such sequences is given in Table \ref{tabl:nXmTorusIndex}.
\begin{table}[ht]
\begin{tabular}{l|l|l|l|l|}
                                    & $\langle r^2, f \rangle$                       & $\langle f \rangle$                       & $\langle r^2 \rangle$                      & $\mathbbm{1}$ \\
\hline
$\mathcal O_V$                      & \tableTableEntry{nXmTorus_rr_f__rr_f}{A222188} & \NA                                       & \NA                                        & \NA                                        \\[10pt]
$\mathcal O_{\langle f \rangle}$    & \tableTableEntry{nXmTorus_rr_f__f}   {A368302} & \tableTableEntry{nXmTorus_f__f} {A368305} & \NA                                        & \NA                                        \\[10pt]
$\mathcal O_{\langle r^2 \rangle}$  & \tableTableEntry{nXmTorus_rr_f__rr}  {A368303} & \NA                                       & \tableTableEntry{nXmTorus_rr__rr}{A368307} & \NA                                        \\[10pt]
$\mathcal O_\mathbbm{1}$            & \tableTableEntry{nXmTorus_rr_f__id}  {A368304} & \tableTableEntry{nXmTorus_f__id}{A368306} & \tableTableEntry{nXmTorus_rr__id}{A368308} & \tableTableEntry{nXmTorus_id__id}{A184271}
\end{tabular}
\caption{An index of tables that describe the number of tilings of the $n \times m$ torus.}
\label{tabl:nXmTorusIndex}
\end{table}

\subsubsection{Under horizontal and vertical reflection}
\begin{proposition}
  When $\orb{r^2,f}{r^2,f} = 2$, such as when
  \[
    T = \left\{\tileColor[black!90!white], \tileColor[black!10!white]\right\}
  \]
  the number of tilings of the $n \times m$ torus up to horizontal and vertical
  reflection
  by two distinct tiles with both horizontal and vertical reflectional symmetry
  is given by the following table:
  \[
  \begin{array}{r}
    n=1 \\ n=2 \\ n=3 \\ n=4 \\ n=5 \\ n=6
  \end{array}
  \left|
    \begin{array}{rrrrrrr}
      2 & 3 & 4 & 6 & 8 & 13 & 18 \\
      3 & 7 & 13 & 34 & 78 & 237 & 687 \\
      4 & 13 & 36 & 158 & 708 & 4236 & 26412 \\
      6 & 34 & 158 & 1459 & 14676 & 184854 & 2445918 \\
      8 & 78 & 708 & 14676 & 340880 & 8999762 & 245619576 \\
      13 & 237 & 4236 & 184854 & 8999762 & 478070832 & 26185264801 \\
      18 & 687 & 26412 & 2445918 & 245619576 & 26185264801 & 2872221202512
    \end{array}
  \right.
  \label{tabl:nXmTorus_rr_f__rr_f}
  \]
\end{proposition}
This is given by OEIS sequence A222188.
\begin{proposition}
  When $\orb{r^2,f}{f} = 1$, such as when
  \[
    T = \left\{\tileU, \tileD\right\}
  \]
  the number of tilings of the $n \times m$ torus up to horizontal and vertical
  reflection
  by a tile horizontal (but not vertical) reflectional symmetry
  is given by the following table:
  \[
  \begin{array}{r}
    n=1 \\ n=2 \\ n=3 \\ n=4 \\ n=5 \\ n=6 \\ n=7
  \end{array}
  \left|
    \begin{array}{rrrrrrr}
      1 & 2 & 2 & 4 & 4 & 9 & 10 \\
      2 & 5 & 9 & 26 & 62 & 205 & 623 \\
      2 & 8 & 22 & 120 & 600 & 3936 & 25556 \\
      4 & 22 & 126 & 1267 & 14164 & 181782 & 2437726 \\
      4 & 44 & 592 & 13600 & 337192 & 8965354 & 245501608 \\
      8 & 135 & 3936 & 178366 & 8980642 & 477655760 & 26184041441 \\
      9 & 362 & 25314 & 2404372 & 245479140 & 26179947021 & 2872203226920 \\
    \end{array}
  \right.
  \label{tabl:nXmTorus_rr_f__f}
  \]
\end{proposition}
This has been added to the OEIS as sequence A368302.
\begin{proposition}
  When $\orb{r^2,f}{r^2} = 1$, such as when
  \[
    T = \left\{\tileDiag, \tileAdiag\right\}
  \]
  the number of tilings of the $n \times m$ torus up to horizontal and vertical
  reflection
  by a tile with $180^\circ$ rotational symmetry
  is given by the following table:
  \[
  \begin{array}{r}
    n=1 \\ n=2 \\ n=3 \\ n=4 \\ n=5 \\ n=6 \\ n=7
  \end{array}
  \left|
    \begin{array}{rrrrrrr}
      1 & 2 & 2 & 4 & 4 & 8 & 9 \\
      2 & 5 & 8 & 22 & 44 & 135 & 362 \\
      2 & 8 & 24 & 120 & 612 & 3892 & 25482 \\
      4 & 22 & 120 & 1203 & 13600 & 177342 & 2404372 \\
      4 & 44 & 612 & 13600 & 337600 & 8962618 & 245492244 \\
      8 & 135 & 3892 & 177342 & 8962618 & 477371760 & 26179772237 \\
      9 & 362 & 25482 & 2404372 & 245492244 & 26179772237 & 2872202028544
    \end{array}
  \right.
  \label{tabl:nXmTorus_rr_f__rr}
  \]
\end{proposition}
This has been added to the OEIS as sequence A368303.
\begin{proposition}
  When $\orbid{r^2,f} = 1$, such as when
  \[
    T = \left\{\tileNW, \tileSW, \tileNE, \tileSE\right\}
  \]
  the number of tilings of the $n \times m$ torus up to horizontal and vertical
  reflection
  by a tile with no symmetry
  is given by the following table:
  \[
  \begin{array}{r}
    n=1 \\ n=2 \\ n=3 \\ n=4 \\ n=5 \\ n=6
  \end{array}
  \left|
    \begin{array}{rrrrrr}
      1 & 4 & 6 & 23 & 52 & 194 \\
      4 & 28 & 194 & 2196 & 26524 & 351588 \\
      6 & 194 & 7296 & 350573 & 17895736 & 954495904 \\
      23 & 2196 & 350573 & 67136624 & 13744131446 & 2932037300956 \\
      52 & 26524 & 17895736 & 13744131446 & 11258999068672 & 9607679419823148 \\
      194 & 351588 & 954495904 & 2932037300956 & 9607679419823148 & 32794211709502184448 \\
    \end{array}
  \right.
  \label{tabl:nXmTorus_rr_f__id}
  \]
\end{proposition}
This has been added to the OEIS as sequence A368304.
\subsubsection{Under horizontal (equivalently vertical) reflection}
\begin{proposition}
  When $\orb{f}{f} = 2$, such as when
  \[
    T = \left\{\tileColor[black!90!white], \tileColor[black!10!white]\right\}
    \qquad\text{or}\qquad
    T = \left\{\tileU, \tileD\right\},
  \]
  the number of tilings of the $n \times m$ torus up to horizontal reflection
  by two distinct tiles with horizontal reflectional symmetry
  is given by the following table:
  \[
  \begin{array}{r}
  n=1 \\ n=2 \\ n=3 \\ n=4 \\ n=5 \\ n=6 \\ n=7
  \end{array}
  \left|
    \begin{array}{rrrrrrr}
      2 & 3 & 4 & 6 & 8 & 14 & 20 \\
      3 & 7 & 14 & 40 & 108 & 362 & 1182 \\
      4 & 13 & 44 & 218 & 1200 & 7700 & 51112 \\
      6 & 34 & 226 & 2386 & 27936 & 361244 & 4869276 \\
      8 & 78 & 1184 & 26892 & 674384 & 17920876 & 491003216 \\
      13 & 237 & 7700 & 354680 & 17950356 & 955180432 & 52367383810 \\
      18 & 687 & 50628 & 4804062 & 490958280 & 52359294854 & 5744406453840
    \end{array}
  \right.
  \label{tabl:nXmTorus_f__f}
  \]
\end{proposition}
This has been added to the OEIS as sequence A368305.
\begin{proposition}
  When $\orbid{f} = 1$, such as when
  \[
  T = \left\{\tileL, \tileR\right\},
  \]
  the number of tilings of the $n \times m$ torus up to horizontal reflection
  by a tile that does not have horizontal reflectional symmetry
  is given by the following table:
  \[
  \begin{array}{r}
  n=1 \\ n=2 \\ n=3 \\ n=4 \\ n=5 \\ n=6 \\ n=7
  \end{array}
  \left|
    \begin{array}{rrrrrrr}
      1 & 2 & 2 & 4 & 4 & 8 & 10 \\
      2 & 5 & 8 & 24 & 56 & 190 & 596 \\
      2 & 9 & 32 & 186 & 1096 & 7356 & 49940 \\
      4 & 26 & 182 & 2130 & 26296 & 350316 & 4794376 \\
      4 & 62 & 1096 & 26380 & 671104 & 17899020 & 490853416 \\
      9 & 205 & 7356 & 350584 & 17897924 & 954481360 & 52357796826 \\
      10 & 623 & 49940 & 4795870 & 490853416 & 52357896710 & 5744387279872
    \end{array}
  \right.
  \label{tabl:nXmTorus_f__id}
  \]
\end{proposition}
This has been added to the OEIS as sequence A368306.
\subsubsection{Under \texorpdfstring{$180^\circ$}{180 degree} rotation}
\begin{proposition}
  When $\orb{r^2}{r^2} = 2$, such as when
  \[
    T = \left\{\tileColor[black!90!white], \tileColor[black!10!white]\right\}
    \qquad\text{or}\qquad
    T = \left\{\tileDiag, \tileAdiag\right\},
  \]
  the number of tilings of the $n \times m$ torus up to $180^\circ$ rotation
  by two distinct tiles with $180^\circ$ rotational symmetry
  is given by the following table:
  \[
  \begin{array}{r}
  n=1 \\ n=2 \\ n=3 \\ n=4 \\ n=5 \\ n=6 \\ n=7
  \end{array}
  \left|
    \begin{array}{rrrrrrr}
      2 & 3 & 4 & 6 & 8 & 13 & 18 \\
      3 & 7 & 13 & 34 & 78 & 237 & 687 \\
      4 & 13 & 48 & 224 & 1224 & 7696 & 50964 \\
      6 & 34 & 224 & 2302 & 27012 & 353384 & 4806078 \\
      8 & 78 & 1224 & 27012 & 675200 & 17920860 & 490984488 \\
      13 & 237 & 7696 & 353384 & 17920860 & 954677952 & 52359294790 \\
      18 & 687 & 50964 & 4806078 & 490984488 & 52359294790 & 5744404057088
    \end{array}
  \right.
  \label{tabl:nXmTorus_rr__rr}
  \]
\end{proposition}
This has been added to the OEIS as sequence A368307.
\begin{proposition}
  When $\orbid{r^2} = 1$, such as when
  \[
    T = \left\{\tileU, \tileD\right\}
    \qquad\text{or}\qquad
    T = \left\{\tileSW, \tileNE\right\},
  \]
  the number of tilings of the $n \times m$ torus up to $180^\circ$ rotation
  by a tile without $180^\circ$ rotational symmetry
  is given by the following table:
  \[
  \begin{array}{r}
  n=1 \\ n=2 \\ n=3 \\ n=4 \\ n=5 \\ n=6 \\ n=7
  \end{array}
  \left|
    \begin{array}{rrrrrrr}
      1 & 2 & 2 & 4 & 4 & 9 & 10 \\
      2 & 5 & 9 & 26 & 62 & 205 & 623 \\
      2 & 9 & 32 & 192 & 1096 & 7440 & 49940 \\
      4 & 26 & 192 & 2174 & 26500 & 351336 & 4797886 \\
      4 & 62 & 1096 & 26500 & 671104 & 17904476 & 490853416 \\
      9 & 205 & 7440 & 351336 & 17904476 & 954546880 & 52358246214 \\
      10 & 623 & 49940 & 4797886 & 490853416 & 52358246214 & 5744387279872
    \end{array}
  \right.
  \label{tabl:nXmTorus_rr__id}
  \]
\end{proposition}
This has been added to the OEIS as sequence A368308.
\subsubsection{Under toroidal action only}
\begin{proposition}
  When $\mathcal{O}^{\mathbbm 1}_{\mathbbm 1} = 2$, such as when
  \[
  T = \left\{\tileColor[black!90!white], \tileColor[white!90!black]\right\},
  \]
  the number of tilings of the $n \times m$ grid up to
  cyclic shifting of rows and columns
  by any two distinct tile designs is given by the following table:
  \[
  \begin{array}{r}
  n=1 \\ n=2 \\ n=3 \\ n=4 \\ n=5 \\ n=6 \\ n=7
  \end{array}
  \left|
    \begin{array}{rrrrrrr}
      2 & 3 & 4 & 6 & 8 & 14 & 20 \\
      3 & 7 & 14 & 40 & 108 & 362 & 1182 \\
      4 & 14 & 64 & 352 & 2192 & 14624 & 99880 \\
      6 & 40 & 352 & 4156 & 52488 & 699600 & 9587580 \\
      8 & 108 & 2192 & 52488 & 1342208 & 35792568 & 981706832 \\
      14 & 362 & 14624 & 699600 & 35792568 & 1908897152 & 104715443852 \\
      20 & 1182 & 99880 & 9587580 & 981706832 & 104715443852 & 11488774559744
    \end{array}
  \right.
  \label{tabl:nXmTorus_id__id}
  \]
\end{proposition}
This is OEIS sequence A184271.

\subsection{The \texorpdfstring{$n \times n$}{n by n} torus}
\label{apss:nXnTorusSequences}
This section gives examples of every choice of symmetry of the $n \times n$
torus together with every essentially different set of tile designs that
consists of a single orbit (or two orbits, in the case of a fully symmetric tile).
Each sequence is annotated with its corresponding entry in the
On-Line Encyclopedia of Integer Sequences.
A table of all such sequences is given in Table \ref{tabl:nXnTorusIndex}.
\begin{table}[ht]
\begin{tabular}{l|l|l|l|l|}
                                       & $\langle r, f \rangle$                 & $\langle r^2, rf \rangle$                & $\langle r \rangle$               & $\langle rf \rangle$ \\
\hline
$\mathcal O_{\langle r, f \rangle}$    & \seqTableEntry{nXnTorus_r_f__r_f}{A255016}   & \NA                                            & \NA                                     & \NA                                      \\[10pt]
$\mathcal O_{\langle r^2, f \rangle}$  & \seqTableEntry{nXnTorus_r_f__rr_f}{A367533}  & \NA                                            & \NA                                     & \NA                                      \\[10pt]
$\mathcal O_{\langle r^2, rf \rangle}$ & \seqTableEntry{nXnTorus_r_f__rr_rf}{A295223} & \seqTableEntry{nXnTorus_rr_rf__rr_rf}{A368139} & \NA                                     & \NA                                      \\[10pt]
$\mathcal O_{\langle r \rangle}$       & \seqTableEntry{nXnTorus_r_f__r}{A367534}     & \NA                                            & \seqTableEntry{nXnTorus_r__r}{A368143}  & \NA                                      \\[10pt]
$\mathcal O_{\langle f \rangle}$       & \seqTableEntry{nXnTorus_r_f__f}{A367535}     & \NA                                            & \NA                                     & \NA                                      \\[10pt]
$\mathcal O_{\langle rf \rangle}$      & \seqTableEntry{nXnTorus_r_f__rf}{A367536}    & \seqTableEntry{nXnTorus_rr_rf__rf}{A368140}    & \NA                                     & \seqTableEntry{nXnTorus_rf__rf}{A255015} \\[10pt]
$\mathcal O_{\langle r^2 \rangle}$     & \seqTableEntry{nXnTorus_r_f__rr}{A367537}    & \seqTableEntry{nXnTorus_rr_rf__rr}{A368141}    & \seqTableEntry{nXnTorus_r__rr}{A368144} & \NA                                      \\[10pt]
$\mathcal O_{\mathbbm 1}$              & \seqTableEntry{nXnTorus_r_f__id}{A367538}    & \seqTableEntry{nXnTorus_rr_rf__id}{A368142}    & \seqTableEntry{nXnTorus_r__id}{A368145} & \seqTableEntry{nXnTorus_rf__id}{A367530}
\end{tabular}
\caption{An index of tables that describe the number of tilings of the $n \times n$ torus.}
\label{tabl:nXnTorusIndex}
\end{table}

\subsubsection{Under the symmetries of the square}
\begin{proposition}
When $\orb{r, f}{r, f} = 2$, such as when
\[
  T = \left\{\tileColor[black!90!white], \tileColor[white!90!black]\right\},
\]
the number of tilings of the $n \times n$ torus up to symmetries of the square
by
two distinct tiles are fixed under all symmetries of the square is given by
\[
  2, 6, 26, 805, 172112, 239123150, 1436120190288, 36028817512382026, \dots
  \label{seq:nXnTorus_r_f__r_f}
\]
\end{proposition}
This is OEIS sequence A255016.
\begin{proposition}
When $\orb{r, f}{r^2, f} = 1$, such as when
\[
  T = \left\{ \tileVert, \tileHor \right\},
\]
the number of tilings of the $n \times n$ torus up to symmetries of the square
by a tile that is fixed under horizontal and vertical reflections
is given by
\[
  1, 4, 18, 733, 170440, 239035502, 1436110601256, 36028815364865610, \dots
  \label{seq:nXnTorus_r_f__rr_f}
\]
\end{proposition}
This has been added to the OEIS as sequence A367533.
\begin{proposition}
When $\orb{r, f}{r^2, rf} = 1$, such as when
\[
T = \left\{\tileDiag, \tileAdiag\right\},
\]
the number of tilings of the $n \times n$ torus up to symmetries of the square
by a tile that is fixed under diagonal and antidiagonal reflections
is given by
\[
  1, 4, 18, 669, 170440, 238773358, 1436110601256, 36028800332480074, \dots
  \label{seq:nXnTorus_r_f__rr_rf}
\]
\end{proposition}
This is OEIS sequence A295223.
\begin{proposition}
When $\orb{r, f}{r} = 1$, such as when
\[
T = \left\{\tileRot{0}, \tileRot{1}\right\},
\]
the number of tilings of the $n \times n$ torus up to symmetries of the square
by a tile that is fixed under $90^\circ$ rotations
is given by
\[
  1, 4, 14, 613, 168832, 238686222, 1436101016320, 36028798185029194, \dots
  \label{seq:nXnTorus_r_f__r}
\]
\end{proposition}
This has been added to the OEIS as sequence A367534.
\begin{proposition}
When $\orb{r, f}{f} = 1$, such as when
\[
  T = \left\{\tileU, \tileD, \tileL, \tileR \right\},
\]
the number of tilings of the $n \times n$ torus up to symmetries of the square
by a tile that is fixed under horizontal (respectively vertical) reflections
is given by
\[
  1, 16, 3692, 33570410, 5629501212064, 16397105856182791856, \dots
  \label{seq:nXnTorus_r_f__f}
\]
\end{proposition}
This has been added to the OEIS as sequence A367535.
\begin{proposition}
When $\orb{r, f}{rf} = 1$, such as when
\[
  T = \left\{\tileNW, \tileSW, \tileNE, \tileSE\right\},
\]
the number of tilings of the $n \times n$ torus up to symmetries of the square
by a tile that is fixed under antidiagonal (respectively diagonal) reflections
is given by
\[
  1, 17, 3692, 33572458, 5629501212064, 16397105857614447792, \dots
  \label{seq:nXnTorus_r_f__rf}
\]
\end{proposition}
This has been added to the OEIS as sequence A367536.
\begin{proposition}
When $\orb{r, f}{r^2} = 1$, such as when
\[
  T = \left\{\tileRR{0},\tileRR{1},\tileRR{2},\tileRR{3}\right\},
\]
the number of tilings of the $n \times n$ torus up to symmetries of the square
by a tile that is fixed under $180^\circ$ rotations
is given by
\[
  1, 23, 3776, 33601130, 5629507922944, 16397105889110874288, \dots
  \label{seq:nXnTorus_r_f__rr}
\]
\end{proposition}
This has been added to the OEIS as sequence A368137.
\begin{proposition}
When $\orbid{r, f} = 1$, such as when
\[
  T = \left\{
    \tileAsym{0},\tileAsym{1},\tileAsym{2},\tileAsym{3},
    \tileAsym{4},\tileAsym{5},\tileAsym{6},\tileAsym{7}
  \right\},
\]
the number of tilings of the $n \times n$ torus up to symmetries of the square
by a tile that is fixed under only the identity
is given by
\[
  1, 154, 1864192, 2199026796168, 188894659314785812480, \dots
  \label{seq:nXnTorus_r_f__id}
\]
\end{proposition}
This has been added to the OEIS as sequence A368138.
\subsubsection{Under diagonal and antidiagonal reflection}
\begin{proposition}
When $\orb{r^2, rf}{r^2, rf} = 2$, such as when
\[
  T = \left\{\tileColor[black!90!white], \tileColor[white!90!black]\right\}
  \qquad\text{or}\qquad
  T = \left\{\tileDiag, \tileAdiag \right\},
\]
the number of tilings of the $n \times n$ torus up to diagonal and
antidiagonal rotations by
two distinct tiles that are symmetric under both reflections is given by
\[
  2, 6, 36, 1282, 340880, 477513804, 2872221202512, 72057600262282324, \dots
  \label{seq:nXnTorus_rr_rf__rr_rf}
\]
\end{proposition}
This has been added to the OEIS as sequence A368139.
\begin{proposition}
When $\orb{r^2, rf}{rf} = 1$, such as when
\[
  T = \left\{\tileSW, \tileNE\right\}
\]
the number of tilings of the $n \times n$ torus up to diagonal and
antidiagonal rotations by
a tile that is symmetric only under antidiagonal reflections is given by
\[
  1, 4, 22, 1154, 337192, 477360876, 2872203226920, 72057597041056852, \dots
  \label{seq:nXnTorus_rr_rf__rf}
\]
\end{proposition}
This has been added to the OEIS as sequence A368140.
\begin{proposition}
When $\orb{r^2, rf}{r^2} = 1$, such as when
\[
  T = \left\{\tileRR{0}, \tileRR{2}\right\}
\]
the number of tilings of the $n \times n$ torus up to diagonal and
antidiagonal rotations by
a tile that is symmetric only under $180^\circ$ rotations is given by
\[
  1, 4, 24, 1154, 337600, 477339020, 2872202028544, 72057595967315028, \dots
  \label{seq:nXnTorus_rr_rf__rr}
\]
\end{proposition}
This has been added to the OEIS as sequence A368141.
\begin{proposition}
When $\orbid{r^2, rf} = 1$, such as when
\[
  T = \left\{\tileR, \tileU, \tileL, \tileD\right\}
\]
the number of tilings of the $n \times n$ torus up to diagonal and
antidiagonal rotations by
a tile that is asymmetric is given by
\[
  1, 23, 7296, 67124308, 11258999068672, 32794211700912270688, \dots
  \label{seq:nXnTorus_rr_rf__id}
\]
\end{proposition}
This has been added to the OEIS as sequence A368142.
\subsubsection{Under \texorpdfstring{$90^\circ$}{90 degree} rotation}
\begin{proposition}
When $\orb{r}{r} = 2$, such as when
\[
  T = \left\{\tileColor[black!90!white], \tileColor[white!90!black]\right\}
  \qquad\text{or}\qquad
  T = \left\{\tileRot{0}, \tileRot{1} \right\},
\]
the number of tilings of the $n \times n$ torus up to $90^\circ$ rotations by
two distinct tiles that are symmetric under $90^\circ$ rotations is given by
\[
  2, 6, 28, 1171, 337664, 477339616, 2872202032640, 72057595967392816, \dots
  \label{seq:nXnTorus_r__r}
\]
\end{proposition}
This has been added to the OEIS as sequence A368143.
\begin{proposition}
When $\orb{r}{r^2} = 1$, such as when
\[
  T = \left\{\tileRR{0}, \tileRR{1} \right\},
\]
the number of tilings of the $n \times n$ torus up to $90^\circ$ rotations by
a tile that is symmetric under $180^\circ$ rotations is given by
\[
  1, 4, 24, 1155, 337600, 477339104, 2872202028544, 72057595967327280, \dots
  \label{seq:nXnTorus_r__rr}
\]
\end{proposition}
This has been added to the OEIS as sequence A368144.
\begin{proposition}
When $\orbid{r} = 1$, such as when
\[
  T = \left\{\tileNW, \tileSW, \tileNE, \tileSE \right\},
\]
the number of tilings of the $n \times n$ torus up to $90^\circ$ rotations by
a tile asymmetric with respect to rotations is given by
\[
  1, 23, 7296, 67124336, 11258999068672, 32794211700912314368, \dots
  \label{seq:nXnTorus_r__id}
\]
\end{proposition}
This has been added to the OEIS as sequence A368145.
\subsubsection{Under diagonal (equivalently antidiagonal) reflection}
\begin{proposition}
When $\orb{rf}{rf} = 2$, such as when
\[
  T = \left\{\tileColor[black!90!white], \tileColor[white!90!black]\right\},
\]
the number of tilings of the $n \times n$ torus up to transposition by
two distinct tiles that are fixed under  is given by
\[
  2, 6, 44, 2209, 674384, 954623404, 5744406453840, 144115192471496836, \dots
  \label{seq:nXnTorus_rf__rf}
\]
\end{proposition}
This is OEIS sequence A255015.
\begin{proposition}
When $\orbid{rf} = 1$, such as when
\[
  T = \left\{\tileU, \tileL\right\},
  \qquad\text{or}\qquad
  T = \left\{\tileNW, \tileSE\right\},
\]
the number of tilings of the $n \times n$ torus up to transposition by
tiles that are asymmetric with respect to this transposition is given by
\[
  1, 4, 32, 2081, 671104, 954448620, 5744387279872, 144115188176529540, \dots
  \label{seq:nXnTorus_rf__id}
\]
\end{proposition}
This has been added to the OEIS as sequence A367530.
                         %

\appendixpagenumbering
\section{Illustrations}
This section of the appendix gives illustrations corresponding to all of the
sequences and tables described in Appendix \ref{apn:sequencesAndTables}, which
shows an example of the tilings arising from
all valid choices of $R \leq D_8$ and all sets of tile designs
consisting of a single orbit.
\subsection{The \texorpdfstring{$n \times m$}{n by m} grid}     %
\subsubsection{Under horizontal and vertical reflection}
\nXmGridIllustration{rr_f__rr_f}{2x3}{
  The $24$ ways of tiling the $2\times 3$ grid up to
  $D_4 = \langle r^2, f\rangle$
  from a set of tile designs that consists of two orbits
  both of which contain an element with stabilizer subgroup
  $D_4$.
}
\nXmGridIllustration{rr_f__f}{3x2}{
  The $24$ ways of tiling the $3\times 2$ grid up to
  $D_4 = \langle r^2, f\rangle$
  from a set of tile designs that consists of one orbit
  containing an element whose stabilizer subgroup is
  $\langle f \rangle \leq D_4$.
}
\nXmGridIllustration{rr_f__rr}{3x2}{
  The $20$ ways of tiling the $3\times 2$ grid up to
  $D_4 = \langle r^2, f\rangle$
  from a set of tile designs that consists of one orbit
  containing an element whose stabilizer subgroup is
  $\langle r^2 \rangle \leq D_4$.
}
\nXmGridIllustration{rr_f__id}{2x2}{
  The $76$ ways of tiling the $2\times 2$ grid up to
  $D_4 = \langle r^2, f\rangle$
  from a set of tile designs that consists of one orbit
  containing an element whose stabilizer subgroup is
  $\mathbbm 1 \leq D_4$.
}
\subsubsection{Under horizontal (equivalently vertical) reflection}
\nXmGridIllustration{f__f}{3x2}{
  The $40$ ways of tiling the $3\times 2$ grid up to
  $\langle f\rangle$
  from a set of tile designs that consists of two orbits
  both of which contain an element with stabilizer subgroup
  $\langle f \rangle$.
}
\nXmGridIllustration{f__id}{3x2}{
  The $32$ ways of tiling the $3\times 2$ grid up to
  $\langle f \rangle$
  from a set of tile designs that consists of one orbit
  containing an element whose stabilizer subgroup is
  $\mathbbm 1 \leq \langle f \rangle$.
}

\subsubsection{Under \texorpdfstring{$180^\circ$}{180 degree} rotation}
\nXmGridIllustration{rr__rr}{3x2}{
  The $36$ $3 \times 2$ grids up to
  $\langle r^2 \rangle$
  from a set of tile designs that consists of two orbits
  both of which contain an element with stabilizer subgroup
  $\langle r^2 \rangle$.
}

\nXmGridIllustration{rr__id}{2x2}{
  The $10$ $2 \times 2$ grids up to
  $\langle r^2 \rangle$
  from a set of tile designs that consists of one orbit
  containing an element whose stabilizer subgroup is
  $\mathbbm 1 \leq \langle r^2 \rangle$.
}

\subsection{The \texorpdfstring{$n \times n$}{n by n} grid}     %
\subsubsection{Under symmetries of the square}
\nXnGridIllustration{r_f__r_f}{2x2}{
  The $6$ distinct ways of tiling the $2 \times 2$ grid up to
  $D_8 = \langle r, f \rangle$
  from a set of tile designs that consists of two orbits
  both of which contain an element with stabilizer subgroup
  $D_8$.
}

\nXnGridIllustration{r_f__rr_f}{2x2}{
  The $4$ distinct ways of tiling the $2 \times 2$ grid up to
  $D_8 = \langle r, f \rangle$
  from a set of tile designs that consists of one orbit
  containing an element whose stabilizer subgroup is
  $\langle r^2, f \rangle \leq D_8$.
}

\nXnGridIllustration{r_f__rr_rf}{2x2}{
  The $6$ distinct ways of tiling the $2 \times 2$ grid up to
  $D_8 = \langle r, f \rangle$
  from a set of tile designs that consists of one orbit
  containing an element whose stabilizer subgroup is
  $\langle r^2, rf \rangle \leq D_8$.
}

\nXnGridIllustration{r_f__r}{3x3}{
  The $70$ distinct ways of tiling the $3 \times 3$ grid up to
  $D_8 = \langle r, f \rangle$
  from a set of tile designs that consists of one orbit
  containing an element whose stabilizer subgroup is
  $\langle r \rangle \leq D_8$.
}

\nXnGridIllustration{r_f__f}{2x2}{
  The $39$ distinct ways of tiling the $2 \times 2$ grid up to
  $D_8 = \langle r, f \rangle$
  from a set of tile designs that consists of one orbit
  containing an element whose stabilizer subgroup is
  $\langle f \rangle \leq D_8$.
}

\nXnGridIllustration{r_f__rf}{2x2}{
  The $43$ distinct ways of tiling the $2 \times 2$ grid up to
  $D_8 = \langle r, f \rangle$
  from a set of tile designs that consists of one orbit
  containing an element whose stabilizer subgroup is
  $\langle rf \rangle \leq D_8$.
}

\nXnGridIllustration{r_f__rr}{2x2}{
  The $39$ distinct ways of tiling the $2 \times 2$ grid up to
  $D_8 = \langle r, f \rangle$
  from a set of tile designs that consists of one orbit
  containing an element whose stabilizer subgroup is
  $\langle r^2 \rangle \leq D_8$.
}

\nXnGridIllustration{r_f__id}{2x2}{
  $50$ of the $538$ distinct ways of tiling the $2 \times 2$ grid up to
  $D_8 = \langle r, f \rangle$
  from a set of tile designs that consists of one orbit
  containing an element whose stabilizer subgroup is
  $\mathbbm 1 \leq D_8$.
}

\subsubsection{Under diagonal and antidiagonal reflection}

\nXnGridIllustration{rr_rf__rr_rf}{3x3}{
  The $168$ tilings of the $3 \times 3$ grid up to
  $\langle r^2, rf \rangle$
  from a set of tile designs that consists of two orbits
  both of which contain an element with stabilizer subgroup
  $\langle r^2, rf \rangle$.
}

\nXnGridIllustration{rr_rf__rf}{3x3}{
  The $144$ tilings of the $3 \times 3$ grid up to
  $\langle r^2, rf \rangle$
  from a set of tile designs that consists of one orbit
  containing an element whose stabilizer subgroup is
  $\langle rf \rangle \leq \langle r^2, rf \rangle$.
}

\nXnGridIllustration{rr_rf__rr}{2x2}{
  The $5$ tilings of the $2 \times 2$ grid up to
  $\langle r^2, rf \rangle$
  from a set of tile designs that consists of one orbit
  containing an element whose stabilizer subgroup is
  $\langle r^2 \rangle \leq \langle r^2, rf \rangle$.
}

\nXnGridIllustration{rr_rf__id}{2x2}{
  The $68$ tilings of the $2 \times 2$ grid up to
  $\langle r^2, rf \rangle$
  from a set of tile designs that consists of one orbit
  containing an element whose stabilizer subgroup is
  $\mathbbm 1 \leq \langle r^2, rf \rangle$.
}

\subsubsection{Under \texorpdfstring{$90^\circ$}{90 degree} rotation}
\nXnGridIllustration{r__r}{3x3}{
  The $140$ tilings of the $3 \times 3$ grid up to
  $\langle r \rangle$
  from a set of tile designs that consists of two orbits
  both of which contain an element with stabilizer subgroup
  $\langle r \rangle$.
}
\nXnGridIllustration{r__rr}{3x3}{
  The $136$ tilings of the $3 \times 3$ grid up to
  $\langle r \rangle$
  from a set of tile designs that consists of one orbit
  containing an element whose stabilizer subgroup is
  $\langle r^2 \rangle \leq \langle r \rangle$.
}
\nXnGridIllustration{r__id}{2x2}{
  The $70$ tilings of the $2 \times 2$ grid up to
  $\langle r \rangle$
  from a set of tile designs that consists of one orbit
  containing an element whose stabilizer subgroup is
  $\mathbbm 1 \leq \langle r \rangle$.
}
\subsubsection{Under diagonal (equivalently antidiagonal) reflection}
\nXnGridIllustration{rf__rf}{2x2}{
  The $12$ tilings of the $2 \times 2$ grid up to
  $\langle rf \rangle$
  from a set of tile designs that consists of two orbits
  both of which contain an element with stabilizer subgroup
  $\langle rf \rangle$.
}
\nXnGridIllustration{rf__id}{2x2}{
  The $8$ tilings of the $2 \times 2$ grid up to
  $\langle rf \rangle$
  from a set of tile designs that consists of one orbit
  containing an element whose stabilizer subgroup is
  $\mathbbm 1 \leq \langle rf \rangle$.
}

\subsection{The \texorpdfstring{$n \times m$}{n by m} cylinder} %
\subsubsection{Under horizontal and vertical reflection}
\nXmCylinderIllustration{rr_f__rr_f}{2x3}{
  The $24$ distinct ways of tiling the $2 \times 3$ cylinder up to
  $D_4 = \langle r^2, f \rangle$
  from a set of tile designs that consists of two orbits
  both of which contain an element with stabilizer subgroup
  $D_4$.
}
\nXmCylinderIllustration{rr_f__f}{2x3}{
  The $20$ distinct ways of tiling the $2 \times 3$ cylinder up to
  $D_4 = \langle r^2, f \rangle$
  from a set of tile designs that consists of one orbit
  containing an element whose stabilizer subgroup is
  $\langle f \rangle \leq D_4$.
}
\nXmCylinderIllustration{rr_f__rrf}{4x2}{
  The $26$ distinct ways of tiling the $4 \times 2$ cylinder up to
  $D_4 = \langle r^2, f \rangle$
  from a set of tile designs that consists of one orbit
  containing an element whose stabilizer subgroup is
  $\langle r^2f \rangle \leq D_4$.
}
\nXmCylinderIllustration{rr_f__rr}{3x2}{
  The $9$ distinct ways of tiling the $3 \times 2$ cylinder up to
  $D_4 = \langle r^2, f \rangle$
  from a set of tile designs that consists of one orbit
  containing an element whose stabilizer subgroup is
  $\langle r^2 \rangle \leq D_4$.
}
\nXmCylinderIllustration{rr_f__id}{2x2}{
  The $20$ distinct ways of tiling the $2 \times 2$ cylinder up to
  $D_4 = \langle r^2, f \rangle$
  from a set of tile designs that consists of one orbit
  containing an element whose stabilizer subgroup is
  $\mathbbm 1 \leq D_4$.
}

\subsubsection{Under horizontal reflection}
\nXmCylinderIllustration{f__f}{3x2}{
  The $20$ distinct ways of tiling the $3 \times 2$ cylinder up to
  $\langle f \rangle$
  from a set of tile designs that consists of two orbits
  both of which contain an element with stabilizer subgroup
  $\langle f \rangle$.
}
\nXmCylinderIllustration{f__id}{2x3}{
  The $20$ distinct ways of tiling the $2 \times 3$ cylinder up to
  $\langle f \rangle$
  from a set of tile designs that consists of one orbit
  containing an element whose stabilizer subgroup is
  $\mathbbm 1 \leq D_4$.
}
\subsubsection{Under vertical reflection}
\nXmCylinderIllustration{rrf__rrf}{2x3}{
  The $24$ distinct ways of tiling the $2 \times 3$ cylinder up to
  $\langle r^2f \rangle$
  from a set of tile designs that consists of two orbits
  both of which contain an element with stabilizer subgroup
  $\langle r^2f \rangle$.
}
\nXmCylinderIllustration{rrf__id}{2x3}{
  The $20$ distinct ways of tiling the $2 \times 3$ cylinder up to
  $\langle r^2f \rangle$
  from a set of tile designs that consists of one orbit
  containing an element whose stabilizer subgroup is
  $\mathbbm 1 \leq D_4$.
}
\subsubsection{Under \texorpdfstring{$180^\circ$}{180 degree} rotation}
\nXmCylinderIllustration{rr__rr}{3x2}{
  The $16$ distinct ways of tiling the $3 \times 2$ cylinder up to
  $\langle r^2 \rangle$
  from a set of tile designs that consists of two orbits
  both of which contain an element with stabilizer subgroup
  $\langle r^2 \rangle$.
}
\nXmCylinderIllustration{rr__id}{2x3}{
  The $20$ distinct ways of tiling the $2 \times 3$ cylinder up to
  $\langle r^2 \rangle$
  from a set of tile designs that consists of one orbit
  containing an element whose stabilizer subgroup is
  $\mathbbm 1 \leq \langle r^2 \rangle$.
}
\subsubsection{Under cylindrical action only}
\nXmCylinderIllustration{id__id}{2x2}{
  The $10$ distinct ways of tiling the $2 \times 2$ cylinder
  from a set of tile designs that consists of two orbits,
  each containing a single tile design.
}

\subsection{The \texorpdfstring{$n \times m$}{n by m} torus}    %
\subsubsection{Under horizontal and vertical reflection}
\nXmTorusIllustration{rr_f__rr_f}{3x2}{
  The $13$ distinct ways of tiling the $3 \times 2$ torus up to
  $\langle r^2, f \rangle$
  from a set of tile designs that consists of two orbits
  both of which contain an element with stabilizer subgroup
  $D_4$.
}
\nXmTorusIllustration{rr_f__f}{3x2}{
  The $8$ distinct ways of tiling the $3 \times 2$ torus up to
  $\langle r^2, f\rangle$
  from a set of tile designs that consists of one orbit
  containing an element whose stabilizer subgroup is
  $\langle f \rangle \leq D_4$.
}
\nXmTorusIllustration{rr_f__rr}{3x2}{
  The $8$ distinct ways of tiling the $3 \times 2$ torus up to
  $\langle r^2, f\rangle$
  from a set of tile designs that consists of one orbit
  containing an element whose stabilizer subgroup is
  $\langle r^2 \rangle \leq D_4$.
}
\nXmTorusIllustration{rr_f__id}{2x2}{
  The $28$ distinct ways of tiling the $2 \times 2$ torus up to
  $\langle r^2, f\rangle$
  from a set of tile designs that consists of one orbit
  containing an element whose stabilizer subgroup is
  $\mathbbm 1 \leq D_4$.
}
\subsubsection{Under horizontal (equivalently vertical) reflection}
\nXmTorusIllustration{f__f}{2x2}{
  The $7$ distinct ways of tiling the $2 \times 2$ torus up to
  $\langle f \rangle$
  from a set of tile designs that consists of two orbits
  both of which contain an element with stabilizer subgroup
  $\langle f \rangle$.
}
\nXmTorusIllustration{f__id}{3x2}{
  The $9$ distinct ways of tiling the $3 \times 2$ torus up to
  $\langle f \rangle$
  from a set of tile designs that consists of one orbit
  containing an element whose stabilizer subgroup is
  $\mathbbm 1 \leq \langle f \rangle$.
}
\subsubsection{Under \texorpdfstring{$180^\circ$}{180 degree} rotation}
\nXmTorusIllustration{rr__rr}{3x2}{
  The $13$ distinct ways of tiling the $3 \times 2$ torus up to
  $\langle f \rangle$
  from a set of tile designs that consists of two orbits
  both of which contain an element with stabilizer subgroup
  $\langle f \rangle$.
}
\nXmTorusIllustration{rr__id}{3x2}{
  The $9$ distinct ways of tiling the $3 \times 2$ torus
  $\langle r^2 \rangle$
  from a set of tile designs that consists of one orbit
  containing an element whose stabilizer subgroup is
  $\mathbbm 1 \leq \langle r^2 \rangle$.
}
\subsubsection{Under toroidal action only}
\nXmTorusIllustration{id__id}{3x2}{
  The $15$ distinct ways of tiling the $3 \times 2$ torus
  from a set of tile designs that consists of two orbits,
  each containing a single tile design.
}

\subsection{The \texorpdfstring{$n \times n$}{n by n} torus}
\label{apss:nXnTorusIllustrations}
\subsubsection{Under the symmetries of the square}
\nXnTorusIllustration{r_f__r_f}{3x3}{
  The $26$ ways of tiling the $3 \times 3$ torus up to
  $D_8 = \langle r, f \rangle$
  from a set of tile designs that consists of two orbits
  both of which contain an element with stabilizer subgroup
  $D_8$.
}
\nXnTorusIllustration{r_f__rr_f}{3x3}{
  The $18$ ways of tiling the $3 \times 3$ torus up to
  $D_8 = \langle r, f \rangle$
  from a set of tile designs that consists of one orbit
  containing an element whose stabilizer subgroup is
  $\langle r^2, f \rangle \leq D_8$.
}
\nXnTorusIllustration{r_f__rr_rf}{3x3}{
  The $18$ ways of tiling the $3 \times 3$ torus up to
  $D_8 = \langle r, f \rangle$
  from a set of tile designs that consists of one orbit
  containing an element whose stabilizer subgroup is
  $\langle r^2, rf \rangle \leq D_8$.
}
\nXnTorusIllustration{r_f__r}{2x2}{
  The $4$ ways of tiling the $2 \times 2$ torus up to
  $D_8 = \langle r, f \rangle$
  from a set of tile designs that consists of one orbit
  containing an element whose stabilizer subgroup is
  $\langle r \rangle \leq D_8$.
}
\nXnTorusIllustration{r_f__f}{2x2}{
  The $16$ ways of tiling the $2 \times 2$ torus up to
  $D_8 = \langle r, f \rangle$
  from a set of tile designs that consists of one orbit
  containing an element whose stabilizer subgroup is
  $\langle f \rangle \leq D_8$.
}
\nXnTorusIllustration{r_f__rf}{2x2}{
  The $17$ ways of tiling the $2 \times 2$ torus up to
  $D_8 = \langle r, f \rangle$
  from a set of tile designs that consists of one orbit
  containing an element whose stabilizer subgroup is
  $\langle rf \rangle \leq D_8$.
}
\nXnTorusIllustration{r_f__rr}{2x2}{
  The $23$ ways of tiling the $2 \times 2$ torus up to
  $D_8 = \langle r, f \rangle$
  from a set of tile designs that consists of one orbit
  containing an element whose stabilizer subgroup is
  $\langle r^2 \rangle \leq D_8$.
}
\nXnTorusIllustration{r_f__id}{2x2}{
  The $154$ ways of tiling the $2 \times 2$ torus up to
  $D_8 = \langle r, f \rangle$
  from a set of tile designs that consists of one orbit
  containing an element whose stabilizer subgroup is
  $\mathbbm 1 \leq D_8$.
}

\subsubsection{Under diagonal and antidiagonal reflection}
\nXnTorusIllustration{rr_rf__rr_rf}{3x3}{
  The $36$ ways of tiling the $3 \times 3$ torus up to
  $\langle r^2, rf \rangle$
  from a set of tile designs that consists of two orbits
  both of which contain an element with stabilizer subgroup
  $\langle r^2, rf \rangle$.
}
\nXnTorusIllustration{rr_rf__rf}{3x3}{
  The $22$ ways of tiling the $3 \times 3$ torus up to
  $\langle r^2, rf \rangle$
  from a set of tile designs that consists of one orbit
  containing an element whose stabilizer subgroup is
  $\langle rf \rangle \leq \langle r^2, rf \rangle$.
}
\nXnTorusIllustration{rr_rf__rr}{3x3}{
  The $24$ ways of tiling the $3 \times 3$ torus up to
  $\langle r^2, rf \rangle$
  from a set of tile designs that consists of one orbit
  containing an element whose stabilizer subgroup is
  $\langle r^2 \rangle \leq \langle r^2, rf \rangle$.
}
\nXnTorusIllustration{rr_rf__id}{2x2}{
  The $23$ ways of tiling the $2 \times 2$ torus up to
  $\langle r^2, rf \rangle$
  from a set of tile designs that consists of one orbit
  containing an element whose stabilizer subgroup is
  $\mathbbm 1 \leq \langle r^2, rf \rangle$.
}
\subsubsection{Under \texorpdfstring{$90^\circ$}{90 degree} rotation}
\nXnTorusIllustration{r__r}{3x3}{
  The $28$ ways of tiling the $3 \times 3$ torus up to
  $\langle r \rangle$
  from a set of tile designs that consists of two orbits
  both of which contain an element with stabilizer subgroup
  $\langle r \rangle$.
}
\nXnTorusIllustration{r__rr}{3x3}{
  The $24$ ways of tiling the $3 \times 3$ torus up to
  $\langle r \rangle$
  from a set of tile designs that consists of one orbit
  containing an element whose stabilizer subgroup is
  $\langle r^2 \rangle \leq \langle r \rangle$.
}
\nXnTorusIllustration{r__id}{2x2}{
  The $23$ ways of tiling the $2 \times 2$ torus up to
  $\langle r \rangle$
  from a set of tile designs that consists of one orbit
  containing an element whose stabilizer subgroup is
  $\mathbbm 1 \leq \langle r \rangle$.
}
\subsubsection{Under diagonal (equivalently antidiagonal) reflection}
\nXnTorusIllustration{rf__rf}{3x3}{
  The $44$ ways of tiling the $3 \times 3$ torus up to
  $\langle rf \rangle$
  from a set of tile designs that consists of two orbits
  both of which contain an element with stabilizer subgroup
  $\langle rf \rangle$.
}
\nXnTorusIllustration{rf__id}{3x3}{
  The $32$ ways of tiling the $3 \times 3$ torus up to
  $\langle rf \rangle$
  from a set of tile designs that consists of one orbit
  containing an element whose stabilizer subgroup is
  $\mathbbm 1 \leq \langle rf \rangle$.
}

\end{document}